\documentclass[reqno]{amsart}
\usepackage{amssymb,amsmath,hyperref}
\usepackage{amsrefs, float}
\usepackage[foot]{amsaddr}
\usepackage{bbold,stackrel, multicol}
 
\newcommand{\Z}{{\textsf{\textup{Z}}}}

\newtheorem{thm}{Theorem}

\newtheorem{cor}[thm]{Corollary}
\newtheorem{defi}[thm]{Definition}

\newtheorem{rem}[thm]{Remark}
\newtheorem{nota}[thm]{Notation}

\newtheorem{princ}[thm]{Principle}

\newtheorem{ack}[thm]{Acknowledgement}

\newtheorem*{tempo*}{Template}
\newtheorem{theorem}[thm]{Theorem}

\newcommand\be{\begin{equation}}
\newcommand\ee{\end{equation}} 

\usepackage{amsmath,amsfonts} 

\usepackage[applemac]{inputenc}

\def\bdefi{\begin{defi}\rm}
\def\edefi{\end{defi}}
\def\bnota{\begin{nota}\rm}
\def\enota{\end{nota}}
\def\FIVE{\Pi_{1}^{1}\text{-\textup{\textsf{CA}}}_{0}}
\def\SIX{\Pi_{2}^{1}\text{-\textsf{\textup{CA}}}_{0}}

\def\SIXK{\Pi_{k}^{1}\text{-\textsf{\textup{CA}}}_{0}^{\omega}}

\def\ATR{\textup{\textsf{ATR}}}

\def\ZF{\textup{\textsf{ZF}}}

\def\L{\textsf{\textup{L}}}

 \def\r{\mathbb{r}}

\def\RCA{\textup{\textsf{RCA}}}
\def\({\textup{(}}
\def\){\textup{)}}

\def\RCAo{\textup{\textsf{RCA}}_{0}^{\omega}}
\def\ACAo{\textup{\textsf{ACA}}_{0}^{\omega}}

\def\WKL{\textup{\textsf{WKL}}}

\def\WWKL{\textup{\textsf{WWKL}}}
\def\bye{\end{document}}
\def\N{{\mathbb  N}}
\def\Q{{\mathbb  Q}}
\def\R{{\mathbb  R}}

\def\SS{\textup{\textsf{S}}}

\def\di{\rightarrow}

\def\asa{\leftrightarrow}

\def\ACA{\textup{\textsf{ACA}}}

\def\QFAC{\textup{\textsf{QF-AC}}}
\def\SAC{\Sigma_{1}^{1}\textup{\textsf{-AC}}_{0}}
\def\AC{\textup{\textsf{AC}}}

\def\Gap{\textup{\textsf{Jump}}}

\def\INDY{\textup{\textsf{IND}}_{0}}

\def\SAC{\textup{\textsf{$\Sigma_{1}^{1}$-AC$_{0}$}}}

\def\NIN{\textup{\textsf{NIN}}}
\def\NCC{\textup{\textsf{NCC}}}
\def\NBI{\textup{\textsf{NBI}}}

\def\IND{\textup{\textsf{IND}}}

\def\eps{\varepsilon}

\def\ECF{\textup{\textsf{ECF}}}

\newcommand{\rinf}{\rightarrow \infty}

\usepackage{graphicx}
\usepackage{tikz}
\usetikzlibrary{matrix, shapes.misc}
\usepackage{comment,tikz-cd}

\numberwithin{equation}{section}
\numberwithin{thm}{section}

\begin{document}
\title{The Biggest Five of Reverse Mathematics}
%\title{Ever the twain shall meet: the Bigger Five of Reverse Mathematics}
%\title[]{On the connection between second- and higher-order Reverse Mathematics}
\author{Dag Normann}
\address{Department of Mathematics, The University 
of Oslo, P.O. Box 1053, Blindern N-0316 Oslo, Norway}
\email{dnormann@math.uio.no}
\author{Sam Sanders}
\address{Department of Philosophy II, RUB Bochum, Germany}
\email{sasander@me.com}
%\keywords{Finite sets, representations, computability theory, Kleene S1-S9, bounded variation, piecewise continuous}
%\subjclass[2010]{03B30, 03F35, 03D55, 03D30}
\begin{abstract}
The aim of \emph{Reverse Mathematics} (RM for short) is to find the minimal axioms needed to prove a given theorem of ordinary mathematics.
These minimal axioms are almost always \emph{equivalent} to the theorem, working over the \emph{base theory} of RM, a weak system of computable mathematics.  
The \emph{Big Five phenomenon} of RM is the observation that a large number of theorems from ordinary mathematics are either provable in the base theory or equivalent to one of only four systems; these five systems together are called the `Big Five'.
The aim of this paper is to greatly extend the Big Five phenomenon as follows: there are two supposedly \emph{fundamentally different} approaches to RM where the main difference is whether the language is restricted to \emph{second-order} objects or if one allows \emph{third-order} objects.  
%For various reasons, these two strands of RM are generally treated or regarded as \emph{fundamentally different}.
In this paper, we unite these two strands of RM by establishing numerous equivalences involving the \textbf{second-order} Big Five systems on one hand, 
%\emph{weak K\"onig's lemma} and e.g.\ \emph{arithmetical comprehension} on one hand 
and well-known \textbf{third-order} theorems from analysis about (possibly) discontinuous functions on the other hand.  
We both study relatively tame notions, like cadlag or Baire 1, and potentially wild ones, like quasi-continuity.  
We also show that \emph{slight} generalisations and variations of the aforementioned third-order theorems fall \emph{far} outside of the Big Five.  
\end{abstract}

%\setcounter{page}{0}
%\tableofcontents
%\thispagestyle{empty}
%\newpage

\maketitle
\thispagestyle{empty}

%\vspace{-7mm}
\section{Introduction and preliminaries}\label{intro}
\subsection{Short summary}\label{sintro}
The aim of the program \emph{Reverse Mathematics} (RM for short) is to find the minimal axioms needed to prove a given theorem of ordinary mathematics.
In a nutshell, the aim of this paper is to greatly extend the so-called Big Five phenomenon, a central topic in RM according to Montalb\'an, as follows. 
\begin{quote}
[...] we would still claim that the great majority of the theorems from classical mathematics are equivalent to one of the big five. This phenomenon is still quite striking. Though we have some sense of why this phenomenon occurs, we really do not have a clear explanation for it, let alone a strictly logical or mathematical reason for it. The way I view it, gaining a greater understanding of this phenomenon is currently one of the driving questions behind reverse mathematics. (see \cite{montahue}*{p.\ 432})
\end{quote}
In more detail, there are at least two supposedly \emph{fundamentally different}\footnote{This opinion is for instance expressed in the latest textbook on RM, namely in \cite{damurm}*{\S12.4}.} approaches to RM where the main difference is whether the language is restricted to \emph{second-order} objects or if one allows \emph{third-order} objects.   
In this paper, we unite these two strands of RM by establishing numerous equivalences involving the \textbf{second-order} Big Five systems on one hand, 
and well-known \textbf{third-order} theorems from analysis about (possibly) discontinuous functions on the other hand.  
We both study relatively `tame' notions, like cadlag and Baire 1, and potentially `wild' ones, like quasi-continuity.  
We also show that \emph{slight} generalisations and variations of the aforementioned third-order theorems fall \emph{far} outside of the Big Five and much stronger (second- and higher-order) systems.  
The reader will agree that while our results are comprehensive, they only scratch the surface of what is possible and lead the way to a whole new research area. 
In evidence, we sketch analogous results for the RM of the second-order \emph{weak weak K\"onig's lemma} and the third-order Vitali covering theorem for uncountable coverings in Section \ref{ferengi}.

\smallskip

Finally, we discuss the detailed aim and motivation of this paper within RM in Section~\ref{krum} and introduce essential definitions in Section \ref{prelim}.

\subsection{Aim and motivation}\label{krum}
Reverse Mathematics (RM for short) is a program in the foundations of mathematics initiated by Friedman (\cites{fried, fried2}) and developed extensively by Simpson and others (\cites{simpson1, simpson2,damurm}); an introduction to RM for the `mathematician in the street' may be found in \cite{stillebron}.  We assume basic familiarity with RM, including Kohlenbach's \emph{higher-order} RM introduced in \cite{kohlenbach2}, while a brief sketch may be found in Section \ref{introrm}.  
Recent developments in higher-order RM, including our own, are in \cite{dagsamV, dagsamVI, dagsamIII, dagsamX, dagsamVII, dagsamIX, dagsamXI}.
All equivalences are proved over Kohlenbach's base theory $\RCAo$ (or slight extensions), as defined in the appendix (Section~\ref{rmbt}).

\smallskip

The biggest difference between `classical' RM and higher-order RM is that the former makes use of the language of \emph{second-order} arithmetic, while the latter uses the language of \emph{higher-order} arithmetic (see Section \ref{introrm} for details).  
Thus, higher-order objects are only indirectly available via so-called codes or representations in classical RM.  It is then a natural question -in the very spirit of RM- what the connection is between third-order objects and their second-order codes.    
Now, continuous functions constitute perhaps the most basic case study and Kohlenbach in \cite{kohlenbach4}*{\S4} studies the connection between:
\begin{itemize}
\item third-order functions on Baire or Cantor space that satisfy the standard `epsilon-delta' definition of continuity,
\item second-order codes for continuous functions on Baire or Cantor space, following the definition from \cite{simpson2}*{II.6}.
\end{itemize}
Kohlenbach shows that \emph{weak K\"onig's lemma} ($\WKL$ for short) suffices to show that a (third-order) continuous function on Cantor space can be represented by a code.
In Section \ref{harf}, we adapt some of Kohlenbach's results to the unit interval, which turns out to be surprisingly hard.  
The representation of the reals in (both second- and higher-order) RM may be found in Section~\ref{kkk}.

\smallskip

With these `coding results' on $[0,1]$ in place, we establish in Section~\ref{TOT} equivalences between the second Big Five system $\WKL$ and the following third-order theorems; all definitions may be found in Section \ref{franticz}.  
\begin{itemize}
\item A cadlag function on the unit interval is bounded (or: Riemann integrable). 
\item A cadlag function on the unit interval has a supremum. 
\item A regulated function on the unit interval is bounded.
\item A bounded upper semi-continuous\footnote{A `famous' recent reference for the study of semi-continuity is Villani's work \cite{viams}.} function on $[0,1]$ has a supremum.  
\item A bounded Baire 1 function $F:[0,1]\di \R$ has a supremum.
\item A bounded upper semi-continuous function on the unit interval that has a supremum, attains it.
\item Cousin's lemma for cadlag (or: lower semi-continuous) functions.
\item Cousin's lemma for regulated $F: [0,1]\di \R$ such that $F(x)=\frac{F(x-)+F(x+)}{2}$ for all $x\in [0,1]$.
\item Cousin's lemma for quasi-continuous functions.
\item Cousin's lemma for Baire 1 functions.
\item \dots
\end{itemize}
While cadlag -or even Baire 1- functions can be said to be `close to continuous', quasi-continuous functions can be quite exotic, as discussed in Remark \ref{donola}.

\smallskip

We obtain similar equivalences for the other Big Five systems, namely $\ACA_{0}$ (Section \ref{hazaha}), $\ATR_{0}$ (Section \ref{FOUR}), and $\FIVE$ (Section~\ref{FIVE}), involving the Jordan decomposition theorem, Cousin's lemma, and supremum principles.  We suggest many other possible equivalences involving third-order theorems, i.e.\ this paper may be lengthy but only scratches the surface of what is possible. 
In evidence, we sketch similar results for $\WWKL_{0}$ and the Vitali covering theorem in Section~\ref{ferengi}.
Thus, the distinction between second- and third-order statements does not seem that crucial to RM as there are \textbf{many} interesting equivalences across this distinction.

\smallskip

Now, many of the aforementioned results are based on the higher-order RM of the following central axiom from \cite{kohlenbach2}:
\be\tag{$\exists^{2}$}
(\exists E:\N^{\N}\di \{0,1\})(\forall f\in \N^{\N})\big( (\exists n\in \N)(f(n)=0)\asa E(f)=0\big).
\ee
The functional $E:\N^{\N}\di \N$ is \emph{dis}continuous at $f=11\dots$ and is usually called `Kleene's quantifier $\exists^{2}$'.
Kohlenbach shows the equivalence between the existence of a discontinuous function on $\R$ and $(\exists^{2})$ in \cite{kohlenbach2}*{\S3}.
We establish a number of interesting equivalences for $(\exists^{2})$ in Section \ref{neweqi}, including the well-known fact that the Riemann integrable functions are not closed under composition. 

\smallskip

Finally, we show in Section \ref{XXX} that \emph{slight} variations or generalisations of all the aforementioned third-order statements cannot be proved from the Big Five or $(\exists^{2})$, and \emph{much} stronger systems.
This is done by deriving from these statements the following version of the uncountability of the reals:
\[
\NIN_{[0,1]}: \textup{there is no injection from the unit interval $[0,1]$ to $\N$}.
\]
Basic mathematical fact as $\NIN_{[0,1]}$ may be, it cannot be proved in $\Z_{2}^{\omega}$ from Section~\ref{lll}, which is a conservative extension of second-order arithmetic $\Z_{2}$.
As a side-result, many well-known inclusions among function spaces, like the statement \emph{all regulated functions are Baire 1}, also imply $\NIN_{[0,1]}$; these inclusions can therefore not be proved in the Big Five and much stronger systems. 

\smallskip

In conclusion, many third-order statements fall into the Big Five classification, while \emph{slight} variations or generalisations of the former fall \emph{far outside} this classification.  
We have no explanation for this phenomenon at this point.

\subsection{Preliminaries}\label{prelim}
We briefly discuss Reverse Mathematics (Section \ref{introrm}) and introduce some mainstream definitions (Section \ref{franticz}).
\subsubsection{Introducing Reverse Mathematics}\label{introrm}
We refer to \cite{stillebron} for a basic introduction to RM and to \cite{simpson2, simpson1,damurm} for an overview of RM.  We expect familiarity with RM, including Kohlenbach's \emph{higher-order} RM from \cite{kohlenbach2}.  A more detailed description of the latter, including the definition of the base theory $\RCAo$, can be found in a technical appendix (Section \ref{HORMreduction}).  
We do introduce the language of higher-order RM, namely as follows.   

\smallskip

First of all, in contrast to `classical' RM based on $\L_{2}$, the language of \emph{second-order arithmetic} $\Z_{2}$, higher-order RM uses $\L_{\omega}$, the richer language of \emph{higher-order arithmetic}.  
Indeed, while $\L_{2}$ is restricted to natural numbers and sets of natural numbers, $\L_{\omega}$ can accommodate sets of sets of natural numbers, sets of sets of sets of natural numbers, et cetera.  
To formalise this idea, we introduce the collection of \emph{all finite types} $\mathbf{T}$, defined by the two clauses:
\begin{center}
(i) $0\in \mathbf{T}$   and   (ii)  if $\sigma, \tau\in \mathbf{T}$ then $( \sigma \di \tau) \in \mathbf{T}$,
\end{center}
where $0$ is the type of natural numbers, and $\sigma\di \tau$ is the type of mappings from objects of type $\sigma$ to objects of type $\tau$.
In this way, $1\equiv 0\di 0$ is the type of functions from numbers to numbers, and  $n+1\equiv n\di 0$.  Viewing sets as given by characteristic functions, we note that $\Z_{2}$ only deals with objects of type $0$ and $1$.    

\smallskip

Secondly, the language $\L_{\omega}$ includes variables $x^{\rho}, y^{\rho}, z^{\rho},\dots$ of any finite type $\rho\in \mathbf{T}$.  Types may be omitted when they can be inferred from context.  
The constants of $\L_{\omega}$ include the type $0$ objects $0, 1$ and $ <_{0}, +_{0}, \times_{0},=_{0}$  which are intended to have their usual meaning as operations on $\N$.
Equality at higher types is defined in terms of `$=_{0}$' as follows: for any objects $x^{\tau}, y^{\tau}$, we have
\be\label{aparth}
[x=_{\tau}y] \equiv (\forall z_{1}^{\tau_{1}}\dots z_{k}^{\tau_{k}})[xz_{1}\dots z_{k}=_{0}yz_{1}\dots z_{k}],
\ee
if the type $\tau$ is composed\footnote{We recall the convention of right associativity of the type arrow, i.e.\ the type $\tau\equiv(\tau_{1}\di \dots\di \tau_{k}\di 0)$ stands for $\tau_{1}\di(\tau_{2}\di (\dots \di( \tau_{k}\di 0 )\dots))$.  
} as $\tau\equiv(\tau_{1}\di \dots\di \tau_{k}\di 0)$.  
Furthermore, $\L_{\omega}$ also includes the \emph{recursor constant} $\mathbf{R}_{\sigma}$ for any $\sigma\in \mathbf{T}$, which allows for iteration on type $\sigma$-objects. 
Formulas and terms are defined as usual.  

\smallskip

Thirdly, while not strictly speaking necessary, it is often convenient to explicitly include types for \emph{finite sequences} of objects.
For a given type $\rho$, the associated type $\rho^{*}$ is the type of finite sequences of type $\rho$ objects.  We discuss the latter and related notations in detail in Notation \ref{skim}.  

\smallskip

Finally, sets of objects of any finite type can be represented via characteristic functions in $\L_{\omega}$, an approach well-known from measure and probability theory and adopted in this paper as in Definition \ref{openset}.

\subsubsection{Some definitions}\label{franticz}
We introduce some standard definitions from analysis, all rather mainstream and taking place in $\RCAo$.  

\smallskip

First of all, we use the standard definition of (uniform) continuity as follows, where $I\equiv[0,1]$ is the unit interval.  
\bdefi[Continuity]~
\begin{itemize}
\item A function $F:[0,1]\di \R$ is \emph{continuous at} $x\in [0,1]$ if
\be\label{tunt}\textstyle
(\forall k\in \N)(\exists N\in \N)(\forall y\in [0,1])(|x-y|<\frac{1}{2^{N}} \di |F(x)-F(y)|<\frac{1}{2^{k}}  ).
\ee
A function $F:[0,1]\di \R$ is \emph{continuous} if \eqref{tunt} holds for all $x\in [0,1]$
\item A \emph{modulus} of continuity is any $G:(\N\times \R)\di \N$ such that $G(k, x)=N$ as in \eqref{tunt}, for $k\in \N, x\in [0,1]$. 
\item A function $F:[0,1]\di \R$ is \emph{uniformly} continuous if:
\be\label{tunt2}\textstyle
(\forall k\in \N)(\exists N\in \N)(\forall x,y\in [0,1])(|x-y|<\frac{1}{2^{N}} \di |F(x)-F(y)|<\frac{1}{2^{k}}  ).
\ee
\item A \emph{modulus of uniform continuity} is any $h:\N\di \N$ such that $h(k)=N$ as in \eqref{tunt2} for any $k\in \N$.
\end{itemize}
\edefi
Secondly, we shall study the following weaker notions, many of which are well-known and hark back to the days of Baire, Darboux, Hankel, and Volterra (\cites{beren,beren2,darb, volaarde2,hankelwoot,hankelijkheid}).  
% (see e.g.\ \cites{nieuwebron, kowalski}).
We will use `sup' and related operators in the same `virtual' or `comparative' way as in second-order RM (see e.g.\ \cite{simpson2}*{X.1}).  In this way, a formula of the form `$\sup A>a$' makes sense as shorthand for a formula in the language of all finite types, even when $\sup A$ need not exist in $\RCAo$.  
As in \cite{basket, basket2}, the definition of Baire $n$-function proceeds via (external) induction over standard $n$.
Sets are defined in Definition \ref{openset} below, namely via characteristic functions.
\bdefi\label{flung} 
For $f:[0,1]\di \R$, we have the following definitions:
\begin{itemize}
\item $f$ is \emph{upper semi-continuous} at $x_{0}\in [0,1]$ if $f(x_{0})\geq_{\R}\lim\sup_{x\di x_{0}} f(x)$,
\item $f$ is \emph{lower semi-continuous} at $x_{0}\in [0,1]$ if $f(x_{0})\leq_{\R}\lim\inf_{x\di x_{0}} f(x)$,
\item $f$ is \emph{quasi-continuous} at $x_{0}\in [0, 1]$ if for $ \epsilon > 0$ and an open neighbourhood $U$ of $x_{0}$, 
there is a non-empty open ${ G\subset U}$ with $(\forall x\in G) (|f(x_{0})-f(x)|<\eps)$.
\item $f$ is \emph{cliquish} at $x_{0}\in [0, 1]$ if for $ \epsilon > 0$ and an open neighbourhood $U$ of $x_{0}$, 
there is a non-empty open ${ G\subset U}$ with $(\forall y, z\in G) (|f(y)-f(z)|<\eps)$.
%\item $f$ is \emph{symmetrically continuous} at $x_{0}\in [0, 1]$ if the limit $\lim_{h\di 0} [f(x_{0}+h)-f(x_{0}-h)]$ equals $0$ (\cite{housen}).
\item $f$ is \emph{regulated} if for every $x_{0}$ in the domain, the `left' and `right' limit $f(x_{0}-)=\lim_{x\di x_{0}-}f(x)$ and $f(x_{0}+)=\lim_{x\di x_{0}+}f(x)$ exist.  
\item $f$ is \emph{c\`adl\`ag} if it is regulated and $f(x)=f(x+)$ for $x\in [0,1)$.
\item $f$ is \emph{Darboux} if it has the intermediate value property, i.e.\ if $a, b\in [0,1], c\in \R$ are such that $a\leq b$ and either $f(a)\leq c\leq f(b)$ or $f(b)\leq c\leq f(a)$, then there is $d\in [a,b]$ with $f(d)=c$.
%\item $f$ is \emph{Baire 1} if it is the pointwise limit of a sequence of continuous functions. 
\item $f$ is \emph{Baire 0} if it is a continuous function. 
\item $f$ is \emph{Baire $n+1$} if it is the pointwise limit of a sequence of Baire $n$ functions.
%\item $f$ is \emph{Baire 1} if it is the pointwise limit of a sequence of continuous functions. 
%\item $f$ is \emph{Baire 2} if it is the pointwise limit of a sequence of Baire 1 functions.
%\item $f$ is \emph{effectively Baire 2} if there is a double sequence $(f_{n, m})_{n, m\in \N}$ of continuous functions such that $\lim_{n\di \infty}\lim_{m\di \infty}f_{n,m}(x)=f(x)$ for all $x\in [0,1]$. 
\item $f$ is \emph{effectively Baire $n$} $(n\geq 2)$ if there is a sequence $(f_{m_{1}, \dots, m_{n}})_{m_{1}, \dots, m_{n}\in \N}$ of continuous functions such that for all $x\in [0,1]$, we have 
\[\textstyle
f(x)=\lim_{m_{1}\di \infty}\lim_{m_{2}\di \infty}\dots \lim_{m_{n}\di \infty}f_{m_{1},\dots ,m_{n}}(x).
\]
%\item $f$ is \emph{effectively Baire $n+3$} if it is the pointwise limit of a sequence of effectively Baire $n+2$ functions. 
%there is a sequence $(f_{n})_{n\in \N}$ of effectively Baire $n+2$ functions such that $\lim_{n\di \infty}\lim_{m\di \infty}f_{n,m}(x)=f(x)$ for all $x\in [0,1]$. 
\item $f$ is \emph{Baire 1$^{*}$} if\footnote{The notion of Baire 1$^{*}$ goes back to \cite{ellis} and equivalent definitions may be found in \cite{kerkje}.  
In particular,  Baire 1$^{*}$ is equivalent to the Jayne-Rogers notion of \emph{piecewise continuity} from \cite{JR}.} there is a sequence of closed sets $(C_{n})_{n\in \N}$ such $[0,1]=\cup_{n\in \N}C_{n}$ and $f_{\upharpoonright C_{m}}$ is continuous for all $m\in \N$.
\item $f$ is \emph{continuous almost everywhere} if it is continuous at all $x\in [0,1]\setminus E$, where $E$ is a measure zero\footnote{A set $A\subset \R$ is \emph{measure zero} if for any $\eps>0$ there is a sequence of basic open intervals $(I_{n})_{n\in \N}$ such that $\cup_{n\in \N}I_{n}$ covers $A$ and has total length below $\eps$.  Note that this notion does not depend on (the existence of) the Lebesgue measure.} set.
\item $f$ is \emph{pointwise discontinuous} if for any $x\in [0,1]$ and $\eps>0$, there is $y\in [0,1]$ such that $f$ is continuous at $y$ and $|x-y|<\eps$ (Hankel, 1870, \cite{hankelwoot}).  % i.e.\ the continuity points are dense. 
\end{itemize}
\edefi
As to notations, a common abbreviation is `usco' and `lsco' for the first two items, while one often just writes `cadlag', i.e.\ without the accents.  
Moreover, if a function has a certain weak continuity property at all reals in $[0,1]$ (or its intended domain), we say that the function has that property.  

\smallskip

Regarding the notion of `effectively Baire $n$' in Definition \ref{flung}, the latter is used, using codes for continuous functions, in second-order RM (see \cite{basket, basket2}). 
Baire himself notes in \cite{beren2}*{p.\ 69} that Baire 2 functions can be \emph{represented} by effectively Baire~2 functions.  By Theorem \ref{reklam}, there is a 
significant difference between the latter two notions.    Similarly, cliquish functions are exactly those functions that can be expressed as the sum of two quasi-continuous functions (\cites{quasibor2, malin}).  
Nonetheless, comparing Theorems \ref{lebber} and \ref{reklam}, these notions behave fundamentally different in RM.  Analogously, functions continuous almost everywhere are exactly those functions that can be expressed as the sum of two `strong' quasi-continuous functions (see \cite{grand} for the latter notion).

\smallskip

Thirdly, the notion of \emph{bounded variation} (abbreviated $BV$) was first explicitly\footnote{Lakatos in \cite{laktose}*{p.\ 148} claims that Jordan did not invent or introduce the notion of bounded variation in \cite{jordel}, but rather discovered it in Dirichlet's 1829 paper \cite{didi3}.} introduced by Jordan around 1881 (\cite{jordel}) yielding a generalisation of Dirichlet's convergence theorems for Fourier series.  
Indeed, Dirichlet's convergence results are restricted to functions that are continuous except at a finite number of points, while functions of bounded variation can have (at most) countable many points of discontinuity, as already studied by Jordan, namely in \cite{jordel}*{p.\ 230}.
Nowadays, the \emph{total variation} of $f:[a, b]\di \R$ is defined as follows:
\be\label{tomb}\textstyle
V_{a}^{b}(f):=\sup_{a\leq x_{0}< \dots< x_{n}\leq b}\sum_{i=0}^{n-1} |f(x_{i})-f(x_{i+1})|.
\ee
If this quantity exists and is finite, one says that $f$ has bounded variation on $[a,b]$.
Now, the notion of bounded variation is defined in \cite{nieyo} \emph{without} mentioning the supremum in \eqref{tomb}; see also \cites{kreupel, briva, brima}.  
Hence, we shall distinguish between the following notions.  Jordan seems to use item \eqref{donp} of Definition~\ref{varvar} in \cite{jordel}*{p.\ 228-229}.
% providing further motivation for the functionals introduced in Definition \ref{JDR}.
\bdefi[Variations on variation]\label{varvar}
\begin{enumerate}  
\renewcommand{\theenumi}{\alph{enumi}}
\item The function $f:[a,b]\di \R$ \emph{has bounded variation} on $[a,b]$ if there is $k_{0}\in \N$ such that $k_{0}\geq \sum_{i=0}^{n-1} |f(x_{i})-f(x_{i+1})|$ 
for any partition $x_{0}=a <x_{1}< \dots< x_{n-1}<x_{n}=b  $.\label{donp}
\item The function $f:[a,b]\di \R$ \emph{has {a} variation} on $[a,b]$ if the supremum in \eqref{tomb} exists and is finite.\label{donp2}
\end{enumerate}
\edefi
\noindent
The fundamental theorem about $BV$-functions (see e.g.\  \cite{jordel}*{p.\ 229}) is as follows.
\begin{thm}[Jordan decomposition theorem]\label{drd}
A function $f : [0, 1] \di \R$ of bounded variation is the difference of two non-decreasing functions $g, h:[0,1]\di \R$.
\end{thm}
Theorem \ref{drd} has been studied extensively via second-order representations in e.g.\ \cites{groeneberg, kreupel, nieyo, verzengend}.
The same holds for constructive analysis by \cites{briva, varijo,brima, baathetniet}, involving different (but related) constructive enrichments.  
Now, arithmetical comprehension suffices to derive Theorem \ref{drd} for various kinds of second-order \emph{representations} of $BV$-functions in \cite{kreupel, nieyo}.
% i.e.\ finite iterations of the Turing jump suffice to compute the associated Jordan decomposition.  
By contrast, the results in \cite{dagsamXII,dagsamX, dagsamXI, dagsamXIII} show that the Jordan decomposition theorem is even `explosive': combining with the Suslin functional from $\FIVE^{\omega}$ (see Section \ref{lll}), one derives $\SIX$.  

\smallskip

Fourth, we shall make use of the following notion of (open and closed) set, which was studied in e.g.\ \cite{dagsamXII,dagsamX, dagsamXI, dagsamXIII, samcount}.
%We motivate this choice in detail in Section \ref{crux}.  
\bdefi[Sets in $\RCAo$]\label{openset}
We let $Y: \R \di \R$ represent subsets of $\R$ as follows: we write `$x \in Y$' for `$Y(x)>_{\R}0$' and call a set $Y\subseteq \R$ Ôopen' if for every $x \in Y$, there is an open ball $B(x, \frac{1}{2^{N}}) \subset Y$ with $N\in \N$.  
A set $Y$ is called `closed' if the complement is open. 
\edefi
For open $Y$ as in the previous definition, the formula `$x\in Y$' has the same complexity (modulo higher types) as in second-order RM (see \cite{simpson2}*{II.5.6}), while given $(\exists^{2})$ from Section \ref{intro}, the former becomes a `proper' characteristic function, only taking values `0' and `$1$'.  Hereafter, an `open set' refers to Definition~\ref{openset}, while `RM-open set' refers to the second-order definition from RM.  
For simplicity, we sometimes assume $\ACAo\equiv \RCAo+(\exists^{2})$ and work with characteristic functions of open sets directly.  
Nonetheless, combining Theorem \ref{plofkip} and \cite{simpson2}*{II.7.1}, an RM-open set is indeed an open set as in Definition~\ref{openset}, working over $\RCAo$.

\smallskip

Next, the notion of `countable set' can be formalised in various ways, namely via Definitions \ref{eni} and \ref{standard}.
\bdefi[Enumerable sets of reals]\label{eni}
A set $A\subset \R$ is \emph{enumerable} if there exists a sequence $(x_{n})_{n\in \N}$ such that $(\forall x\in \R)(x\in A\di (\exists n\in \N)(x=_{\R}x_{n}))$.  
\edefi
This definition reflects the RM-notion of `countable set' from \cite{simpson2}*{V.4.2}.  
We note that given $\mu^{2}$ from Section \ref{lll}, we may replace the final implication in Definition~\ref{eni} by an equivalence. 
Our definition of `countable set' is now as follows in $\RCAo$. 
\bdefi[Countable subset of $\R$]\label{standard}~
A set $A\subset \R$ is \emph{countable} if there exists $Y:\R\di \N$ such that $(\forall x, y\in A)(Y(x)=_{0}Y(y)\di x=_{\R}y)$. 
If $Y:\R\di \N$ is also \emph{surjective}, i.e.\ $(\forall n\in \N)(\exists x\in A)(Y(x)=n)$, we call $A$ \emph{strongly countable}.
\edefi
The first part of Definition \ref{standard} is from Kunen's set theory textbook (\cite{kunen}*{p.~63}) and the second part is taken from Hrbacek-Jech's set theory textbook \cite{hrbacekjech} (where the term `countable' is used instead of `strongly countable').  
For the rest of this paper, `strongly countable' and `countable' shall exclusively refer to Definition \ref{standard}, \emph{except when explicitly stated otherwise}.

\section{Main results}
\subsection{Introduction}
We obtain the following results in Sections \ref{harf}-\ref{XXX}.
\begin{itemize}
\item We study the connection between continuous functions on the reals and their codes in Section \ref{harf}, mostly working over $\RCAo$ or assuming $\WKL$.
\item We obtain numerous equivalences involving the Big Five and third-order theorems about (possibly) discontinuous functions (Sections \ref{TOT}-\ref{FOUR}). 
\item We obtain equivalences for $(\exists^{2})$ in Section \ref{neweqi} where the associated principles also stem from mainstream mathematics.  
\item In Section \ref{XXX}, we show that slight variations or generalisations from the third-order statements in the previous three items cannot be proved from the Big Five, $(\exists^{2})$, and much stronger systems, like $\Z_{2}^{\omega}$ from Section \ref{lll}. 
\end{itemize}
As discussed in Remark \ref{donola}, some of our results deal with functions `close to continuous', like the cadlag ones, while other results deal with functions that can be `far from continuous', like the quasi-continuous ones. 

\smallskip

Finally, we discuss some known results due to Kohlenbach regarding continuous and discontinuous functions in the following remark. 
\begin{rem}\label{LEM2}\rm
First of all, Kohlenbach establishes a number of interesting `coding results' for functions on $2^{\N}$ and $\N^{\N}$ in \cite{kohlenbach4}*{\S4}, as follows. 
\begin{itemize}
\item By \cite{kohlenbach4}*{Theorem 4.4}, $\RCAo$ proves the equivalence between the following for a functional $Y:\N^{\N}\di \N$ continuous on Baire space $\N^{\N}$:
\begin{itemize}
\item the functional $Y$ has a \emph{continuous} modulus of continuity,
\item there is a total \emph{RM-code} (\cite{simpson2}*{II.6.1}) that equals $Y$ on $\N^{\N}$,
\item there is a total \emph{Kleene associate} (\cite{kohlenbach4}*{Def.\ 4.3}) that equals $Y$ on $\N^{\N}$.
\end{itemize}
\item Using a construction due to Dag Normann, $\RCAo+\WKL$ proves that a continuous $Y:2^{\N}\di \N^{\N}$ has a modulus of (uniform) continuity (\cite{kohlenbach4}*{Prop.~4.10}).  
By the previous items, there is also an RM-code that equals $Y$ on $2^{\N}$.
In this way, the usual second-order RM-results apply to such $Y$, namely via the aforementioned code.  
For instance, over $\RCAo+\WKL$, such $Y$ is bounded on $2^{\N}$ by \cite{simpson2}*{IV.2.2}, and similar results apply immediately. 
\end{itemize}
Secondly, working over $\RCAo$, Kohlenbach establishes a number of interesting equivalences involving discontinuous functions in \cite{kohlenbach2}*{\S3}, as follows.
\begin{itemize}
\item The axiom $(\exists^{2})$ from Section \ref{intro} is equivalent to the existence of a discontinuous function on $\R$, like e.g.\ Heaviside's function. 
\item The axiom $(\exists^{2})$ from Section \ref{intro} is equivalent to $(\mu^{2})$, i.e.\ the existence of Feferman's mu-operator from Section \ref{lll}.
\end{itemize}
Using classical logic, the first item yields that $\neg(\exists^{2})$ is equivalent to \emph{Brouwer's theorem}, i.e.\ the statement that all functions on $\R$ are continuous.  
In the below, we will make use of the above facts, often without very detailed references.  
\end{rem}

\subsection{From codes to continuous functions and back again}\label{harf}
We establish the following connections between continuous functions on the reals and their codes.
\begin{itemize}
\item A code for a continuous function on $\R$ represents a third-order continuous function, working over $\RCAo$ (Theorem \ref{plofkip}).
\item A third-order continuous function on $[0,1]$ can be represented by an RM-code (Theorem \ref{gofusefl}), working over $\RCAo+\WKL$.  
\item Over $\RCAo$, $\WKL$ is equivalent to basic properties of (third-order) continuous functions on the unit interval (Theorem \ref{ploppy}).
\end{itemize}
The proof of Theorem \ref{gofusefl} is rather involved, while similar results like the boundedness of continuous functions, have (more) basic proofs by Theorems \ref{liguster} and \ref{liguster2}.

\smallskip

First of all, $\RCAo$ is a conservative extension of $\RCA_{0}$ (see e.g.\ Remark~\ref{ECF}).  
In this light, it is desirable that theorems of $\RCA_{0}$ also yield theorems of $\RCAo$.  Given the coding practise of RM, this is not always straightforward and we therefore establish Theorem \ref{plofkip}, which expresses that (second-order) codes for continuous functions give rise to third-order continuous functions, working in the base theory.
Our definition of `RM-code for continuous function' is the standard one (\cite{simpson2}*{II.6.1}) and as in the latter, we often identify a code and the function it represents. 
The following proof is also evidence for the necessity of $\QFAC^{1,0}$ in $\RCAo$.
\begin{thm}[$\RCAo$]\label{plofkip}
Let $\Phi$ be an RM-code for an $\R\di \R$-function.  There is a third-order $F:\R\di \R$ such that $F(x)$ equals the value of $\Phi$ at $x$ for any $x\in \R$. 
\end{thm}
\begin{proof}
For total RM-codes of functionals $\N^{\N}\di \N$, one applies $\QFAC^{1,0}$ to: 
\[
\textup{`the RM-code is defined at each point of $\N^{\N}$'}
\]
to obtain a third-order functional $\Psi:\N^{\N}\di \N$ equal to the (value of the) code everywhere; this argument may be found in e.g.\ the proof of 1) $\di$ 3) in \cite{kohlenbach4}*{Prop.~4.4}.  
We now show that the same procedure works for RM-codes of $[0,1]\di \R$-functions.  Indeed, a code for an $\R\di \R$-function is a set $\Phi\subset [\N\times \Q\times \Q^{+}\times \Q\times \Q^{+}] $ satisfying certain properties.   
The formula `$\Phi$ is total on $\R$' has the following form (which is suitable for $\QFAC^{1,0}$):
\be\label{genemensintereeesthet}\textstyle
(\forall x\in \R, k\in \N)(\exists (n, a, r, b, s)\in \Phi)(d(x, a)<_{\R}r \wedge s<_{\Q}\frac{1}{2^{k}}).  
\ee
Intuitively, the fourth component $b\in \Q $ of $\Phi$ contains rational approximations to the value of the code $\Phi$ at $x\in \R$, while $s\in \Q$ is an upper bound on the difference between $b$ and the value of $\Phi$ at $x\in \R$.  
Hence, apply $\QFAC^{1,0}$ to \eqref{genemensintereeesthet} to obtain $G$ such that $G(x, k)$ is the quintuple as in \eqref{genemensintereeesthet}.  
Note that $G(x, k)(4)$ may not be extensional on the reals as in item \eqref{EXTEN} in Definition \ref{keepintireal}.  Now define $F:\R\di \R$ by $[F(x)](k):=G(x, k+1)(4)$ and note that $F$ is indeed extensional on the reals.  
Clearly, $F(x)$ equals the value of $\Phi$ at every $x\in \R$. 
\end{proof}
\noindent
Unfortunately, the theorem does not generalise to codes for Baire 1 functions (in the sense of \cites{basket, basket2}).   
Indeed, by Theorem \ref{flame}, $(\exists^{2})$ is equivalent to the statement that a code for a Baire 1 function represents a third-order function. 

\smallskip
 
Secondly, by Theorem \ref{plofkip}, we can make the leap from `second-order codes for continuous functions' to `third-order continuous functions' without problems. 
Theorem \ref{gofusefl} expresses that the other direction is possible too, additionally assuming \emph{weak K\"onig's lemma} $\WKL$ in the base theory.  
As will become clear, the associated proof is based on that of \cite{kohlenbach4}*{Prop.\ 4.10}, which is in turn based on a construction due to Dag Normann, as noted in \cite{kohlenbach4}*{p.\ 94}.
\begin{thm}[$\RCAo+\WKL$]\label{gofusefl}
Any $F:\R\di \R$ continuous on $[0,1]$ has a modulus of uniform continuity $h:\N\di \N$ on $[0,1]$.
\end{thm}
\begin{proof}
First of all, \cite{kohlenbach4}*{Prop.\ 4.10} establishes that, working over $\RCAo+\WKL$, any $F:\N^{\N}\di \N^{\N}$ continuous on $2^{\N}$ has a modulus of \emph{uniform} continuity. 
In the proof of \cite{kohlenbach4}*{Prop.\ 4.4}, there is an explicit formula for an RM-code defined in terms of such a modulus.  
For completeness, we now sketch the proof of \cite{kohlenbach4}*{Prop.~4.10}, which consists of two steps.  
As a first step, the formula $A(k, \sigma)$ in \eqref{gorgkkk} is a slight modification of the innermost universal formula in the definition of continuity for $F$ on $2^{\N}$, where $\sigma^{0^{*}}\leq_{0^{*}}1$ is a finite binary sequence: 
\be\label{gorgkkk}
(\forall  g,h\leq_{1}1)(\overline{g}|\sigma|=_{0^{*}}\sigma=_{0^{*}} \overline{h}|\sigma| \di F(g)(k)=F(h)(k))  \big].
\ee
By definition, we have $(\forall f\leq_{1}1, k^{0})(\exists N^{0})A(k, \overline{f}N)$.  Despite the quantifiers in \eqref{gorgkkk}, $\WKL$ suffices to define its characteristic function $\chi_{A}^{(0\times 0^{*})\di 0}$, i.e.\ we have
\be\label{surelovethat}
(\forall \sigma^{0^{*}}\leq_{0^{*}}1, k^{0})( \chi_{A}(k, \sigma) =0 \asa  A(k, \sigma) ).
\ee
The existence of $\chi_{A}$ is proved in the next paragraph of this proof.  Now, $\sigma \in T_{k}\asa \neg A(k, \sigma)$ defines a 0/1-tree $T_{k}$, which has no path by the above.  
By $\WKL$, the tree $T_{k}$ is finite for any $k^{0}$, implying $(\forall k^{0})(\exists N^{0})(\forall f\leq_{1}1)A(k, \overline{f}N)$.  
The latter yields
\[
(\forall k^{0})(\exists N^{0})(\forall \sigma^{0^{*}} \leq_{0^{*}}1)[ |\sigma| =N\di  A(k, \overline{\sigma}N)],
\]
and applying $\QFAC^{0,0}$ readily yields the required modulus of uniform continuity.  

\smallskip

As a second step, we now establish the existence of $\chi_{A}$ as in \eqref{surelovethat}. 
Due to the continuity of $F$, it suffices to prove the existence of $\chi$ such that:
\[
(\forall \sigma^{0^{*}}\leq 1,  k^{0})\big( \chi(\sigma, k)=0\asa  (\forall \tau^{0^{*}}\leq 1)(F(\sigma*\tau*00\dots)(k)=F(\sigma*00\dots)(k))   \big).
\]
Now define a sequence of $0/1$ trees as follows: $\tau \in T_{k, \sigma}$ in case either of the following:
\begin{itemize}
\item $(\forall \gamma^{0^{*}}\leq 1)\big(|\gamma|\leq |\tau|  \di F(\sigma*\tau*00\dots)(k)=F(\sigma* \gamma*00\dots)(k)\big)$,
\item $(\exists \tilde{\gamma}^{0^{*}}\leq 1)(\exists l\leq |\tau| ) \big(  \tau= \tilde{\gamma}*\overline{00\dots}l$ with $|\tilde{\gamma}|$ minimal such that:
\[
F(\sigma * \tilde{\gamma}*00\dots)(k)\ne F(\sigma*00\dots)(k)\big).
\]
\end{itemize}
Now, each tree $T_{k, \sigma}$ is infinite and by the sequential version of $\WKL$ (equivalent to $\WKL$ by \cite{kooltje}*{Prop.\ 3.1}), there is a sequence of paths $f_{k, \sigma}$ in $T_{k, \sigma}$ for $k\in \N$ and $\sigma^{0^{*}}\leq1$.  
Using the continuity of $F$, one readily verifies that for $ \sigma^{0^{*}}\leq 1,  k^{0}$:
\[
 (\forall \tau^{0^{*}}\leq 1)(F(\sigma*\tau*00\dots)(k)=F(\sigma*00\dots)(k))\asa F(\sigma*f)(k)=F(\sigma*00\dots)(k).
\]
which is as required to obtain \eqref{surelovethat}.  For the next paragraph, we point out the following, assuming a fixed enumeration of all finite sequences: 
if we require that in the second item defining $T_{k, \sigma}$, the sequence $\tilde{\gamma} $ is the minimal sequence with the stated property, measured by sequence number, then $\neg A(k,\sigma)$ implies that $T_{k,\sigma}$ has a single branch witnessing $\neg A(k,\sigma)$, while if $A(k,\sigma)$ holds, any branch in $T_{k,\sigma}$ will witness $A(k,\sigma)$. 

\smallskip

Finally, we modify the previous paragraph to accommodate functions continuous on the unit interval. 
For convenience, we work with \emph{ternary} trees where a tree element $\sigma \in \{-1, 0, 1\}^{<\N}$ is a finite sequence in the alphabet $\{-1, 0, 1\}$.  
Similarly, $f \in \{-1, 0, 1\}^{\N}$ means that $f(k)\in \{-1,0, 1\}$ for all $k\in \N$.  
Clearly, each $f \in \{-1, 0, 1\}^{\N}$ codes a real number $\rho(f) = 1/2 + \sum_{n=0}^{\infty}f(n)2^{-(n+2)}$, where the partial sums form a fast converging Cauchy-sequence as in Definition \ref{keepintireal}.
Now, it is well-known in computer science that any $k$-ary tree admits a representation as a binary tree (see \cite{onderdeknuth, knuthzelf}), and the associated (effective) conversion is sometimes called the \emph{Knuth transform} (\cite{pfff}*{p.\ 146}).    As expected, the latter is readily formalised in $\RCA_{0}$ and hence $\WKL$ is equivalent to the existence of a path for infinite ternary trees, and the same for the associated sequential versions from \cite{kooltje}*{Prop.\ 3.1}.

\smallskip

Next, fix $F:\R \rightarrow \R$ continuous on $[0,1]$ and consider the formula:
\[\textstyle
A(\sigma, n)\equiv (\forall g \in \{-1,0,1\}^\N)[|F(\rho(\sigma \ast 00\dots)) - F(\rho(\sigma \ast g))| \leq \frac{1}{2^n}].
\]
We now use $\WKL$ to prove the existence of a function $B^{(0^{*}\times 0)\di 0}$ such that for all $\sigma \in \{-1,0,1\}^{<\N}$ and $n \in \N$:
\be\label{pimpo}
A(\sigma, n+1) \di (B(\sigma, n)=0)\di A(\sigma, n).
\ee
Using \eqref{pimpo}, one readily finds a modulus of uniform continuity for $F$ as in the first part of the proof.
In order to define $B$ satisfying \eqref{pimpo}, we define a sequence $S_{\sigma , n}$ of infinite ternary trees. 
By $\WKL$, these have a sequence of infinite branches, and the actual $B$ depends on which sequence of branches we select. 
%I base myself on the assumption that for all $f$, the distance between $F(f)$ and $F(f)(k)$ differs with at most $2^{-k}$, if the estimate is $2^{1-k}$ the construction must be adjusted accordingly. 
We use the convention that the elements of $\{-1,0,1\}^{<\N}$ are enumerated first by length, and then by the lexicographical ordering.

\smallskip

We now define $S_{\sigma,n}$ as follows: for the (finite) set of sequences $\gamma \in \{-1,0,1\}^{<\N}$ of length $k\in \N$, there are two cases to be considered, namely items \eqref{dagjan2} and \eqref{dagjan}.
\begin{enumerate} 
\item If for all $\gamma\in \{-1,0,1\}^{<\N}$ of length $k$, we have that for all $l \leq k$,\label{dagjan2}
\[\textstyle
 \big| [F(\rho(\sigma \ast 00\dots))](l) -[F(\rho(\sigma \ast \gamma \ast 00\dots))](l)\big | \leq_{\Q} 2^{-n} + 2^{1-l}
\] 
then all $\gamma \in \{-1,0,1\}^{<\N}$ of length $k$ are in $S_{\sigma , n}$.
\item If the previous item is false, there is a least $\gamma'$ of length $\leq k$ such that for some $l \leq k$ we have that 
\[
\big| [F(\rho(\sigma \ast 00\dots))](l) -[F(\rho(\sigma \ast \gamma' \ast 00\dots))](l)\big| >_{\Q} 2^{-n} + 2^{1-l}
\] 
We then let the extension $\gamma'\ast 0\cdots0$ to a sequence of length $k$ be in $S_{\sigma,n}$.\label{dagjan} 
\end{enumerate}
We now make two important observations about the trees $S_{\sigma, k}.  $
Firstly, if for a fixed $k\in \N$, there is a sequence of length $k$ in $S_{\sigma,n}$ following item \eqref{dagjan}, then the same sequence, only extended with zeros, will be the single sequence of length $k'$ for any $k' > k$.  
In this case, the only branch in $S_{\sigma,n}$ is a ternary $g^{1}$ such that $|F(\rho(\sigma \ast 0^\ast)) - F(\rho(\sigma \ast g))| > 2^{-n}$.
Secondly, if $|F(\rho(\sigma \ast 00\dots)) - F(\rho(\sigma \ast g))| \leq 2^{-(n+1)}$ holds for all ternary $g^{1}$, then this formula holds for all branches $g$ in $S_{\sigma,n}$.

\smallskip

Now, let $g_{\sigma, n}$ be a branch in $S_{\sigma,n}$ provided by sequential $\WKL$.  Then at least one of the following two items is the case:
\begin{itemize}
\item $|F(\rho(\sigma \ast 00\dots)) -  F(\rho(\sigma\ast g_{\sigma, n}))| < 2^{-n}$,
\item $|F(\rho(\sigma \ast 00\dots)) - F(\rho(\sigma \ast g_{\sigma, n}))] > 2^{-(n+1)}$.
\end{itemize}
The $n+4$-th rational approximation of $|F(\rho(\sigma \ast 00\dots)) -  F(\rho(\sigma\ast g_{\sigma, n}))| $ tells us which item holds.   
In case the first item holds, we put $B(\sigma,n)=0$, and $1$ otherwise. This function $B$ satisfies \eqref{pimpo} and we are done.
\end{proof}
The following remark discusses the representation used in the previous proof.  
\begin{rem}[Representations]\label{situa}\rm
Regarding the proof of Theorem \ref{gofusefl}, the use of sequences based on $\{-1,0,1\}$ and the map $\rho$ is known as the \emph{negative binary representation}. 
The set of such representations is a computable retract of the set of representations as given in Section \ref{kkk}; this representation is useful for representing $[0,1]$ over a compact space, or $\R$ over a $\sigma$-compact space, as in e.g.\ the proof of item \eqref{W8} of Theorem~\ref{tank}.  
\end{rem}
As an exercise, the reader can verify that a continuous \emph{increasing} function on $[0,1]$ has a modulus of continuity in $\RCAo$.  
The following corollary is useful.
\begin{cor}[$\RCAo+\WKL$]\label{fliep}
For a sequence $(F_{n})_{n\in \N}$ of continuous $[0,1]\di \R$-functions, there is a sequence of RM-codes $(\Phi_{n})_{n\in \N}$ such that $F_{n}(x)$ equals $\Phi_{n}(x)$ for all $x\in [0,1]$ and $n\in \N$.
\end{cor}
\begin{proof}
One readily defines an RM-code from a modulus of uniform continuity for a $[0,1]\di \R$-function.  
The principle $\WKL$ is equivalent to the `sequential' version of $\WKL$, i.e.\ that for a sequence of infinite 0/1-trees, there is a sequence of paths through the respective trees (\cite{kooltje}*{Prop.\ 3.1}). 
The latter readily yields the required sequence of RM-codes, via a sequence of moduli of uniform continuity.  
\end{proof}
Thirdly, a continuous function on $[0,1]$ has a modulus of uniform continuity by Theorem \ref{gofusefl} but the proof is rather involved.  
As it happens, the proof that continuous functions are bounded is easier, and (mostly) suffices for the development of higher-order RM.
\begin{thm}[$\RCAo+\WKL$]\label{liguster}
A continuous $F: [0,1]\di \R$ is bounded.
\end{thm}
\begin{proof}
For $F$ as in the theorem, define $G:2^{\N}\di \N$ by 
\be\label{tehsup}
G(f):=\lceil [F(\r(f))](2) \rceil+1, 
\ee
where $\r(f):=\sum_{n=0}^{\infty}\frac{f(n)}{2^{n+1}}$ can also be found in Definition \ref{keepintireal}.
We now split the proof in two cases. 
First of all, if $G$ as in \eqref{tehsup} is discontinuous on $2^{\N}$, we obtain $(\exists^{2})$ as the latter is equivalent to the existence of a discontinuous function on $\N^{\N}$ by \cite{kohlenbach2}*{Prop.\ 3.7}.
Suppose $F$ is unbounded on $[0,1]$; by the continuity of the former, we have $(\forall n\in \N)(\exists q\in \Q\cap [0,1])(|F(q)|>_{\R} n)$.  
Applying $\QFAC^{0,0}$, we obtain a sequence $(q_{n})_{n\in \N}$ such that $(\forall n\in \N)(q_{n}\in [0,1]\wedge |F(q_{n})|> n)$.  
Since $(\exists^{2})\di \ACA_{0}$, we may use the well-known (second-order) convergence theorems by \cite{simpson2}*{III.2.7}.
Thus, $(q_{n})_{n\in \N}$ has a sub-sequence with limit $y\in [0,1]$.  Clearly, $F$ is discontinuous at $y$, a contradiction.  
Hence, $F$ is bounded on $[0,1]$. 

\smallskip

Secondly, if $G$ as in \eqref{tehsup} is continuous on $2^{\N}$, then it has a modulus of uniform continuity by \cite{kohlenbach4}*{Prop.\ 4.11}.  
Hence, $G$ is bounded on $2^{\N}$, implying that $F$ is also bounded on $[0,1]$; the latter follows by contradiction and the fact that individual real numbers have binary representations in $\RCA_{0}$ (see \cite{polahirst}).
\end{proof}
%The theorem has a number of interesting corollaries.
Fourth, we recall that $\WKL$ is equivalent to the statement \emph{for a code of a uniformly continuous function, there is a modulus of uniform continuity} (\cite{simpson2}*{IV.2.9}).
\begin{thm}[$\RCAo+\WKL$]\label{liguster2}
A uniformly continuous $F: [0,1]\di \R$ has a modulus of uniform continuity.
\end{thm}
\begin{proof}
Let $F:[0,1]\di \R$ be uniformly continuous. 
In particular, we have
\be\label{fruhal}\textstyle
(\forall k\in \N)(\exists g\in 2^{\N})\underline{(\forall x, y\in [0,1]\cap \Q)(|x-y|< \r(g)\di |F(x)-F(y)|\leq\frac{1}{2^{k}})},
\ee
where $\r(f):=\sum_{n=0}^{\infty}\frac{f(n)}{2^{n+1}}$ is a real number in $[0,1]$. 
As noted in \cite{simpson2}*{Table 4, Notes}, $\WKL$ is equivalent to $\Pi_{1}^{0}$-$\textsf{AC}_{0}$, where the latter is:
\[
(\forall n \in \N)(\exists X\subset \N)\varphi(n, X) \di ( \exists (Z_{n})_{n\in \N})(\forall n \in \N)[\varphi(n, Z_{n}) \wedge Z_{n}\subseteq\N],
\]
for any $\varphi\in \Pi_{1}^{0}$.
The underlined formula in \eqref{fruhal} is $\Pi_{1}^{0}$, as $\lambda q.F((q, q, \dots))$ is merely a sequence of reals if $q$ is a variable over $\Q$.
Hence, apply $\Pi_{1}^{0}$-$\textsf{AC}_{0}$ to \eqref{fruhal} and note that the resulting function yields a modulus of uniform continuity. 
\end{proof}
Fifth, we obtain the following equivalences.
\begin{thm}[$\RCAo$]\label{ploppy} The following are equivalent to $\WKL$.
\begin{itemize}
\item A continuous $F: [0,1]\di \R$ is bounded.
\item A uniformly continuous $F: [0,1]\di \R$ has a modulus of uniform continuity.
\item A continuous $F:[0,1]\di \R$ is Riemann integrable \(\cite{simpson2}*{IV.2.7}\).
\end{itemize}
\end{thm}
\begin{proof}
That $\WKL$ implies the first two items from the theorem, follows from Theorems \ref{liguster} and \ref{liguster2}.
To obtain the third item from $\WKL$, use Corollary \ref{fliep}, combined with the second-order results for Riemann integration (\cite{simpson2}*{IV.2.5}).
To show that the first item implies $\WKL$, note that an RM-code $\Phi$ for a continuous function on $[0,1]$ yields a (third-order) continuous function $F:[0,1]\di \R$ by Theorem \ref{plofkip}. 
Then $F$ is bounded and so is the function represented by $\Phi$.  We now obtain $\WKL$ via \cite{simpson2}*{IV.2.3}.  An analogous proof goes through for the second and third items.  Indeed, the latter for codes are equivalent to $\WKL$ by \cite{simpson2}*{IV.2.7 and IV.2.9}.
\end{proof}
In conclusion, we have adapted some of Kohlenbach's `coding results' from \cite{kohlenbach4}*{\S4}, namely from $2^{\N}$ to $[0,1]$.  
We have presented a fairly constructive but lengthy proof (Theorem \ref{gofusefl}).   We have also obtained shorter but less constructive proofs of similar results (Theorems \ref{liguster} and \ref{liguster2}). 
Along the way, we have shown that $\WKL$ is equivalent to \emph{third-order} statements (see Theorem~\ref{ploppy}).  
The consensus view here seems to be that (third-order) continuous functions are `really' second-order, as evidenced by Corollary \ref{fliep}.  
In this way, equivalences like Theorem \ref{ploppy} do not really connect second- and third-order arithmetic.   The aim of the next section
is to exhibit `more real' connections, i.e.\ equivalences between $\WKL$ and third-order theorems that do not have an obvious second-order counterpart.

\subsection{Equivalences for weak K\"onig's lemma}\label{TOT}
We obtain equivalences between $\WKL$ and certain third-order statements in higher-order RM (Sections~\ref{XXY} and~\ref{hargeil}). 
In Section \ref{ferengi}, we sketch similar results for the RM of \emph{weak weak K\"onig's lemma} ($\WWKL$ for short) from \cite{simpson2}*{X.1}.

\subsubsection{Boundedness and supremum principles}\label{XXY}
We establish our first series of third-order statements equivalent to $\WKL$ (Theorem~\ref{tank}), the former being boundedness and supremum principles from analysis. 
We note in passing that the textbook proof that $BV$-functions are bounded on $[0,1]$ (see e.g.\ \cite{voordedorst}) goes through in $\RCAo$.

\smallskip

While some of the theorems under study are basic, others like item \eqref{W43} seem advanced as the class of \emph{quasi-continuous} functions goes \emph{far} beyond even the Borel or measurable functions, as discussed in Remark \ref{donola}.  Quasi-continuity goes back to Baire (\cite{beren2}) and is used in domain theory (\cites{lawsonquasi1, cauza,gieren3, gieren2}).

\smallskip

Regarding item \eqref{W6}, the assumption $F(x)=\frac{F(x-)+F(x+)}{2}$ and variations is found in e.g.\ \cite{voordedorst,waterdragen, gofer, gofer2}.  
Regarding item~\eqref{W423}, cadlag functions are an important class in stochastics and econometrics while Remark~\ref{donola} explains why items \eqref{WF}-\eqref{WH}, \eqref{WG}-\eqref{WJ}, \eqref{W27}, \eqref{W272}, \eqref{W72}, and \eqref{W722} are non-trivial.  
Regarding item~\eqref{W82}, Darboux sub-classes are topics of study in their own right (see e.g.\ \cite{DB1, DB2, malin, malin2}).  
The fragment of countable choice $\QFAC^{0,1}$ is defined in Section~\ref{rmbt} while the exact role of the Axiom of Choice is discussed in Remark~\ref{PINX}.
\begin{thm}[$\RCAo+\QFAC^{0,1}$]\label{tank} 
The following are equivalent to $\WKL$.
\begin{enumerate}
\renewcommand{\theenumi}{\roman{enumi}}
\item A regulated $F: [0,1]\di \R$ is bounded.\label{W1}
\item A regulated and continuous almost everywhere $F: [0,1]\di \R$ is bounded.\label{WF}
\item A regulated and pointwise discontinuous $F: [0,1]\di \R$ is bounded.\label{WI}
\item A regulated and not everywhere discontinuous $F: [0,1]\di \R$ is bounded.\label{WH}
%\item Any symmetrically continuous $F:[0,1]\di \R$ is bounded.\label{W1337}
%\item A regulated Darboux $F: [0,1]\di \R$ is bounded.\label{W12}
\item Any usco $F:[0,1]\di \R$ is bounded above.\label{W2}
\item Any usco and not everywhere discontinuous $F:[0,1]\di \R$  is bounded above.\label{WG}
\item Any usco and pointwise discontinuous $F:[0,1]\di \R$  is bounded above.\label{WJ}
\item Any lsco $F:[0,1]\di \R$ is bounded below.\label{W22}
\item Any usco and Baire 1 function $F:[0,1]\di \R$ is bounded above.\label{W27}
\item Any usco and effectively Baire $n$ $F:[0,1]\di \R$ is bounded above \($n\geq 2$\).\label{W272}
\item A regulated and usco $F:[0,1]\di \R$ is bounded.\label{W4}
\item A regulated and quasi-continuous $F:[0,1]\di \R$ is bounded.\label{W42}
\item A bounded usco function on $[0,1]$ that has a supremum, attains it. \label{W3}
\item A bounded regulated usco function on $[0,1]$ with a supremum, attains it.\label{W5}
\item A bounded lsco function on $[0,1]$ that has an infimum, attains it. \label{W32}
\item A regulated $F: [0,1]\di \R$ such that $F(x)=\frac{F(x-)+F(x+)}{2}$ for all $x\in (0,1)$, is bounded.\label{W6}
\item A regulated and lsco $F: [0,1]\di \R$ is bounded.\label{W7}
\item A regulated and Baire 1 function $F: [0,1]\di \R$ is bounded.\label{W72}
\item A regulated effectively Baire $n$ function $F: [0,1]\di \R$ is bounded \($n\geq 2$\).\label{W722}
\item A bounded and quasi-continuous $F:[0,1]\di \R$ has a sup \(and inf\).\label{W43}
\item A cadlag function $F:[0,1]\di \R$ is bounded \(\cite{bollard}*{Problem IV.3}\).\label{W423}
\item A cadlag function $F:[0,1]\di \R$ has a sup \(and inf\).\label{W424}
\item A bounded Baire 1 function $F:[0,1]\di \R$ has a supremum.\label{W8}
\item A bounded Darboux Baire 1 function $F:[0,1]\di \R$ has a supremum.\label{W82}
%\item A bounded Baire 2 function $F:[0,1]\di \R$ has a supremum.\label{W9}
\end{enumerate}
We do not use $\QFAC^{0,1}$ in relation to items \eqref{W42}, \eqref{W6}, \eqref{W7}, and \eqref{W43}-\eqref{W82}.  %do not need $\QFAC^{0,1}$.\\
\end{thm}
\begin{proof}
First of all, item \eqref{W1} readily implies $\WKL$ as $F(x+)=F(x)=F(x-)$ for all $x\in (0,1)$ in case $F$ is continuous; Theorem \ref{ploppy} now yields $\WKL$.
For the reversal, assume $\WKL$ and fix some regulated $F:[0,1]\di \R$.  In case the latter is continuous, it is also bounded by Theorem \ref{ploppy}.
In case $F$ is discontinuous, we have access to $(\exists^{2})$ by \cite{kohlenbach2}*{\S3}.  Now suppose $(\forall n\in \N)( \exists x\in [0,1])( |F(x)|>n)$ and apply $\QFAC^{0,1}$ to obtain $(x_{n})_{n\in \N}$ such that $|F(x_n)|>n$ for all $n\in \N$.  Use $\mu^{2}$ to guarantee $F(x_{n+1})>\max (n+1, F(x_n))$ for all $n\in \N$, if necessary.  
Since $(\exists^{2})\di \ACA_{0}$, we have access to the well-known second-order convergence theorems (see \cite{simpson2}*{III.2}).  
Thus, there is a convergent sub-sequence $(y_{n})_{n\in \N}$ of $(x_{n})_{n\in \N}$, say with limit $y\in [0,1]$.  
Then either there are infinitely many $n\in \N$ such that $y_{n}<y$ or infinitely many $m\in \N$ such that $y_{m}>y$; note that this case distinction is decidable using $\exists^{2}$.  
In the former case (the latter being symmetric), $F(y_{n})$ becomes arbitrarily large as $n\di \infty$.  In particular, $F(y-)$ does not exist, a contradiction, and $F$ must be bounded on $[0,1]$, and item \eqref{W1} follows. 

\smallskip

Secondly, the equivalence for item \eqref{W2} (and items \eqref{WF}-\eqref{WH}, \eqref{WG}-\eqref{WJ}, \eqref{W22}-\eqref{W4}, and \eqref{W72}-\eqref{W722}) follows in the same way as for item \eqref{W1}.  Indeed, item \eqref{W2} for instance implies $\WKL$ since a continuous function is trivially usco (and lsco, Baire 1, cadlag, or effectively Baire $n$), while $\WKL$ already follows in \cite{simpson2}*{IV.2.3} from the existence of an upper bound.  
For the reversal, one proceeds as in the previous paragraph, noting that $F$ cannot be usco at the limit point $y\in [0,1]$.   The equivalence involving items \eqref{W22}-\eqref{W4} and \eqref{W72}-\eqref{W722} is now immediate.

\smallskip

Thirdly, item \eqref{W3} readily implies $\WKL$ as a continuous function is trivially usco, i.e.\ combining Theorem \ref{plofkip} and \cite{simpson2}*{IV.2.3} yields $\WKL$.
For the reversal, assume $\WKL$ and fix an usco function $f:[0,1]\di [0,1]$ that has a supremum $y_{0}\in [0,1]$.  In case $f$ is continuous, $\WKL$ yields an RM-code (Corollary \ref{fliep}). 
Hence, the well-known second-order result in \cite{simpson2}*{IV.2.3} yields the required maximum.  In case $F$ is discontinuous, we have access to $(\exists^{2})$ by \cite{kohlenbach2}*{\S3}.  
By definition, we have $(\forall n\in \N)(\exists x\in [0,1])(f(x)\geq y_{0}-\frac{1}{2^{n}} )$.   
Apply $\QFAC^{0,1}$ to obtain a sequence $(x_{n})_{n\in \N}$ such that $(\forall n\in \N)(f(x_{n})\geq y_{0}-\frac{1}{2^{n}} )$.
Since $(\exists^{2})\di \ACA_{0}$, we have access to the well-known second-order convergence theorems (see \cite{simpson2}*{III.2}).  
Let $(z_{n})_{n\in \N}$ be a convergent sub-sequence of $(x_{n})_{n\in \N}$, say with limit $w_{0}$.  By assumption ($f$ being usco and $y_{0}$ its supremum), we have
\[\textstyle
y_{0} \ge f(w_0) \ge \limsup_{x \to w_0} f(x) \ge \limsup_{n \to \infty} f(z_{n}) \ge \lim_{n\di \infty} y_{0} - \frac 1{2^{n}} = y_{0}, 
\]
which implies $f(w_{0})=y_{0}$ as required for item \eqref{W3}.  The equivalence involving items \eqref{W5} and \eqref{W32} is now immediate.  

\smallskip

Fourth, for items \eqref{W42}, \eqref{W6}, \eqref{W7}, and \eqref{W423}, the equivalence is proved as for items \eqref{W1} and \eqref{W2} with the only modification that $(\forall n\in \N)( \exists x\in [0,1])( |F(x)|>n)$ implies $(\forall n\in \N)( \exists q\in [0,1]\cap \Q)( |F(q)|>n)$ due to the extra conditions in these items.  Hence, we can apply $\QFAC^{0,0}$ (rather than $\QFAC^{0,1}$), included in $\RCAo$.

\smallskip

Fifth, for item \eqref{W43}, the latter yields $\WKL$ by \cite{simpson2}*{IV.2.3}; indeed, Theorem \ref{plofkip} converts an RM-code for a continuous function into a third-order continuous function, which is trivially quasi-continuous.  Now assume $\WKL$ and let $F:[0,1]\di \R$ be quasi-continuous and bounded.  In case the latter is also continuous, Theorem~\ref{gofusefl} provided a modulus of uniform continuity and \cite{simpson2}*{IV.2.3} provides the required supremum.  In case $F$ is discontinuous, we obtain $(\exists^{2})$ by \cite{kohlenbach2}*{\S3}.  The usual interval-halving technique (using $\exists^{2}$) then readily yields the required supremum as 
\be\label{deru}
(\exists x\in [0,1])(F(x)>r)\asa (\exists q\in [0,1]\cap \Q)(F(q)>r),
\ee
for any $r\in \R$, as cadlag implies quasi-continuity.  An analogous proof goes through for item \eqref{W424}, as cadlag functions are quasi-continuous. 

\smallskip

For item \eqref{W8}, let $F:[0,1]\di \R$ be the pointwise limit of $(F_{n})_{n\in \N}$, where each $F_n:[0,1]\di \R$ is continuous on $[0,1]$. Let $\Phi_n$ be an RM-code for $F_n$ and $M_n$ be the modulus of uniform continuity for $F_n$, all provided by Corollary \ref{fliep} (and Remark~\ref{LEM2}). We may assume that $F$ is not continuous, whence we have access to $\exists^2$ by \cite{kohlenbach2}*{\S3}. We now show that for $r \in \Q$, $\sup_{x\in [0,1]} F(x) > r$ is definable in $\exists^2$; the proof is based on the equivalence between items (A) and (B) below. 

\smallskip

Now, $\exists^{2}$ can (uniformly) convert between various representations of real numbers (see \cite{polahirst} for the latter).  Thus,  we may assume that any $x\in [0,1]$, which actually is a fast converging sequence of rational numbers (see Definition \ref{keepintireal}), is obtained from a negative binary representation $f_{x}$ as in Remark \ref{situa}.  Note that there is a bijective correspondence between the negative binary representations and the fast converging sequences obtained from them.  

\smallskip

We let $I(x, k)$ be the interval of reals $y\in [0,1]$ represented by a negative binary representation extending that of $\overline{f_{x}}k $. 
We assume $\Phi_n$ to be given as a set of pairs of intervals $\langle[a,b],[c,d]\rangle$ with rational endpoints such that, in addition to an approximation requirement,  if $y \in [a,b]$ then $F_n(y) \in [c,d]$; this is the most frequently used domain representation.  
Using $\exists^2$, the latter representation is equivalent to any other (RM-)representation.

\smallskip

\noindent We now show that the following are equivalent:
\begin{enumerate}
\item[(A)] $\sup_{x\in [0,1]} F(x) > r$,
\item[(B)] There exists $x\in [0,1]$, $n, k\in \N$ such that for $m \geq n$ and $\langle[a,b],[c,d]\rangle \in \Phi_m$, if $j = M_m(k+1)$ and $I(x, j) \cap [a,b] \neq \emptyset$, then $[c,d]$ contains an element $ \geq r + 2^{-(k+1)} $.
\end{enumerate} 
To show that (A) $\rightarrow$ (B), assume for some $x\in [0,1]$ that $F(x) > r$ and let $n,k\in \N$ be such that $F_m(x) > r - 2^{-k}$ for all $m \geq n$. 
We now verify (B) for this choice of $n$, $k$ and $x$. 
Let $m \geq n$ and $\langle [a,b],[c,d]\rangle \in \Phi_m$ with $j = M_m(k+1)$ be such that $I(x, j) \cap [a,b] \neq \emptyset$. Let $y\in [0,1]$ be in this intersection, implying $|x-y| < \frac{1}{2^{j}}$ and hence $|F_m(x) - F_m(y)| < 2^{-(k+1)}$. 
Since $F_m(y) \in [c,d]$ and $F_m(y) > r + 2^{-(k+1)}$ by the triangle inequality, (B) follows for the aforementioned choice of $n, k, x$.

\smallskip

%\noindent $\ast\ast \rightarrow \ast$:
To show that (B) $\di $ (A), let $x$, $n$ and $k$ be as stated in (B). 
We show that $F(x) > r$ by showing that $F_m(x) \geq r + 2^{-(k+2)}$ for all $m \geq n$. Let $\langle [a,b],[c,d]\rangle \in \Phi_m$ be such that $x \in [a,b]$ and $|d-c| < 2^{-(k+2)}$. 
Clearly $[a,b] \cap I(x, k) \neq \emptyset$, since e.g.\ $x$ is in both sets.  
Since $F_m(x) \in [c,d]$ and $[c,d]$ contains an element $\geq r + 2^{-(k+1)}$, we must have that $F_m(x) \geq r + 2^{-( k+ 2)}$.

\smallskip

\noindent
In (B), we first have existential quantifiers for $x$, $n$ and $k$, then universal quantifiers over $\N$ and $\Q^4$ and the remaining matrix is decidable in the parameters. As `$(\exists x \in [0,1])$' actually is a quantifier over $\{-1,0,1\}^\N$ and the latter is computably identifiable with $2^\N$, the equivalence (A) $\asa$ (B) shows that $\sup_{x\in [0,1]} F(x) > r$ can be expressed as a second-order formula of the form
\[
(\exists n\in \N) (\exists f \in 2^\N)( \forall m\in \N) R(r,n,f,m),
\] 
where $R$ is Turing computable in the second-order objects $(\Phi_k)_{k \in \N}$ and $(M_k)_{k \in \N}$. By $\WKL$, the formula $(\exists f \in 2^\N)( \forall m\in \N) R(r,n,f,m)$ is equivalent to:
\[
(\forall k\in \N)(\exists \sigma \in 2^{<\N})\big[|\sigma|=k \wedge (\forall m\in \N)R(r,n,\sigma*00\dots,m)    \big],
\]
which is arithmetical.
This shows the existence of $\sup_{x\in [0,1]} F(x)$ for $F:[0,1]\di \R$ in Baire 1, given as the limit of a sequence of continuous functions, assuming $\exists^2$. 
The previous goes through for item \eqref{W82} since continuous functions on $[0,1]$ are Darboux, which follows by imitating the second-order intermediate value theorem as can be found in \cite{simpson2}*{II.6.6}.
\end{proof}
In light of Theorem \ref{tank}, a single second-order equivalence from analysis can give rise to \emph{many different} equivalences in higher-order RM.  
There is however a limit: while the supremum principle for effectively Baire~2 functions is equivalent to the Big Five system $\FIVE$ (see Theorem \ref{flapke}), the former principle for Baire~1$^{*}$ or Baire~2 functions is not provable in $\Z_{2}^{\omega}$ by Theorem \ref{truppke}.  Nonetheless, the third-order RM of $\WKL$ can only be called \emph{extremely robust} following Section \ref{sintro}.
  
\smallskip

In the next remark, we discuss the role of the Axiom of Choice in Theorem \ref{tank}.  
\begin{rem}[On the Axiom of Choice]\label{PINX}\rm
The Axiom of Choice ($\AC$ for short) plays an interesting role in Theorem \ref{tank}, namely related to the results in \cites{dagsamV, dagsamVII}.  
As in the latter, we say that a statement $T$ in the language of finite types \emph{exhibits the Pincherle phenomenon} if the following two items are satisfied.
\begin{itemize}
\item The statement $T$ is provable \emph{without} $\AC$ but only in relatively strong systems, namely $T$ is provable in $\Z_{2}^{\Omega}$, but not in $\Z_{2}^{\omega}$ (see Section~\ref{lll}). 
\item The statement $T$ is provable in weak systems \emph{assuming} (fragments of) $\AC$, namely the system $\RCAo+\WKL+\QFAC^{0,1}$ proves $T$.
\end{itemize}
In short, the Pincherle phenomenon is the observation that $\AC$ makes certain theorems `easier to prove' even though we do not strictly \emph{need} $\AC$.   
We first observed this phenomenon in \cite{dagsamV} for a theorem by Salvatore Pincherle from \cite{tepelpinch}*{p.\ 67}, while \emph{many} examples may be found in \cite{dagsamVII} and elsewhere.

\smallskip

One readily verifies that the Pincherle phenomenon is exhibited by items \eqref{W1}, \eqref{W2}, \eqref{W22}, \eqref{W4}, \eqref{W3}, and \eqref{W5} from Theorem~\ref{tank}.
Indeed, the model $\bf Q^{*}$ of $\Z_{2}^{\omega}$ from \cite{dagsamX} is such that there is an unbounded regulated function on $[0,1]$, i.e.\ item \eqref{W1} is not provable in $\Z_{2}^{\omega}$ and the same for the other items.  More interestingly, items \eqref{W27}, \eqref{W272}, \eqref{W72}, and \eqref{W722} exhibit a kind of \emph{weak} Pincherle phenomenon as these items are already\footnote{The use of $\QFAC^{0,1}$  can be replaced by $\Sigma_{1}^{1}\textsf{-AC}_{0}$ in light of the equivalence (A)$\asa$(B) from the proof of Theorem \ref{tank}.} provable in $\RCAo+\Sigma_{1}^{1}\textsf{-AC}_{0}$, which is conservative over $\ACA_{0}$.  Thus, the extra `Baire 1' condition in these items makes them `easier to prove', where we note that the addition of this condition is non-trivial by Remark \ref{donola}.
\end{rem}
Next, we have an important corollary to Theorem \ref{tank}, where a \emph{set} is `Baire~$1$' if the characteristic function is Baire 1.  
The general notion of Baire set may be found in \cite{kodt}*{p.\ 21, Def.\ 4} under a different name; we refer to \cite{lorch} for an introduction and to \cite{dudley}*{\S7} for equivalent definitions, including that of Borel set in Euclidean space.   
\begin{thm}[$\RCAo+\WKL$]\label{klank}
For any open Baire 1 set $O\subset [0,1]$, there exist $(a_{n})_{n\in \N}, (b_{n})_{n\in \N}$ such that $x\in O\asa (\exists n\in \N)(x\in (a_{n}, b_{n}) )$ for all $x\in [0,1]$.
\end{thm}
\begin{proof}
We make use of item \eqref{W8} of Theorem \ref{tank}.  In particular, the proof of this item immediately generalises to infima involving rational parameters, i.e.\ we have:
\begin{center}
\emph{for a bounded Baire 1 function $f:[0,1]\di \R$, there is $F:\Q^{2}\di \R$ such that for all $p, q\in \Q\cap[0,1]$, the real $F(p, q)$ equals $\inf_{x\in [p, q]}f(x)$.}
\end{center}
To see this, we observe that when $p < q$ for $p, q\in \Q$, the formula $x \in [p,q]$ can be expressed by a $\Pi^0_1$-formula in the negative binary representation $f_x$ from Remark \ref{situa}.
Now let $O$ be an open Baire 1 set and note that the representation is trivial in case $O=\emptyset$.  Hence, we may assume there is $x_{0}\in O$ and $m_{0}$ such that $B(x_{0}, \frac{1}{2^{m_{0}}})\subset O$. 
Let $\big((p_{n}, q_{n})\big)_{n\in \N}$ be an enumeration of all intervals in $[0,1]$ with rational end-points.  
Now define the following sequence of intervals:
\be\label{tugger}
 (a_{n}, b_{n}):=
 \begin{cases}
 B(x_{0}, \frac{1}{2^{m_{0}}}) & \textup{ in case } \inf_{x\in [p_{n}, q_{n}]}\mathbb{1}_{O}(x)<\frac{1}{2}\\
(p_{n}, q_{n})& \textup{ in case } \inf_{x\in [p_{n}, q_{n}]}\mathbb{1}_{O}(x)>0
 \end{cases}.
\ee
Note that the case distinction in \eqref{tugger} is decidable (in $\RCAo$) and that in each case $(a_{n}, b_{n})\subset O$. 
The theorem is now immediate.  
\end{proof}
%NEWZ
Theorem \ref{klank} essentially expresses that a Baire 1 open set can be represented by a code for an open set (see \cite{simpson2}*{II.5.6}).
The general case for arbitrary open sets is not provable in $\Z_{2}^{\omega}$ from Section \ref{lll} (see \cite{dagsamVII}).
Nonetheless, assuming $\WKL$, any applicable second-order theorem generalises from `codes for open sets' to `third-order open sets that are Baire 1'.
Examples include the Heine-Borel theorem for countable coverings of closed sets (\cite{brownphd}*{Lemma 3.13} and \cite{simpson2}*{IV.1.5}), the Tietze extension theorem (\cite{withgusto}), the Urysohn lemma (\cite{simpson2}*{II.7.3}), and the Baire category theorem (\cite{simpson2}*{II.4.10}).
The same holds \emph{mutatis mutandis} for open sets with quasi-continuous characteristic functions or open sets as in Definition \ref{openset} where $Y:\R\di \R$ is Baire 1.  
The following theorem establishes a similar theorem for countable sets; enumerating 
general (strongly) countable sets cannot be done in $\Z_{2}^{\omega}$ (see \cite{dagsamXI, dagsamXII}).   
\begin{thm}[$\ACAo$]\label{weerklank}
For any Baire 1 set $A\subset [0,1]$ and Baire 1 function $Y:[0,1]\di \N$ injective on $A$, there is $(x_{n})_{n\in \N}$ which includes all elements of $A$.
\end{thm}
\begin{proof}
Let $A\subset [0,1]$ and $Y:[0,1]\di \N$ be as in the theorem.  
The standard proof shows that the product of Baire 1 functions is Baire 1.  
Thus, Theorem \ref{tank} guarantees that $\inf_{[p,q]}(\mathbb{1}_{A}(x) Y(x) )$ makes sense for rational $p, q\in [0,1]$.
Assuming $A\ne \emptyset$, we have $\inf_{[0,1]}(\mathbb{1}_{A}(x) Y(x) )=n_{0}>0$.  
Now replace $[0,1]$ by $[0, \frac12]$ and $[\frac12, 1]$ to check in which of the latter sub-intervals the unique $x_{0}\in A$ with $Y(x_{0})=n_{0}$ is to be found.  
The usual interval-halving technique now provides this real and repeat for $\inf_{[0,1]}(\mathbb{1}_{A\setminus \{x_{0}\}}(x) Y(x) )=n_{1}$ to enumerate $A$.  
\end{proof}
Finally, we finish this section with some conceptual remarks.
\begin{rem}\label{donola}\rm
First of all, it is a basic fact that $BV$, usco, lsco, and regulated functions are Baire 1, but this \textbf{cannot} be proved from the Big Five or $(\exists^{2})$, and much stronger systems by Theorem \ref{reklam}.  
To be absolutely clear, $\ACAo+\FIVE$ and much stronger systems are consistent with the existence of $BV$, usco, lsco, and regulated functions that are not Baire 1, explaining e.g.\ items \eqref{W27} and~\eqref{W72} in Theorem~\ref{tank}.  Similar results hold for `effectively Baire $n$', explaining for instance item \eqref{W722} in Theorem~\ref{tank}.

\smallskip

Secondly, the cadlag functions are arguably `close to continuous', but one should be careful with such claims: 
by Theorem \ref{drel}, it is consistent with the Big Five and much stronger systems that there is a regulated (or usco) function that is 
discontinuous everywhere (see also \cite{samBIG}), explaining items \eqref{WF}-\eqref{WH} and \eqref{WG}-\eqref{WJ} in Theorem~\ref{tank}.
Furthermore, quasi-continuous functions can be quite `wild': if $\mathfrak{c}$ is the cardinality of $\R$, there are $2^{\mathfrak{c}}$ non-measurable quasi-continuous $[0,1]\di \R$-functions and $2^{\mathfrak{c}}$ measurable  quasi-continuous $[0,1]\di [0,1]$-functions (see \cite{holaseg}).  Also, the class of quasi-continuous functions is closed under taking \emph{transfinite} limits (\cite{nieuwebronna}).

\smallskip

Thirdly, the regulated functions boast \emph{many} sub-spaces (see \cite{voordedorst}) and the same for the Baire 1 functions (see \cite{vuilekech}); one can presumably formulate a version of items \eqref{W1} or \eqref{W72} for many of those.  In general, many variations of Theorem~\ref{tank} are possible, based on any function space containing the continuous functions.  
For instance, one can replace `quasi-continuity' by the weaker notion `lower quasi-continuity' (see \cite{ewert}) in most equivalences in this paper.  
Moreover, derivatives that are continuous almost everywhere, are quasi-continuous (\cite{marcus2}), suggesting many possible variations.  
The notion of \emph{strong} quasi-continuity also seems very promising, especially in light of its intimate connection to continuity almost everywhere (\cite{grand}), as well as the notion of \emph{countably continuous} and related concepts (\cite{thommy}). 
%Since the latter notion is not as famous as the former, we do not go into details.  

\smallskip

Fourth, we have studied the (countable) Axiom of Choice in higher-order RM in \cite{dagsamX}.
We believe that the choice principle $\NCC$ from the latter, which is provable in $\ZF$, can replace $\QFAC^{0,1}$ in Theorem \ref{tank} and the below. 
\end{rem}

\subsubsection{Covering lemmas}\label{hargeil}
We show that $\WKL$ is equivalent to a number of third-order covering lemmas (Theorem \ref{lebber}).
We have shown in \cite{dagsamIII} that the general case, called \emph{Cousin's lemma}, is not provable from $\WKL$ and much stronger systems. 

\smallskip

First of all, $\WKL$ is equivalent to compactness results like the Heine-Borel theorem for countable coverings (\cite{simpson2}*{IV.1}) and Cousin's lemma for (codes of) continuous functions (\cite{basket, basket2}). 
In general, Cousin's lemma (\cite{cousin1}) is formulated as follows.
\begin{center}
\emph{For $\Psi:[0,1]\di \R^{+}$, the covering $\cup_{x\in [0,1]}B(x, \Psi(x))$ of $[0,1]$ has a finite sub-covering, i.e.\ there are $x_{0}, \dots, x_{k}\in [0,1]$ where $\cup_{i\leq k}B(x_{i}, \Psi(x_{i}))$ covers $[0,1]$.}
\end{center}
Secondly, we establish the following theorem to be contrasted with items \eqref{ta0}-\eqref{ta21} in Theorem \ref{reklam}.  We stress that Cousin's lemma deals with uncountable coverings. 
Recall Remark \ref{donola} which explains why items \eqref{HBX}-\eqref{HBH} are non-trivial. 
\begin{thm}[$\RCAo$]\label{lebber}
The following are equivalent to $\WKL$.
\begin{enumerate}
\renewcommand{\theenumi}{\roman{enumi}}
\item Cousin's lemma for RM-codes of continuous functions.\label{HB-1}
\item Cousin's lemma for continuous functions.\label{HB0}
\item Cousin's lemma for lsco functions.\label{HB1}
\item Cousin's lemma for lsco Baire 1 functions.\label{HBX}
\item Cousin's lemma for lsco effectively Baire $n+2$ functions.\label{HBY}
\item Cousin's lemma for lsco functions that are continuous almost everywhere.\label{HBF}
\item Cousin's lemma for lsco functions that are pointwise discontinuous.\label{HBJ}
\item Cousin's lemma for lsco functions that are not everywhere discontinuous.\label{HBH}
\item Cousin's lemma for quasi-continuous functions.\label{HB2}
\item Cousin's lemma for cadlag functions.\label{HB3}
\item Cousin's lemma for regulated $F: [0,1]\di \R$ such that $F(x)=\frac{F(x-)+F(x+)}{2}$ for all $x\in [0,1]$.\label{HB4}
\item Cousin's lemma for regulated $F: [0,1]\di \R$ such that for all $x\in [0,1]$: \label{HB5}
\[
\min (F(x-), F(x+))\leq F(x)\leq \max (F(x-), F(x+)).
\]
\item Cousin's lemma for Baire 1 functions.\label{HB32}
%\item Cousin's lemma for Baire 2 functions???\label{HB321}
\end{enumerate}
\end{thm}
\begin{proof}
Zeroth of all, the equivalence between $\WKL$ and item \eqref{HB-1} has been proved in \cites{basket, basket2}.
The combination of Theorems \ref{plofkip} and \ref{gofusefl} then establishes the equivalence between $\WKL$ and item \eqref{HB0}.

\smallskip

First of all, assume item \eqref{HB1} and note that by Theorem \ref{plofkip}, an RM-code for a continuous function equals a (third-order) continuous function.  
The latter is trivially lsco and item \eqref{HB1} establishes Cousin's lemma for RM-codes of continuous functions, and hence $\WKL$.  % \cite{basket, basket2}). 
Now assume $\WKL$ and fix lsco $\Psi:[0,1]\di \R^{+}$.   If the latter is continuous, item \eqref{HB0} provides a finite sub-covering.  If the latter is discontinuous, we have access to $(\exists^{2})$ by \cite{kohlenbach2}*{\S3}.
In case $(\exists N\in \N)(\forall q\in [0,1]\cap \Q)(\Psi(q)\geq \frac1{2^{N}})$, item \eqref{HB1} is immediate as enough rational numbers form a finite sub-covering.   
Finally, in case $(\forall N\in \N)(\exists q\in [0,1]\cap \Q)(\Psi(q)<\frac1{2^{N}})$, apply $\QFAC^{0,0}$ included in the base theory to obtain a sequence of rationals $(q_{n})_{n\in \N}$ such that $\Psi(q_{n})<\frac{1}{2^{n}}$ for all $n\in \N$.  
Since $(\exists^{2})\di \ACA_{0}$, we have access to the usual second-order convergence theorems (see \cite{simpson2}*{III.2}), i.e.\ $(q_{n})_{n\in \N}$ has a convergent sub-sequence $(r_{n})_{n\in \N}$, say with limit $y\in [0,1]$.  
Then $\Psi(y)=0$ as $\Psi$ is lsco and $\Psi$ comes arbitrarily close to $0$ close to $y$, a contradiction, and we are done. 

\smallskip

Secondly, item \eqref{HB2} implies $\WKL$ in the same way as for item \eqref{HB1}.  Now assume $\WKL$ and fix quasi-continuous $\Psi:[0,1]\di \R^{+}$.
In case the latter is also continuous, Theorem~\ref{gofusefl} provides a modulus of uniform continuity and \cite{simpson2}*{IV.2.3} implies that $\inf_{x\in [0,1]}\Psi(x)>0$, which readily yields the finite sub-covering.  
In case $F$ is discontinuous, we obtain $(\exists^{2})$ by \cite{kohlenbach2}*{\S3}.  Now consider the following for any $r\in \R$:
\be\label{fki}
(\exists x\in [0,1])(\Psi(x)>r)\asa (\exists q\in [0,1]\cap \Q)(\Psi(q)>r),
\ee
which holds due to the definition of quasi-continuity.  In this light, use $\mu^{2}$ to define $G(x)$ as the least $N\in \N$ such that 
\be\label{tofro}\textstyle
(\exists q\in \Q\cap [0,1])(\frac{1}{2^{N}}< |x- (q+ \Psi(q))|\wedge x \in B(q, \Psi(q))  ).
\ee
Now define  an increasing sequence $(r_{n})_{n\in \N}$ of rationals via $r_{0}:= \frac{1}{2^{G(0)}}$ and $r_{n+1}:= r_{n}+\frac{1}{2^{G(r_{n})}}$.  
In case this sequence stays in $[0,1]$, it converges to some $y\in [0,1]$, which leads to a contradiction.  Hence, $r_{n_{0}}>1$ for some $n_{0}\in \N$, readily yielding a finite sub-covering. 
Since cadlag functions are (trivially) quasi-continuous, the equivalence for item \eqref{HB3} also follows.  Similarly, one readily observes that \eqref{fki} also holds for functions as in items \eqref{HB4} and \eqref{HB5}, i.e.\ the above proof for item \eqref{HB2} yields the equivalences involving items \eqref{HB4} and \eqref{HB5}.  An alternative proof of items \eqref{HB1} and \eqref{HB2} proceeds by noting that one can restrict $\cup_{x\in [0,1]}B(x, \Psi(x))$ to rationals for quasi-continuous or lsco $\Psi:[0,1]\di \R^{+}$, yielding a countable sub-covering $\cup_{q\in [0,1]\cap \Q}B(q, \Psi(q))$ of $[0,1]$ to which the second-order Heine-Borel theorem from \cite{simpson2}*{IV.1} applies; this is relevant to the proof of Corollary \ref{reklamcor}.

\smallskip

Thirdly, for item \eqref{HB32}, it suffices to prove the latter from $\WKL$ as continuous functions are trivially Baire 1.  To establish item \eqref{HB32}, in case $\Psi:[0,1]\di \R^{+}$ is continuous, use item \eqref{HB0}.  
In case $\Psi:[0,1]\di \R^{+}$ is discontinuous, we obtain $(\exists^{2})$ by \cite{kohlenbach2}*{\S3}, and hence $\ACA_{0}$.  
Let $(\Psi_{n})_{n\in \N}$ be a sequence of continuous functions with pointwise limit $\Psi$.  
Corollary \ref{fliep} converts this sequence into a sequence $(\Phi_{n})_{n\in \N}$ of codes for continuous functions.  However, the latter is a code for a Baire 1 function in the sense of \cite{basket, basket2}.
By the latter, Cousin's lemma for {codes for} Baire 1 functions, is also equivalent to $\ACA_{0}$, i.e.\ Cousin's lemma for $\Psi$ now follows via the second-order lemma for $(\Phi_{n})_{n\in \N}$ as $(\exists^{2})\di \ACA_{0}$.   
\end{proof}
Now, Cousin's lemma for \emph{codes for} Baire 1 functions, is equivalent to $\ACA_{0}$ (\cite{basket, basket2}).  
Moreover, by Theorem \ref{flame}, $(\exists^{2})$ is equivalent to the statement: \emph{a code for a Baire 1 function denotes a third-order function}.  
Hence, following item \eqref{HB32} of Theorem~\ref{lebber}, the use of second-order codes changes the logical strength of Cousin's lemma.  
We obtain sharper results on Cousin's lemma in Section \ref{vate}.

\smallskip

Finally, Theorem \ref{klank} expresses that open Baire 1 sets have RM-codes, assuming $\WKL$.  
Now, consider the following version of the (countable) Heine-Borel theorem:
\begin{center}
\emph{let $(O_{n})_{n\in \N}$ be a sequence of open Baire 1 sets, the union of which covers $[0,1]$.  Then there is $m\in \N$ such that $\cup_{n\leq m}O_{n}$ covers $[0,1]$.}
\end{center}
This is a direct generalisation of \cite{simpson2}*{V.1.5} and can be included in Theorem~\ref{lebber}.
The general case, i.e.\ with `Baire 1' omitted, exhibits the Pincherle phenomenon from Remark \ref{PINX}, as shown in \cite{dagsamVII}.

\subsubsection{More on covering lemmas}\label{ferengi}
We show that $\WWKL$ is equivalent to a number of third-order covering theorems (Theorem \ref{flebber}), where the former is \emph{weak weak K\"onig's lemma} as in \cite{simpson2}*{X.1.7}. 
We conjecture that $\RCAo+\WWKL$ cannot prove Theorem \ref{gofusefl}, i.e.\ we cannot use the associated coding results in this section.   

\smallskip

First of all, as suggested by its name, $\WWKL$ is a certain restriction of $\WKL$, namely to trees of positive measure.  Montalb\'an states in \cite{montahue} that $\WWKL$ is robust, i.e.\ equivalent to small perturbations of itself, in the same way as the Big Five are; $\WWKL_{0}$ is even called the `sixth Big system' in \cite{sayo}.   
Now, $\WWKL$ is equivalent to the \emph{Vitali covering theorem} for countable coverings, and to numerous variations (see \cite{simpson2}*{X.1} and \cite{sayo}*{Lemma 8}), including the following.
\begin{center}
\emph{Let $\big((a_{n},b_{n})\big)_{n\in \N}$ be a sequence of open intervals that covers $[0,1]$.  Then for any $\eps>0$, there is $m\in \N$ such that $\cup_{n\leq m}(a_{n}, b_{n})$ has measure $>1-\eps$.}
\end{center}
The following generalisation to uncountable coverings, in the spirit of Cousin's lemma, is not provable from the Big Five and much stronger systems (see \cite{dagsamVII}). 
\begin{center}
\emph{For $\Psi:[0,1]\di \R^{+}$ and $\eps>0$, there are $x_{0}, \dots, x_{k}\in [0,1]$ such that $\cup_{i\leq k}B(x_{i}, \Psi(x_{i}))$ has measure $1-\eps$.}
\end{center}
Vitali indeed considers uncountable coverings in \cite{vitaliorg}, going as far as expressing his surprise regarding the uncountable case.  
We shall refer to the second centred statement as \emph{Vitali's principle} as it constitutes the `combinatorial essence' of the Vitali covering theorem, in our opinion.   
The first centred statement will be called \emph{Vitali's principle for countable coverings}.

\smallskip

Secondly, we establish the following theorem to be contrasted with Corollary~\ref{reklamcor}.  
Recall Remark \ref{donola} which explains why items \eqref{VBX}-\eqref{VBH} are non-trivial. 
\begin{thm}[$\RCAo$]\label{flebber}
The following are equivalent to $\WWKL$.
\begin{enumerate}
\renewcommand{\theenumi}{\roman{enumi}}
\item Vitali's principle for RM-codes of continuous functions.\label{VB-1}
\item Vitali's principle for continuous functions.\label{VB0}
\item Vitali's principle for lsco functions.\label{VB1}
\item Vitali's principle for lsco Baire 1 functions.\label{VBX}
\item Vitali's principle for lsco effectively Baire $n+2$ functions.\label{VBY}
\item Vitali's principle for lsco functions that are continuous almost everywhere.\label{VBF}
\item Cousin's lemma for lsco functions that are pointwise discontinuous.\label{VBJ}
\item Vitali's principle for lsco functions that are not everywhere discontinuous.\label{VBH}
\item Vitali's principle for quasi-continuous functions.\label{VB2}
\item Vitali's principle for cadlag functions.\label{VB3}
\item Vitali's principle for regulated $F: [0,1]\di \R$ such that $F(x)=\frac{F(x-)+F(x+)}{2}$ for all $x\in [0,1]$.\label{VB4}
\item Vitali's principle for regulated $F: [0,1]\di \R$ such that for all $x\in [0,1]$:\label{VB5}
\[
\min (F(x-), F(x+))\leq F(x)\leq \max (F(x-), F(x+)).
\]
\item Vitali's principle for Baire 1 functions.\label{VB32}
%\item Cousin's lemma for Baire 2 functions???\label{HB321}
\end{enumerate}
\end{thm}
\begin{proof}
We establish the equivalences involving $\WWKL$ and items \eqref{VB-1} and \eqref{VB0}.  Invoking the law of excluded middle as in $(\exists^{2})\vee \neg(\exists^{2})$ then finishes the proof.  Indeed, in the former case, $(\exists^{2})\di \ACA_{0}$, which makes $\WWKL$ and all items outright provable in light of Theorem \ref{lebber}.  In case $\neg(\exists^{2})$, all functions on $\R$ are continuous (\cite{kohlenbach2}*{\S3}) and items \eqref{VB1}-\eqref{VB32} reduce to item \eqref{VB0}.

\smallskip

Assume $\WWKL$ and fix continuous $\Psi:[0,1] \di \R^+$.
To show that the countable union $\cup_{q \in [0,1] \cap \Q} B(q, \Psi(q))$ covers $[0,1]$, consider $x \in [0,1]$ and apply the 
definition of continuity of $\Psi$ for $k$ such that $\frac{1}{2^{k}} \leq \Psi(x)/2$, i.e. we obtain $N\in \N$ such that for $y \in B(x, \frac1{2^{N}})$, we have $|\Psi(x) - \Psi(y)|< \Psi(x)/2$.
Then for any $q \in B(x, \frac{1}{2^{N}})\cap \Q$ close enough to $x$, we have $x \in B(q, \Psi(q))$, as required.  
As noted above, $\WWKL$ is equivalent to Vitali's principle for countable coverings (\cite{simpson2}*{X.1}), i.e.\ we may apply the latter to $\cup_{q\in \Q\cap [0,1]}B(q, \Psi(q))$ to obtain item \eqref{VB0}.
For item \eqref{VB-1}, apply Theorem~\ref{plofkip} and use item \eqref{VB0}.

\smallskip

For the reversals, these essentially follow from the proof of \cite{basket2}*{Theorem 4.2}, which takes place in $\RCA_{0}$ and establishes that Cousin's lemma for (codes for) continuous functions, implies $\WKL$.    
In more detail, in the aforementioned proof, one fixes a countable covering $\cup_{n\in \N}(a_{n}, b_{n})$ of $[0,1]$ and defines a continuous function $\delta:[0,1]\di \R^{+}$.  
This function is then shown to have an RM-code and to satisfy:
\be\label{tonsk}\textstyle
(\forall x\in[0,1])\big[\delta(x)>\frac{1}{2^{k}} \di B(x, \delta(x))\subseteq  \cup_{m\leq k} (a_{m}, b_{m})\big].
\ee
In light of \eqref{tonsk}, a finite sub-covering for $\cup_{x\in [0,1]}B(x, \delta(x))$ immediately yields a finite sub-covering for $\cup_{n\in \N}(a_{n}, b_{n})$, i.e.\ Cousin's lemma for (codes for) continuous functions implies the Heine-Borel theorem for countable coverings, and hence $\WKL$ via \cite{simpson2}*{IV.1.1}.  Now, the definition of the RM-code of $\delta$ and the proof of \eqref{tonsk} take place in $\RCA_{0}$, i.e.\ we may simply apply item \eqref{VB-1} to $\cup_{x\in [0,1]}B(x, \delta(x))$ and obtain Vitali's principle for countable coverings, and hence $\WWKL$ via \cite{simpson2}*{X.1}.  
In light of Theorem \ref{plofkip}, item \eqref{VB0} also yield $\WWKL$.
\end{proof}
Finally, certain equivalences from Theorem \ref{flebber} can be proved (or expanded) using \cite{samcie21}*{Cor.\ 2.6}, 
where it is shown that $\WWKL$ is equivalent to Vitali's principle restricted to $\Psi:[0,1]\di\R^{+}$ that are continuous \emph{almost everywhere}. 
\subsection{Equivalences for arithmetical comprehension}\label{hazaha}
In this section, we establish some equivalences between arithmetical comprehension $\ACA_{0}$ and third-order theorems from analysis, including the \emph{Jordan decomposition theorem} as in Theorem~\ref{drd}.  
We have shown in \cite{dagsamXI} that the general case of the latter cannot be proved from the Big Five and much stronger systems like $\Z_{2}^{\omega}$.
Regarding definitions, the system $\ACAo$ is defined as $ \RCAo+(\exists^{2})$ and basic properties are in Section \ref{lll}.

\smallskip

First of all, we need the following theorem, where a \emph{jump discontinuity} of a function $f:\R\di \R$ is a real $x\in \R$ such that the left and right limits $f(x-)$ and $f(x+)$ exist, but are not equal. 
By \cite{dagsamXI}*{\S3.3}, listing \emph{all} points of discontinuity of $BV$-functions cannot be done in in the Big Five and much stronger systems.  
\begin{thm}[$\ACAo$]\label{falm}
If $f:[0,1]\di \R$ is regulated, there is a sequence of reals containing all jump discontinuities of $f$.
\end{thm}
\begin{proof}
Let $f:[0,1] \rightarrow\R$ be regulated. We say that $x \in (0,1)$ is a \emph{jump} if $\Gap(f,x) := |f(x+) - f(x-)| $ is $>0$, which is equivalent to the following:
\[\textstyle
(\exists k\in \N)(\forall m\in \N)(\exists q, r\in [0,1]\cap \Q)(q<x<r\wedge |q-r|<\frac{1}{2^{m}}\wedge |f(q)-f(r)|>\frac{1}{2^{k}}  ),
\]
where we note that $f$ only occurs with rational inputs. Hence, the set of jumps is arithmetically definable from $f$.

\smallskip

In the below, $p,q, a,b$ and $\delta > 0$ etc.\ are assumed to be variables over the rationals, while $x$ is a variable over the reals. 
First of all, we prove that items (1) and (2) as follows are equivalent. 
\begin{enumerate}
\item There is exactly one jump $x \in (a,b)$ with $\Gap(f,x) \geq \delta$.
\item
\begin{enumerate}
\item For all $n \in \N$  there are $p,q$ such that $a < p < q < b$, $q - p < 2^{-n}$, and $|f(p) - f(q)| > \delta-2^{-n}$.
\item There is an $n \in \N$  such that for all pairs $(p_1,q_1)$ and $(p_2,q_2)$ satisfying (a) with respect to $n$,  $\delta$, $a$, $b$ and $f$, we have that $(p_1,q_1) \cap (p_2,q_2) \neq \emptyset$.
\end{enumerate}
\end{enumerate}
Assume item (1) and note that (2).(a) follows from the fact that there is at least one point in $(a,b)$ with a jump $\geq \delta$. 
To prove (2).(b), we use the fact  there is only one $x \in (a,b)$ with a jump $\geq \delta$. To prove this uniqueness, we show that:
\begin{center}
\emph{there is $n\in \N$ such that $(p,q)$ satisfying \(2\).\(a\) contains an $x$ with jump $\geq \delta$.}
\end{center}
Suppose the centred claim is false. Then for each $k$, there will be $p_k< q_k$ such that $|f(p_k) -f(q_k)| > \delta- 2^{-k}$, $q_k -p_k < 2^{-k}$, and $x \not \in (p_k,q_k)$. 
Since we may pick a convergent sub-sequence (provable in $\ACAo$ by \cite{simpson2}*{III.2}), we can, without loss of  generality, assume that $(p_k)_{k \in \N}$ has a limit $y$. 
If $y$ is one of the objects $a$ or  $b$ , both $p_k$ and $q_q$ will approximate $y$ from the same side, violating the assumption of one-sided limits. 
If $y = x$, then for each $k$ both $p_k$ and $q_k$ will be on the same side of $x$, and then infinitely many will approach $x$ from the same side, again violating the assumption of one-sided limits. 
For any other value of $y$ we can either argue as above, or obtain that there is also a jump at $y$ of a size $\geq \delta$, contradicting the assumption of item (1).

\smallskip

Now assume (2) and note that by (2).(a) there cannot be more than one jump $x$ in $(a,b)$ with $\Gap(f,x) \geq \delta$. Now, let $n$ be as in (2).(b) and note that there is a pair $p_k < q_p$ of rational points for each $k \geq n$ satisfying (2).(b). We can find such $p_k,q_k$ via an effective search, and by (2).(b) we must have that both $(p_k)_{k \geq n}$ and $(q_k)_{k \geq n}$ converge, and to the same limit $x$. Clearly, using arguments as in the previous paragraph, $x$ is a jump in $(a,b)$ with $\Gap(f,x) \geq \delta$.

\smallskip

Having established the equivalence between (1) and (2), we see that the set of triples $(a,b,\delta)$ such that $(a,b)$ contains exactly one jump $x$ with $\Gap(f,x) \geq \delta$ is arithmetically definable from the restriction of $f$ to the rationals, and using the characterisation we see that the unique $x$ then is definable from the same restriction using $\exists^2$. In this way, we can enumerate the set of jumps.
\end{proof}
We can now generalise Corollary \ref{fliep} as follows.  
%Note that the underline formula in Corollary \ref{fromein} guarantees that $y, z$ are `on the same side' of $x$.
\begin{cor}[$\RCAo+\WKL$]\label{fromein}
A cadlag function $f:[0,1]\di \R$ has a modulus of cadlag, i.e.\ there is $G:(\R\times \N)\di \N$ such that
\be\label{flaunt}\textstyle
%(\forall x\in [0,1], k\in \N)(\forall y, z\in B(x, \frac{1}{2^{G(x, k+1)}}))\big[  \underline{(y-z)(z-x)>0} \di   |f(y)-f(z)|<\frac{1}{2^{k}} \big].
(\forall k\in \N, x, y, z\in [0,1])\left[ 
\begin{array}{c}
\textstyle y, z\in  (x-\frac{1}{2^{G(x,k)}}, x) \di   |f(y)-f(z)|<\frac{1}{2^{k}} \\ 
\wedge\\ 
y\in (x, x+\frac{1}{2^{G(x, k)}}) \di   |f(x)-f(y)|<\frac{1}{2^{k}}
\end{array}
\right].
\ee
\end{cor}
\begin{proof}
In case $f:[0,1]\di \R$ is continuous, use Corollary \ref{fliep} to obtain a modulus of continuity, which readily yields a modulus of cadlag. 
In case $f:[0,1]$ is discontinuous, we obtain $(\exists^{2})$ by \cite{kohlenbach2}*{\S3} and we may use the theorem to obtain a sequence $(x_{n})_{n\in \N}$ that lists all points of discontinuity of $f$. 
Now define a modulus of cadlag $G:(\R\times \N)\di \N$ based on the following case distinction.
\begin{itemize}
\item In case $x\ne x_{n}$ for all $n\in \N$, then $G(x, k)$ is the least $N\in \N$ such that for all $y\in (x-\frac{1}{2^{N}}, x+\frac{1}{2^{N}})\cap \Q$, we have $|f(x)-f(y)|<\frac{1}{2^{k+1}}$. 
\item In case $x= x_{n_{0}}$ for some $n_{0}\in \N$, then $G(x, k-1)$ is the least $N\in \N$ such that the formula in big square brackets in \eqref{flaunt} holds for all $ y, z\in \Q\cap [0,1]$.
\end{itemize}
Then $G$ is as required by the corollary and we are done. 
\end{proof}
One can also use the previous theorem and corollary to show that cadlag functions are Baire 1 in a relatively weak system, but the technical details are somewhat tedious.
This should be contrasted with Theorem \ref{reklam} as by the latter the Big Five cannot prove that e.g.\ regulated functions are Baire 1.  One similarly establishes (part of)
the \emph{Lebesgue decomposition theorem} (see e.g.\ \cite{lebes1}).

\smallskip

Secondly, Theorem \ref{falm} has interesting consequences, e.g.\ Theorem \ref{XZ}, which should be contrasted with Theorem \ref{reklam2}.  
We shall make (seemingly essential) use of the following fragment of the induction axiom, which also follows from $\QFAC^{0,1}$.
\bdefi[$\IND_{2}$]\label{INDT}
Let $Y^{2}, k^{0}$ satisfy $(\forall n\leq k)(\exists f\in 2^{\N})(Y(f, n)=0)$.  
There is $w^{1^{*}}$ such that $(\forall n\leq k)(\exists i<|w|)(Y(w(i), n)=0)$.
\edefi
We note that the class $NBV$ from \cite{voordedorst}*{Def.~1.2}, \cite{rijnwand}*{\S1.1}, or \cite{Holland}*{p.\ 103}, is essentially the intersection between $BV$ and the cadlag functions.  
As discussed in \cite{voordedorst}, the classical Riemann-Stieltjes integral provides a natural one-to-one correspondence between the dual of the space $C([a,b])$ of continuous functions and the space $NBV([a,b])$ of (normalized) $BV$-functions.  
\begin{thm}[$\RCAo+\IND_{2}$]\label{XZ} The following are equivalent to $\ACA_{0}$.
\begin{itemize}
\item The Jordan decomposition theorem for cadlag $BV$-functions.
\item The Jordan decomposition theorem for $BV$-functions satisfying the equality $f(x)=\frac{f(x+)+f(x-)}{2}$ for $x\in (0,1)$.
\item The Jordan decomposition theorem for quasi-continuous $BV$-functions.
\end{itemize}
We do not need $\IND_{2}$ for the first item. 
\end{thm}
\begin{proof}
We establish the equivalence between $\ACA_{0}$ and the first item based on Theorem \ref{falm} and the observation that cadlag functions do not have removable discontinuities.  One proceeds analogously for the second and third item, as the functions therein also do not have removable discontinuities. 
As shown in \cite{dagsamXI}*{Theorem 3.33}, a $BV$-function is regulated assuming $\IND_{2}$.

\smallskip

First of all, assume $\ACA_{0}$ and fix a cadlag $BV$-function $f:[0,1]\di \R$.  
In case the latter is continuous, Corollary \ref{fliep} provides an RM-code.
We can now apply the second-order RM results from \cite{nieyo}*{\S3} to obtain codes for (continuous) increasing functions $g, h:[0,1]\di \R$ such that $f=g-h$ on $[0,1]$. 
Theorem \ref{plofkip} thus yields the first item from the theorem, in this case.  In case $f:[0,1]\di \R$ is discontinuous, we have access to $(\exists^{2})$ by \cite{kohlenbach2}*{\S3}.  
By Theorem \ref{falm}, there is a sequence $(x_{n})_{n\in \N}$ of all reals in $[0,1]$ where $f$ is discontinuous; indeed, since $f$ is cadlag, it only has jump discontinuities by definition.  
Given $(x_{n})_{n\in \N}$, the supremum in \eqref{tomb} can be replaced by a supremum over $\Q$ and $\N$.  As a result, we can define $V_{a}^{b}(f)$ as in \eqref{tomb} using $\exists^{2}$, where $a, b\in [0,1]$ are parameters.  Clearly, $g(x):=\lambda x.V_{0}^{x}(f)$ is an increasing function, and the same for $h(x):=g(x)-f(x)$ via an elementary argument. 
Hence, $f=g-h$ in this case as well, and the first item follows.  

\smallskip

Secondly, assume the first item of the theorem.  We now invoke the law of excluded middle as in $(\exists^{2})\vee \neg(\exists ^{2})$. 
In the former case we are done, as $(\exists^{2})\di \ACA_{0}$ is trivial.  In the latter case, i.e.\ we have $\neg(\exists^{2})$, all functions on $\R$ are continuous by \cite{kohlenbach2}*{\S3}.  
Hence, the first item now expresses:  
\begin{center}
\emph{a continuous $f : [0, 1] \di \R$ of bounded variation is the difference of two continuous non-decreasing $g, h:[0,1]\di \R$.}
\end{center}
Now fix some code $\Phi$ for a continuous $BV$-function and use Theorem \ref{plofkip} to obtain third-order $f:[0,1]\di \R$ that equals the value of $\Phi$ everywhere. 
By the centred statement, there are two continuous non-decreasing $g, h:[0,1]\di \R$ such that $f=g-h$ on $[0,1]$.  
Now consider:
\[\textstyle
(\forall x\in [0,1], k\in \N)(\exists N\in \N)(|g(x)-g(x+\frac{1}{2^{N}})|<\frac{1}{2^{k}} \wedge |g(x)-g(x-\frac{1}{2^{N}})|<\frac{1}{2^{k}} ).
\]
Applying $\QFAC^{1,0}$, one obtains a (continuous) modulus of continuity for $g$, as $g$ is non-decreasing.
Following Remark \ref{LEM2}, this readily yields an RM-code for $g$ (and $h$), i.e.\ we have also established the second-order version of 
the centred statement.  The latter implies $\ACA_{0}$ by \cite{nieyo}*{\S3}, and we are done. 
\end{proof}
One possible addition to the previous theorem is as follows:  a real function is usco if and only if it is \emph{sequentially usco}, i.e.\ the pointwise limit of a \emph{descreasing} sequence of continuous functions (see e.g.\ \cite{bengelkoning}*{p.\ 62}).  
This sequential notion goes back to Baire's equivalent definition of usco (see \cite{beren}) and the associated restriction of Jordan decomposition theorem is readily\footnote{Fix $f\in BV$ and let $(f_{n})_{n\in \N}$ be a decreasing sequence of continuous functions with pointwise limit $f$.  By Theorem \ref{falm}, we only need to enumerate the removable discontinuities of $f$.  
Using $\exists^{2}$, one readily enumerates the \emph{strict local maxima} of a {continuous} $g:[0,1]\di \R$ (\cite{samwollic22}*{p.\ 272}), i.e.\ those $x\in [0,1]$ such that $(\exists N\in \N)(\forall y\in B(x, \frac{1}{2^{N}}))( x\ne y \di g(y)<g(x))$.  
Now let $(x_{m})_{n\in \N}$ be an enumeration of all strict local maxima of all $f_{n}$.  For any $m\in \N$, $x_{m}$ is a removable discontinuity of $f$ if and only if there is $n_{0}\in \N$ such that $x_{m}$ is a strict local maximum of $f_{n}$ for $n\geq n_{0}$.} seen to be equivalent to $\ACA_{0}$.  A similar result can be obtained for Baire 1$^{*}$ formulated using RM-codes for closed sets, in light of \cite{dagsamXIII}*{Lemma 4.11}.  

\smallskip

Next, Theorem \ref{falm} has the following consequence.  We refer to \cite{simpson2}*{p.\ 136} for the details on Riemann integration in $\RCA_{0}$. 
%\end{proof}
\begin{thm}[$\RCAo$]\label{defzie} The axiom $\WKL$ is equivalent to: 
\begin{itemize}
\item a code for a continuous function on the unit interval is Riemann integrable. 
\item a continuous function on the unit interval is Riemann integrable. 
\item a cadlag function on the unit interval is Riemann integrable. 
\end{itemize}
\end{thm}
\begin{proof}
The equivalence for the first item is immediate by \cite{simpson2}*{IV.2.7}. 
For the second item, one additionally uses Theorem \ref{plofkip} and Corollary \ref{fliep}.  

\smallskip

Now assume the third item and fix some code $\Phi$ for a continuous function on $[0,1]$.  Use Theorem \ref{plofkip} to convert the latter into a continuous 
third-order function, which is trivially cadlag.  By the first item, this function is Riemann integrable, and hence so is the function represented by $\Phi$.  We obtain $\WKL$ by the first item. 

\smallskip

Now assume $\WKL$ and let $f:[0,1]\di \R$ be cadlag.  If the latter is also continuous, we may use the second item to obtain the third one.   In case $f$ is discontinuous, we obtain $(\exists^{2})$ by \cite{kohlenbach2}*{\S3}.
Use Theorem \ref{falm} to obtain a sequence $(x_{n})_{n\in \N}$ which enumerates all the points where $f$ is discontinuous (as cadlag functions do not have removable discontinuities).
The usual `epsilon-delta' proof now goes through assuming a modulus as provided by Theorem \ref{fromein}.  
\end{proof}
Finally, we mention some related results from the RM of $\ACA_{0}$.  
\begin{rem}\rm
The RM of $\ACA_{0}$ involves some theorems from analysis, like e.g.\ \cite{simpson2}*{IV.2.11 and III.2.2}.
In the same way as above, one shows that the following are also equivalent to $\ACA_{0}$ over $\RCAo$.
We use `RM-closed' to refer to the second-order definition of codes for closed set in RM (\cite{simpson2}*{II.5.6}).
\begin{itemize}
\item Let $F:C\di \R$ be cadlag where $C\subset [0,1]$ is an RM-closed set.  Then $\sup_{x\in C}F(x)$ exists. 
\item Let $F:C\di \R$ be cadlag and usco where $C\subset [0,1]$ is an RM-closed set.  Then $F$ attains a maximum value on $C$.
\item Let $(f_{n})_{n\in \N}$ be a Cauchy sequence \(relative to the sup norm\) of continuous functions.  Then the limit function exists and is continuous. 
\item Let $(f_{n})_{n\in \N}$ be a Cauchy sequence \(relative to the sup norm\) of cadlag functions.  Then the limit function exists and is cadlag. 
\end{itemize}
Another promising theorem is the compactness theorem (\cite{thebill}*{Theorem 14.3}) for the Skorohod space (of cadlag functions), which is presented as a generalisation of the Arzel\`a-Ascoli theorem.
The latter is part of the RM of $\ACA_{0}$ by \cite{simpson2}*{III.2.9}.  Similarly, a version of the Arzel\`a-Ascoli theorem for quasi-continuous functions exists, namely \cite{haloseg}*{Prop.\ 2.22} and related theorems.
\end{rem}

\subsection{Equivalences for $\Pi_{1}^{1}$-comprehension}\label{FIVE}
We establish some equivalences for $\FIVE$ involving third-order theorems from analysis.  

\smallskip

First of all, we establish Theorem \ref{flapke} to be contrasted with Theorem~\ref{truppke}.
Here, $\Sigma_{1}^{1}\textsf{-IND}$ is the induction axiom for $\Sigma_{1}^{1}$-formulas and $\IND_{2}$ is as in Definition \ref{INDT}. 
As to notation, fix $(r_{n})_{n\in \N}$, a standard injective enumeration of the non-negative rational numbers. 
For $B \subset \Q^+$, we say that `$B$ is $\Sigma^1_1$ with parameter $x\in \N^{\N}$', if $A = \{a : r_a \in B\}$ is $\Sigma^1_1$ with parameter $x$. Since we do not always have access to $\Sigma^1_1$-comprehension, we refer to both $A$ and $B$ as (defined) classes.

\begin{theorem}[$\ACAo$]\label{flapke}
The following are equivalent.
\begin{enumerate}
\renewcommand{\theenumi}{\roman{enumi}}
\item For any $x\in \N^{\N}$, any bounded $\Sigma^{1,x}_1$-class in $\Q^+$ has a supremum.\label{birst}
\item A bounded effectively Baire 2 $f:[0,1]\di \R$ has a supremum.\label{becond}
\item For $n \geq 2$, a bounded and  effectively Baire $n$ $f:[0,1]\di \R$ has a supremum.\label{bird}
\end{enumerate}
Assuming $\IND_{2}+\Sigma_{1}^{1}\textsf{\textup{-IND}}$, these items are equivalent to $\FIVE$.
\end{theorem}
\begin{proof}
We first prove that item \eqref{birst} implies items \eqref{becond} and \eqref{bird}.
Let $f:[0,1]\di [0,1]$ be effectively Baire 2, i.e.\ there is a double sequence $(f_{n,m})_{n,m\in \N}$ of continuous functions such that $f(x)=\lim_{n\di \infty}\lim_{m\di \infty}f_{n,m}(x)$ for $x\in [0,1]$.  
Now consider the following for $r\in \Q$:
\begin{align}
&(\exists y\in [0,1])(f(y)>r)\notag\\
&\asa (\exists x\in [0,1]) (\exists n\in \N)( \forall i \geq n)( \exists m\in \N) (\forall j \geq m)( f_{i,j}(x) > r). \label{hieo}
\end{align}
By Corollary \ref{fliep}, we can replace $f_{i, j}$ by a sequence of RM-codes, rendering \eqref{hieo} part of the language of second-order arithmetic. 
%By the Kleene normal form theorem (see \cite{simpson2}*{V.5.4}), \eqref{hieo} is therefore (equivalent to) $\Sigma_{1}^{1}$, 
Using $\exists^{2}$, \eqref{hieo} is equivalent to a $\Sigma_{1}^{1}$-formula, i.e.\ we may form the set $\{r\in \Q : (\exists y\in [0,1])(f(y)>r) \}$ using item~\eqref{birst}, from which the supremum of $f$ is readily defined using $\exists^{2}$, i.e.\ item~\eqref{becond} follows.  Item \eqref{bird} is proved in the same way, where \eqref{hieo} becomes more complicated due to the presence of more arithmetical quantifiers originating from the definition of `effectively Baire $n$'.  

\smallskip

Secondly, we prove that item \eqref{becond} implies item \eqref{birst}. Let $B = \{r_a : a \in A\}$ be bounded, where $A$ is $\Sigma^1_1$ and given by:
\[
a \in A \leftrightarrow (\exists x \in 2^\N) (\forall m\in \N) (\exists n\in \N) R(a,x,m,n),\] 
where $R$ is primitive recursive. We now construct continuous functions $F_{n,m}:2^\N \rightarrow \R$ such that the double limit $F = \lim_{m \rinf}\lim_{n \rinf}F_{n,m}$ is well defined and such that $\sup F = \sup B$. 
Identifying $2^\N$ with the Cantor set, we extend each $F_{n,m}$ to a continuous function $f_{n,m}:[0,1] \rightarrow \R$ by extending the graph with straight lines. 
Note that all limits commute with this extension and that the corresponding extension $f$ of $F$ is Baire 2 with the same supremum.
For each $a \in \N$, let $G_a(x) = \lim_{m \rinf}\lim_{n \rinf} G_{a,n,m}(x)$ be the characteristic function of the set $\{x \in 2^N : (\forall m\in \N) (\exists n\in \N) R(a,x,m,n)\}$, where
\[
G_{a,n,m}(x) := 
\begin{cases}
1 &  \textup{if $(\forall i \leq m)( \exists j \leq n) R(a,x,i,j)$}\\
0 & \textup{otherwise}. 
\end{cases}.
\]
We now define $F_{n,m}:[0,1]\di \R$ by cases as follows.
\begin{itemize}
\item If $x$ is of the form $\underbrace{1*\dots *1}_{\textup{$m+1$ times}}*~y$, we define $F_{n,m}(x): = 0$.
\item If for $a \leq m$, $x$ is of the form $\underbrace{1*\dots *1}_{\textup{$a$ times}}*~0* y$, define $F_{n,m}(x) := r_a  G_{a,n,m}(y)$.
\end{itemize}
If $x = 11\dots$  then $F_{n,m}(x) = 0$ for all $n,m\in \N$, so in  the double limit we have that $F(x) = 0$. If not, $x$ is of the form $\underbrace{1*\dots *1}_{\textup{$a$ times}}*0* y$ for some $a \geq 0$. 
For all $m \geq a$ and all $n\in \N$  we have that $F_{n,m}(x) = r_a  G_{a,n,m}(y)$. Then $F(x) = r_a$ if $(\forall m\in \N) (\exists n\in \N) R(a,y,m,n)$, and 0 otherwise. 
Then $\sup F = \sup B$, so the latter exists by the assumption that the former exists.

\smallskip

Thirdly, item \eqref{birst} clearly follows from $\FIVE$ and it is a tedious but straightforward verification that the reversal goes through assuming $\IND_{2}+\Sigma_{1}^{1}\textsf{\textup{-IND}}$. 
\end{proof}
We note that the use of $(\exists^{2})$ as part of the base theory in Theorem \ref{flapke} is necessary: in isolation, items \eqref{becond} and \eqref{bird} do not exceed $\WKL_{0}$ in terms of second-order consequences.  
This follows via the $\ECF$-interpretation from Remark \ref{ECF}.  

\smallskip

Secondly, we have the following corollary to Theorem \ref{flapke}, to be contrasted with Theorem~\ref{truppke2}.  We say that a \emph{set} is `effectively Baire~$n$' if the characteristic function has this property.  
The notion of Baire set may be found in \cite{kodt}*{p.\ 21} under a different name; we refer to \cite{lorch} for an introduction and to \cite{dudley}*{\S7} for equivalent definitions, including that of Borel set in Euclidean space.  
\begin{thm}[$\ACAo+\FIVE$]\label{flappy}
For any open effectively Baire $n$ set $O\subset [0,1]$, there exists an RM-code $(n\geq 2)$.  
%$(a_{n})_{n\in \N}, (b_{n})_{n\in \N}$ such that $x\in O\asa (\exists n\in \N)(x\in (a_{n}, b_{n}) )$ for all $x\in [0,1]$.
\end{thm}
\begin{proof}
We make use of the items in Theorem \ref{flapke}.  In particular, the proof of these immediately generalises to infima involving rational parameters, i.e.\ we have 
\begin{center}
\emph{For a bounded effectively Baire $n$ function $f:[0,1]\di \R$, there is $F:\Q^{2}\di \R$ such that for all $p, q\in \Q\cap[0,1]$, the real $F(p, q)$ equals $\inf_{x\in [p, q]}f(x)$ .}
\end{center}
Now consider the sequence in \eqref{tugger} as in the proof of Theorem \ref{klank}. 
\end{proof}
The previous proof essentially establishes that an effectively Baire $n$ open set can be represented by a code for an open set (see \cite{simpson2}*{II.5.6}).
Hence, any theorem from the RM of $\FIVE$ immediately generalises from `codes for open sets' to `third-order open sets that are effectively Baire $n$'.
The RM of $\FIVE$ contains considerable results on codes for open and closed sets sets (see \cite{brownphd, browner, browner2, simpson2}), including the Cantor-Bendixson theorem.
The same holds \emph{mutatis mutandis} for open sets with quasi-continuous characteristic functions.  One can similarly generalise Theorem \ref{weerklank} to effectively Baire $n$ functions.

\smallskip

\subsection{Equivalences for arithmetical transfinite recursion}\label{FOUR}
We establish equivalences for $\ATR_{0}$ involving third-order theorems from analysis.  
We also establish Theorem~\ref{falnt2} which shows that adding the extra condition `Baire 1'
converts theorems about $BV$-functions from `not provable in $\Z_{2}^{\omega}$' to `provable from $\ATR_{0}$ plus induction'.
Remark \ref{donola} again explains why there is no contradiction here. 

\smallskip

First of all, we have a corollary to \cite{basket2}*{Theorem 6.5}, to be contrasted with item~\eqref{ta21} from Theorem \ref{reklam}.  Here, $\Delta_{2}^{1}\textsf{-IND}$ is the induction axiom for $\Delta_{2}^{1}$-formulas. 
\begin{thm}[$\ACAo$ + $\Delta_{2}^{1}\textsf{-IND}$]\label{floepje}
The following are equivalent to $\ATR_{0}$.
\begin{itemize}
%\item The system $\ATR_{0}$.  
\item Cousin's lemma for codes for Baire 2 functions. 
\item Cousin's lemma for effectively Baire 2 $\Psi:[0,1]\di \R^{+}$.
\item Cousin's lemma for effectively Baire $n$ $\Psi:[0,1]\di \R^{+}$ $(n\geq 2)$.
\end{itemize}
\end{thm}
\begin{proof}
It is known that $\ATR_{0}$ is equivalent to Cousin's lemma for codes for Baire 2 (or: any $n\geq 2$) functions, working over $\RCAo$ plus $\Delta_{2}^{1}$-induction (see \cite{basket, basket2}).  
Now, a code for a Baire $n$ function is essentially an effectively Baire $n$ function where the continuous functions are given by codes. 
As noted below \cite{basket2}*{Def.\ 6.1}, $\ACA_{0}$ suffices to show that a code for a Baire $n$ function has a (unique) value.   
Hence, $\exists^{2}$ readily defines a third-order function taking these values everywhere on $[0,1]$.  
Similarly, an effective Baire $n$ function readily becomes a code for a Baire $n$ function by replacing the continuous functions by codes for continuous functions (see Corollary~\ref{fliep}).
In this way, the base theory connects the items from the theorem and we are done. 
\end{proof}
By the previous, (full) Cousin's lemma plus $(\exists^{2})$ implies $\ATR_{0}$ assuming some induction.  
The use of $(\exists^{2})$ is again essential as Cousin's lemma in isolation does not exceed $\WKL_{0}$ in terms of second-order consequences.  
By Theorem \ref{flame}, $(\exists^{2})$ is equivalent to the statement that a code for a Baire 1 (or Baire $n$) function denotes a third-order function.  
Hence, the strength of Cousin's lemma for codes for Baire~2 functions is actually due to the coding of Baire 2 functions.

\smallskip

Secondly, we obtain equivalences involving $\ATR_{0}$ and the Jordan decomposition theorem, to be contrasted with Theorem \ref{reklam2}.  
%We say that a $\R\di \R$-function is `arithmetical' (resp.\ analytic) in case its graph is given by a (second-order?) arithmetical (resp.\ analytic) formula. 
Moreover, Theorem \ref{falnt} also shows that the RM of $\ATR_{0}$ is a special case of the higher-order RM of the (full) Jordan decomposition theorem, where the latter is developed in \cite{dagsamXI}*{\S3.3}. 
In the below, notions like \emph{arithmetical}, $\Sigma^1_1$, etc.\ are based on the `standard' definition, i.e.\ with the understanding that we (only) allow parameters of type 0 and of type $1$. 
We say that a function $f:[0,1]\di \R$ is $\Sigma^1_1$ if its graph is $\Sigma^1_1$, which is equivalent to the graph being $\Delta^1_1$, and to being Borel measurable.
\begin{thm}[$\ACAo+\IND_{2}+\Sigma_{2}^{1}\textsf{-IND}$]\label{falnt}
The following are equivalent to $\ATR_{0}$.
\begin{enumerate}
%\item $\ATR_{0}$
\renewcommand{\theenumi}{\roman{enumi}}
\item For arithmetical formulas $\varphi$ such that 
\be\label{fln2}
(\forall n\in \N)(\exists \textup{ at most one } X\subset \N)\varphi(X, n),
\ee
the set $\{ n\in \N:(\exists X\subset \N)\varphi(X, n)\}$ exists. \label{bont1}
\item For arithmetical $f:[0,1]\di \R$ in BV, there is a sequence $(x_{n})_{n\in \N}$ enumerating all points where $f$ is discontinuous.\label{bont2}  
\item For a $\Sigma^1_1$-function  $f:[0,1]\di \R$ in BV, there is a sequence $(x_{n})_{n\in \N}$ enumerating all points where $f$ is discontinuous.  \label{bont3}
\item The Jordan decomposition theorem \(Theorem \ref{drd}\) restricted to arithmetical \(or: $\Sigma_{1}^{1}$\) functions in $BV$. \label{bont4}
\item A non-enumerable arithmetical set in $\R$ has a limit point.\label{bont5}
%\item The Jordan decomposition theorem \(Theorem \ref{drd}\) restricted to $\Sigma^1_1$-functions in $BV$. \label{bont5}
%\item The Jordan decomposition theorem \(Theorem \ref{drd}\) for $\Delta_{1}^{1}$???? $BV$-functions. \label{bont6}
%\item For $f:[0,1]\di \R$ in BV and effectively Baire 2, enumerate all points of discontinuity.  
\end{enumerate}
\end{thm}
\begin{proof}
The equivalence between $\ATR_{0}$ and item \eqref{bont1} is found in \cite{simpson2}*{V.5.2}.
Now assume item \eqref{bont2} and fix arithmetical $\varphi$ such that \eqref{fln2}.
Note that we can use $\mu^{2}$ to find those $n\in \N$ such that $(\exists X\subset \N)\varphi(X, n)$ and there is $m_{0}$ such that $X(m)=1$ for $m\geq m_{0}\in \N$, so without loss of generality we may assume that there are no such $n$.
Define the function $f:[0,1]\di \R$ as follows:
\be\label{fln}
f(x):=
\begin{cases}
\frac{1}{2^{n+3}} & \textup{ the least $n\in \N$ such that } \varphi(\mathfrak{b}(x), n)\\
0  & \textup{ otherwise }
\end{cases},
\ee
where $\mathfrak{b}: [0,1]\di 2^{\N}$ converts real numbers to a binary representation, choosing a tail of zeros if applicable.  
In light of \eqref{fln2}, for every $n\in \N$, there is at most one $x\in [0,1]$ such that $f(x)=\frac{1}{2^{n}}$.  
Hence, the sum $\sum_{i=0}^{k-1} |f(x_{i})-f(x_{i+1})|$ as in \eqref{tomb} is at most $\sum_{n=1}^{k}\frac{1}{2^{n}}$, which is at most $1$.
By definition, $f$ is arithmetical and item~\eqref{bont2} provides a sequence $(x_{m})_{m\in \N}$ with all points where $f$ is discontinuous.
Hence, we have for all $n\in \N$ that  
\[
(\exists X\subset \N)\varphi(X, n)\asa (\exists m\in \N)\varphi(\mathfrak{b}(x_{m}), n),
\]
where, as we assumed,  $X$ in the left-hand side will not have a tail of $1$'s.  

\smallskip

Item \eqref{bont4} implies item \eqref{bont3} as $\exists^{2}$ allows us to enumerate the points of discontinuity of increasing functions (see \cite{dagsamXII}*{Lemma 7}). Of course item \eqref{bont3} implies \eqref{bont2}. We also have that item \eqref{bont3} implies item \eqref{bont4} as follows:
the sequence in item \eqref{bont3} allows us to replace the supremum in \eqref{tomb} by one over $\N$ and $\Q$.  
Hence, $\exists^{2}$ can define the increasing function $g(x):=\lambda x.V_{0}^{x}(f)$.  By noting that $h:=f-g$ is also increasing, item \eqref{bont4} follows. 

\smallskip

Finally, assume item \eqref{bont1} and fix a $\Sigma^1_1$-function $f\in BV$ with bound $k_{0}=1$ as in Definition \ref{varvar}.  
Now consider the set
\be\label{fln3}\textstyle
D_{k}:= \{  x\in [0,1]: |f(x+) -f(x-)|>\frac{1}{2^{k}}  \},
\ee
where we note that $\IND_{2}$ suffices to show that $f\in BV$ is regulated (\cite{dagsamXI}*{Theorem~3.33}).  
The set $D_{k}$ is $\Sigma^1_1$  because the graph of $f$ is.  Moreover, since each element $x\in D_{k}$ contributes at least $\frac{1}{2^{k}}$ to the variation of $f$, $D_{k}$ can have at most $2^{k}$ many elements.  
Using $\Sigma_{2}^{1}$-induction, one obtains\footnote{For $X\subset \R$, $N\in \N$, define the notation `$|X|\leq N$', i.e.\ $X$ has at most $N$ elements, as:  
\be\label{dango}
(\forall w^{1^{*}})\big(\big[|w|>N \wedge (\forall i,j<|w|)(i\ne j\di w(i)\ne w(j)) \big]\di (\exists k<|w|)((w(k)\not \in X))  \big).
\ee
Using \eqref{dango}, let $\varphi(n, X)$ be the following formula: 
\be\label{dango2}
|X|\leq n \di (\exists v^{1^{*}})(\forall x\in \R)\big(\big[x\in X\di (\exists i<|v|)( v(i)=x)\big] \wedge |v|\leq n \big),
\ee
expressing that a set with at most $n$ elements can be enumerated by a finite sequence of length $n$. 
Then \eqref{dango} is $\Pi_{1}^{1}$ if $X$ is $\Sigma_{1}^{1}$ while \eqref{dango2} is then $\Sigma_{2}^{1}$.   
Hence, for $X$ in $\Sigma_{1}^{1}$, $\Sigma_{2}^{1}$-induction on $\varphi(n, X)$ establishes the desired enumeration.\label{urba}
} 
an enumeration of $D_{k}$ for fixed $k\in \N$.  
By \cite{simpson2}*{V.4.10}, which is provable in $\ATR_{0}$, $\cup_{k\in \N}D_{k}$ can now be enumerated, and item \eqref{bont3} follows. 

\smallskip

For item \eqref{bont5}, fix $\varphi$ as in \eqref{fln2} and define the set $A\subset \R$ by putting $x\in A$ in case $x\in [n+1, n+2)$ and $\varphi(\mathfrak{b}(x-(n+1)), n)$.
Since $A\cap [0,n]$ contains at most $n$ elements, $A$ has no limit points, i.e.\ item \eqref{bont5} readily yields item \eqref{bont1}. 
For the reversal, let $A\subset \R$ be a set without limit points.  By contraposition, $A\cap [-n,n]$ is finite for each fixed $n\in \N$ (for which we use $\SAC$).
As in the previous paragraphs of the proof, we can enumerate $A=\cup_{n\in \N}\big(A\cap [-n,n]\big )$.
\end{proof}
As above, $(\exists^{2})$ is essential for the equivalence in Theorem \ref{falnt} as the Jordan decomposition theorem in isolation cannot go beyond $\ACA_{0}$ in terms of second-order consequences, a fact observed again using $\ECF$ from Remark \ref{ECF}.  

\smallskip

Thirdly, we obtain a version of Theorem \ref{falm} for $BV$-functions that are also in Baire 1.  
As discussed in Remark \ref{donola}, while $BV$-functions are Baire 1, this basic fact is not provable in $\ACAo$ and much stronger systems like $\Z_{2}^{\omega}$.  
By \cite{dagsamXI}*{\S3.3}, listing \emph{all} points of discontinuity of $BV$-functions similarly cannot be done in $\Z_{2}^{\omega}$.
\begin{thm}[$\ACAo+\IND_{2}+\Sigma_{2}^{1}\textsf{-IND}+\ATR_{0}$]\label{falnt2}
For Baire 1 $f:[0,1]\di \R$ in $BV$, there is a sequence $(x_{n})_{n\in \N}$ enumerating all points where $f$ is discontinuous.  
\end{thm}
\begin{proof}
Let $f:[0,1]\di \R$ be Baire $1$ and in $BV$, say with variation bounded by $1$.
In light of Theorem \ref{falm}, we only need to enumerate the `removable' discontinuities if $f$, i.e.\ those $x\in (0,1)$ for which $f(x)\ne f(x+)$ and $f(x+)=f(x-)$.  
Let $(f_{n})_{n\in \N}$ be a sequence of continuous functions with pointwise limit $f$ on $[0,1]$.   Now consider the following formula
\be\label{xdx}\textstyle
(\exists n_{0}\in \N)(\forall n, m\geq n_{0})(\forall q\in B(x, \frac{1}{2^{m}})\cap \Q)(|f_{n}(x)-f(q)|>\frac{1}{2^{k}}),   
\ee
which holds in case $f$ has a removable discontinuity at $x\in (0,1)$ such that $|f(x)-f(x+)|>\frac1{2^{k}}$.   There can only be $2^{k}$ many pairwise distinct $x\in [0,1]$ such that \eqref{xdx}, as each
such real contributes at least $\frac1{2^{k}}$ to the total variation.  Clearly, the formula \eqref{xdx} is equivalent to (second-order) arithmetical as $f$ only occurs with rational input and $f_{n}$ can be replaced uniformly by a sequence of codes $\Phi_{n}$.  
Using $\Sigma_{2}^{1}$-induction, one can enumerate all reals satisfying \eqref{xdx} for fixed $k\in \N$, as in the proof of Theorem \ref{falnt} and Footnote \ref{urba}.
%Recall that \cite{simpson2}*{V.4.10} expresses that for a sequence of enumerable analytic sets $(A_{n})_{n\in \N}$, the union $\cup_{n\in \N}A_{n}$ is also enumerable, working in $\ATR_{0}$. 
Again using \cite{simpson2}*{V.4.10}, we can enumerate all reals satisfying \eqref{xdx}, and we are done. 
\end{proof}
With some effort, one generalises Theorem \ref{falnt2} from `Baire 1' to `effectively Baire $n$'; it goes without saying that \eqref{xdx} becomes more complicated. 
The same goes for the generalisation from `$BV$' to `regulated', which seems to require $\QFAC^{0,1}$.
Unfortunately, Theorem \ref{falnt2} cannot be pushed down to $\ACA_{0}$ as the union of enumerable arithmetical sets does not necessarily\footnote{To see, let $X = \langle X_ 1, . . . , X_n\rangle $ be in $A_n$ if and only if 
$X_1 = \emptyset$ and for all $i < n$, $X_{i+1}$ is the Turing jump of $X_i$. This constitutes the first $n$ elements in the jump hierarchy, coded as one object, and $A_{n}$ is arithmetical of a complexity independent of $n$. 
Now, each $A_n$ is a singleton with an arithmetical element, but the union does not have any arithmetical enumeration.
} have an arithmetical enumeration.  

\smallskip

Finally, we now establish the following, to be contrasted with the final item of Theorem \ref{truppke}, Theorem \ref{reklam2}, and Theorem \ref{drel}.
We again stress Remark \ref{donola} which explains why there is no contradiction here: rather strong systems are unable to prove that $BV$ or usco functions are in fact Baire 1. 
\begin{thm}[$\ACAo+\IND_{2}+\Sigma_{2}^{1}\textsf{-IND}+\ATR_{0}$]\label{falnt3}
The following are provable.  
\begin{itemize}
\item The Jordan decomposition theorem \(Theorem \ref{drd}\) for $BV$-functions in Baire 1. 
\item A bounded Baire 1 $BV$-function $F:[0,1]\di \R$ has a supremum.
\item {For a Riemann integrable $BV$-function $f:[0,1]\di [0,1]$ in Baire 1 with $\int_{0}^{1}f(x)dx=0$, there is $x\in [0,1]$ such that $f(x)=0$.}
\end{itemize}
\end{thm}
\begin{proof}
For the first item, the sequence provided by Theorem \ref{falnt2} allows one to replace `supremum over $\R$' by `supremum over $\N$' in  \eqref{tomb}.  Hence, we can define $g(x):=\lambda x.V_{0}^{x}(f)$ using $\exists^{2}$.  Now, $g$ is (trivially) increasing and one readily verifies the same for $h=f-g$, i.e.\ a Jordan decomposition is immediate.  The second item follows in the same way.  The third item follows by using \cite{simpson2}*{II.4.10} to obtain a real $y\in [0,1]$ not in the sequence provided by Theorem \ref{falnt2}; by definition, $f$ must be continuous at $y$, and $f(y)=0$ readily follows. 
\end{proof}
With some effort, one generalises Theorem \ref{falnt3} from `Baire 1' to `effectively Baire $n$'. It goes without saying that the proof becomes more complicated.

\subsection{Equivalences for Kleene's arithmetical quantifier}\label{neweqi}
We establish interesting equivalences for $(\exists^{2})$. 
To this end, Thomae's function as follows is useful:
\be\label{thomae}
f(x):=
\begin{cases} 
0 & \textup{if } x\in \R\setminus\Q\\
\frac{1}{q} & \textup{if $x=\frac{p}{q}$ and $p, q$ are co-prime} 
\end{cases}.
\ee
Thomae introduces this function around 1875 in \cite{thomeke}*{p.\ 14, \S20}) to show that Riemann integrable functions can have a dense set of discontinuity points.  
As in the previous, the coding of Baire $n$ functions is taken from \cite{basket, basket2}.  
\begin{thm}[$\RCAo+\WKL$]\label{flame}
The following are equivalent to $(\exists^{2})$.
\begin{enumerate}
\renewcommand{\theenumi}{\roman{enumi}}
\item There exists Riemann integrable $f:[0,1]\di [0,1], g:[0,1]\di \R$ such that $g\circ f$ is not Riemann integrable.\label{F1}
\item  There exists a function that is not Riemann integrable. \label{F12}
\item There exists regulated $f:[0,1]\di [0,1], g:[0,1]\di \R$ such that $g\circ f$ is not regulated.\label{F2}
\item  There exists a function that is not regulated. \label{F3}
\item There exists $f:[0,1]\di [0,1], g:[0,1]\di \R$ in Baire 1 such that $g\circ f$ is not in Baire 1.\label{F4}
%\item There exists $f:[0,1]\di [0,1], g:[0,1]\di \R$ in $BV$ such that $g\circ f$ is not in $BV$. ???
%Experimental: 
\item There exists a function $f:[0,1]\di \R$ that is not Baire 1.\label{F5}
\item There exists usco $f:[0,1]\di [0,1], g:[0,1]\di \R$ such that $g\circ f$ is not usco.\label{F6}
\item  There exists a function that is not usco. \label{F7}
%\item (exp?) There exists $\Lambda$ and a function that is in not $\Lambda BV$. 
%\item (exp?) For all $\Lambda$, there is a function that is in not $\Lambda BV$. 
\item There exists a function that is not quasi-continuous. \label{F10}
\item There exists a function that is not cliquish. \label{F11}
\item There exists a function $f:[0,1]\di \R$ that is unbounded. \label{F1212}
\item There exists a Baire 1 function $f:[0,1]\di \R$ that is unbounded. \label{F122}
\item There exists a function $f:[0,1]\di \R$ that is not locally bounded\footnote{A function $f:[0,1]\di \R$ is \emph{locally bounded} if for all $x\in [0,1]$, there is $N\in\N$ such that $(\forall y\in B(x, \frac{1}{2^{N}})\cap [0,1])(|f(y)|\leq N)$.}. \label{F13}
\item There exists Darboux functions $f:[0,1]\di [0,1], g:[0,1]\di \R$ such that $g+ f$ is not Darboux.\label{F14}
%\item There exists quasi-continuous functions $f:[0,1]\di \R, g:[0,1]\di \R$ such that $g+ f$ is not quasi-continuous????.\label{F16}
\item There is a bounded Darboux $f:[0,1]\di \R$ which does not attains its sup. \label{F15}
\item For a code for a Baire 1 function on $[0,1]$, there exists a third-order function that equals the value of the code on $[0,1]$.\label{F16}  
\item For a code for a Baire $n$ function on $[0,1]$, there exists a third-order function that equals the value of the code on $[0,1]$ $(n\geq 2)$.\label{F17}
%\item Something about cadlag!
\end{enumerate}
We only need $\WKL$ for the items \eqref{F1}, \eqref{F12}, \eqref{F1212}, \eqref{F122}, \eqref{F15}.  
\end{thm}
\begin{proof}
First of all, assume $(\exists^{2})$ and consider Thomae's function $f$ as in \eqref{thomae}; one readily verifies that $f$ is Riemann integrable (with integral equal to zero) and regulated (with zero as left and right limits) on any interval.
Now define $g:[0,1]\di \R$ as $0$ in case $x=0$, and $1$ otherwise; this function is trivially Riemann integrable and regulated.  
However, $g\circ f$ is Dirichlet's function $\mathbb{1_{Q}}$, i.e.\ the characteristic function of the rationals, which is readily shown to be \emph{not} Riemann integrable and \emph{not} regulated.
Thus, $(\exists^{2})$ implies items \eqref{F1}-\eqref{F3}. 

\smallskip

Secondly, assume item \eqref{F2} (similar for item \eqref{F3}) and note that $g\circ f$ must be discontinuous, as continuous functions are trivially regulated. 
However, the existence of a discontinuous function on $\R$ yields $(\exists^{2})$ by \cite{kohlenbach2}*{\S3}.
Similarly, for items \eqref{F1} and \eqref{F12}, $\WKL$ suffices to obtain an RM-code for a continuous function on Cantor space (see \cite{kohlenbach4}*{\S4}); the same goes through \emph{mutatis mutandis} for functions on $[0,1]$ by Corollary \ref{fliep}.  
Hence, $\WKL$ suffices to show that a continuous function on $[0,1]$ is Riemann integrable by \cite{simpson2}*{IV.2.6}.
Thus, $g\circ f$ must be discontinuous, which yields $(\exists^{2})$ by \cite{kohlenbach2}*{\S3}.  Similarly, for items \eqref{F4} and \eqref{F5}, a function not in Baire 1 must be discontinuous, as continuous functions are trivially Baire 1; in this way, we obtain a discontinuous function and hence $(\exists^{2})$ by \cite{kohlenbach2}*{\S3}. The first six items now each imply $(\exists^{2})$, the first two using $\WKL$ as noted above.

\smallskip

Thirdly, assume $(\exists^{2})$ and note that Thomae's function is Baire 1.  In particular, finding a sequence of continuous functions converging to $f$ as in \eqref{thomae} is straightforward (using $\exists^{2}$).
The same holds for $g:[0,1]\di \R$ defined as $0$ in case $x=0$, and $1$ otherwise.  We now show that  $\mathbb{1_{Q}}=g\circ f$ is not Baire 1, establishing items \eqref{F4} and \eqref{F5}. 
To this end, suppose $(f_{n})_{n\in \N}$ is a sequence of continuous functions with pointwise limit $\mathbb{1_{Q}}$.  We first prove the following statement in the next paragraph.
\begin{center}
\emph{For any non-empty $[a,b]\subset [0,1]$, there is an arbitrarily large $N\in\N$ and a non-empty $[c,d]\subset [a,b]$ such that  $f_{N}([c, d])=[\frac{1}{4}, \frac{3}{4}]$.}
\end{center}
To establish the previous centred statement, fix a non-trivial interval $[a,b]\subset [0,1]$ and fix $x<y$ such that $x\in \Q\cap [a,b]$ and $y\in [a,b]\setminus \Q$.
Since $(f_{n})_{n\in \N}$ converges pointwise to $\mathbb{1_{Q}}$, there exists arbitrarily large $N\in \N$ such that $f_{N}(x)\geq \frac{3}{4}$ and $f_{N}(y)\leq\frac14$. 
By the intermediate value theorem (provable in $\RCA_{0}$ for RM-codes by \cite{simpson2}*{II.6.6}, and hence in $\ACAo$ for continuous functions), there exists an interval  $[c,d]\subseteq [x,y]\subset [a,b]$ such that $f_{N}([c,d])=[\frac14,\frac34]$.
The previous centred statement has been proved, working in $\ACAo$.  

\smallskip

By \cite{kohlenbach3}*{\S3}, $(\exists^{2})$ is equivalent to the existence of a functional witnessing the intermediate value theorem.  Hence, following the previous paragraph, $\exists^{2}$
readily yields a functional that returns the numbers $N\in \N$ and $c, d\in [0,1]$ as in the centred statement on input $[a,b]$ and $m\in \N$, where $N\geq m$.  
Using the latter functional, one readily obtains sequences $(c_{n})_{n\in \N}$, $(d_{n})_{n\in \N}$, and $g\in \N^{\N}$ such that $g(n)\geq n$, $f_{g(n)}([c_{n}, d_{n}])=[\frac14, \frac34]$, and $|c_{n}-d_{n}|<\frac{1}{2^{n}}$ for all $n\in \N$.
However, if $c=\lim_{n\di \infty }c_{n}$, then $\mathbb{1_{Q}}(c)=\lim_{n\di \infty }f_{g(n)}(c)\in [\frac14, \frac34]$, a contradiction.  Hence, we have proved item \eqref{F4} and \eqref{F5}.  
Since $\mathbb{1_{Q}}$ is not usco (or quasi-continuous or cliquish), the equivalence between $(\exists^{2})$ and items \eqref{F6}-\eqref{F11} follows in the same way. 

\smallskip

To prove item \eqref{F122} (and item \eqref{F1212}) from $(\exists^{2})$, let $f_{n}(x)$ be $2^{2n}$ in case $x\in [0,\frac{1}{2^{n}}]$ and $1/x$ if $x\in (\frac{1}{2^{n}}, 1]$.  Then each $f_{n}:[0,1]\di \R$ is continuous and the sequence converges pointwise to the function which is $1/x$ for $x>0$ and $0$ otherwise.  The latter is unbounded and Baire 1.  To prove $(\exists^{2})$ from item \eqref{F122} (or item \eqref{F1212}), note that the function provided by the latter must be discontinuous by Theorem \ref{liguster}.  However, a discontinuous function  yields $(\exists^{2})$ by \cite{kohlenbach2}*{\S3}. 
The equivalence for items \eqref{F13} follows in the same way, but without using $\WKL$ as continuous functions are trivially locally bounded. 

\smallskip
\noindent
For item \eqref{F14}, use $(\exists^{2})$ to define the following functions $f, g:[0,1]\di \R$
\[
f(x):=
\begin{cases}
\sin(\frac1x) & x\ne 0\\
1 & x=0
\end{cases} \qquad
g(x):=
\begin{cases}
-\sin(\frac1x) & x\ne 0\\
0 & x=0
\end{cases}.
\]
Clearly, $f(x)+g(x)= \mathbb{1}_{\{0\}}$, which is not Darboux.  For the reversal, a continuous function has the intermediate value property, which is provable in $\RCAo+\WKL$ by combining 
Corollary \ref{fliep} with the second-order intermediate value theorem (\cite{simpson2}*{II.6.6}).  Hence, a function that is not Darboux, is discontinuous, yielding $(\exists^{2})$ by \cite{kohlenbach2}*{\S3}.
Note that we can avoid the use of $\WKL$ by imitating the proof of \cite{simpson2}*{II.6.6} in $\RCAo$ for (third-order) continuous functions. 

\smallskip

For item \eqref{F15}, consider $f:[0,1]\di [0,1]$ defined by $f(0)=0$ and $f(x):=e^{-x}\cos(\frac1x)$ for $x>0$ using $(\exists^{2})$.
For the reversal, a continuous function is Darboux as in the previous paragraph, and attains is supremum by combining Cor.\ \ref{fliep} and the second-order results in \cite{simpson2}*{IV.2.3}.  
Hence, item \eqref{F15} expresses the existence of a discontinuous function, and $(\exists^{2})$ follows by \cite{kohlenbach2}*{\S3}.

\smallskip

For items \eqref{F16} and \eqref{F17}, a code for a Baire $n$ function is essentially an effectively Baire $n$ function where the continuous functions are given by codes. 
As noted below \cite{basket2}*{Def.\ 6.1}, $\ACA_{0}$ suffices to show that a code for a Baire $n$ function has a (unique) value.   
Hence, $\exists^{2}$ readily defines a third-order function taking these values everywhere on $[0,1]$.  
For the reversal, define a `Baire 1 code for the Heaviside function' in $\RCAo$ and use items \eqref{F16} or \eqref{F17} to obtain the (discontinuous) Heaviside function, yielding $(\exists^{2})$ by \cite{kohlenbach2}*{\S3}.  
\end{proof}
\noindent
The previous theorem yields the following strange result by contraposition: 
\begin{center}
\emph{if all functions on $\R$ are Baire 1, then all functions on $\R$ are continuous.}
\end{center}
In this light, Brouwer's theorem is not an isolated event, but rather the limit of a certain restriction process.   

\smallskip

Finally, one cannot generalise item \eqref{F5} of Theorem~\ref{flame} to Baire 2 by Theorem~\ref{trupp}.  
One \emph{can} generalise the latter item to `effectively Baire~2' by considering a well-known effectively Baire 3 function: the characteristic function of Borel's normal numbers (\cite{normaleborel}).  
The technical details are however tedious and the same holds for `effectively Baire $n$', where examples of such functions are given in \cite{keylesh}.  
%Moreover, $(\exists^{2})$ is equivalent to the existence of a functional that takes as input a continuous $f:[0,1]\di [0,1]$ and outputs a fixed point (\cite{kohlenbach2}*{Prop.\ 3.14}); one readily generalises this result to \emph{half-continuous} functions (\cite{halfcont}). 

\subsection{Beyond the Big Five}\label{XXX}
\subsubsection{Introduction}\label{lintro}
In the above, we have obtained equivalences between well-known second-order principles like the Big Five on one hand, and a number of 
third-order theorems on the other hand.  This was based on the higher-order RM of $(\exists^{2})$, which we have also developed.   
In this section, we show that similarly basic third-order statements or slight generalisations, go \textbf{far} beyond the RM of the Big Five and $(\exists^{2})$.  
We do so by deriving from the former the following:
\be\tag{$\NIN_{[0,1]}$}
(\forall Y:[0,1]\di \N)(\exists x, y\in [0,1])(x\ne y\wedge Y(x)=Y(y)),
\ee
which expresses that there is no injection from $[0,1]$ to $\N$. 
By \cite{dagsamIX}*{\S4}, $\NIN_{[0,1]}$ is not provable in relatively strong systems like $\Z_{2}^{\omega}$, which is a conservative extension of $\Z_{2}$  (see Section \ref{lll}).
The following list provides some interesting examples.
\begin{itemize}
\item The existence of a function not in Baire $1$ is equivalent to $(\exists^{2})$ (Theorem~\ref{flame}), while the existence of a function not in Baire 2 (or Baire $1^{*}$) implies $\NIN_{[0,1]}$ (Theorem \ref{trupp}).  
\item Cousin's lemma for lsco (or: quasi-continuous) functions is equivalent to $\WKL$ (Theorem \ref{lebber}), while this lemma for usco (or: cliquish)  functions implies $\NIN_{[0,1]}$ (Theorem \ref{reklam}).
\item Cousin's lemma for effectively (or: codes for) Baire 2 functions is part of the RM of $\ATR_{0}$ (Theorem \ref{floepje}) while the generalisation to Baire 2 (or Baire $1^{*}$) functions implies $\NIN_{[0,1]}$ (Theorem \ref{reklam}).
\item The supremum principle for effectively (or: codes for) Baire 2 functions is part of the RM of $\FIVE$ (Theorem \ref{flapke}) while the generalisation to Baire~2 (or Baire $1^{*}$) functions implies $\NIN_{[0,1]}$ (Theorem \ref{truppke}).
\item Jordan's decomposition theorem for cadlag $BV$-functions is equivalent to $\ACA_{0}$ (Theorem \ref{XZ}), while this theorem for usco $BV$-functions implies $\NIN_{[0,1]}$ (Theorem \ref{reklam2}). 
\end{itemize}
In our opinion, these examples show that one should not put too much emphasis on the distinction `second- versus third-order', as there are plenty equivalences between second- and third-order theorems. 
The real fundamental `divide' is whether a given theorem follows from conventional comprehension alone (say up to $\Z_{2}^{\omega}$), or whether it implies $\NIN_{[0,1]}$ or similar principles not provable in $\Z_{2}^{\omega}$.   

\smallskip

An important side-result of this section (see Theorem \ref{reklam}) is that many well-known inclusions among function spaces, like the statements \emph{$BV$-functions are Baire 1} and \emph{Baire 1$^{*}$ functions are Baire 1}, also imply $\NIN_{[0,1]}$.
In this way, such inclusions cannot be established in the Big Five and much stronger systems. 

\smallskip

Finally, we mention in passing that the results in this section also identify certain problems with the representation or coding of (slightly) discontinuous functions in the language of second-order arithmetic.

\subsubsection{Beyond Baire 1 functions}
In this section, we show that the equivalences in Theorems \ref{tank} and \ref{flame} cannot be generalised to Baire 2 or Baire 1$^{*}$.  

\smallskip

First of all, we shall need the following fragment of the induction axiom, also studied in \cite{dagsamXI}*{\S2.2.2} with some non-trivial equivalences.
\bdefi[$\INDY$]
Let $Y^{2}$ satisfy $(\forall n\in \N)(\exists \textup{ at most one } f\in 2^{\N})(Y(f, n)=0)$.  
For $k\in \N$, there is $w^{1^{*}}$ with $|w|=k$ such that for $m\leq k$, we have:
\[
(w(m)\in 2^{\N}\wedge Y(w(m), m)=0) \asa (\exists f\in 2^{\N})(Y(f, m)=0).
\]
\edefi
\noindent
We now have the following result, to be contrasted with item \eqref{F5} in Theorem \ref{flame}.  
There is no contradiction here:  Baire $1^{*}$ functions are of course Baire 1, but by Theorem~\ref{reklam}, this is not provable from the Big Five and much stronger systems.  

%Recall that all Baire $1^{*}$ are Baire 1, say in $\ZF$, but 
\begin{thm}[$\ACAo+\IND_{0}$]\label{trupp}
The principle $\NIN_{[0,1]}$ follows from either:
\begin{itemize}
\item {there is a $[0,1]\di \R$ function that is not Baire 2},
\item {there is a $[0,1]\di \R$ function that is not Baire 1$^{*}$}. 
\end{itemize}
\end{thm}
\begin{proof}
Fix an arbitrary function $f:[0,1]\di \R$ and let $Y:[0,1]\di \N$ be injective.  Now define $f_{n}(x)$ as $f(x)$ in case $Y(x)\leq n$, and $0$ otherwise. 
Clearly, $f$ is the pointwise limit of the sequence $(f_{n})_{n\in \N}$.  Now fix some $n_{0}\in \N$ and use $\IND_{0}$ to enumerate all $x\in [0,1]$ such that $Y(x)\leq n_{0}$.  
With this finite sequence, one readily defines a sequence of continuous functions converging to $f_{n_{0}}$, which shows that the latter is Baire 1.
This shows that any function $f:[0,1]\di \R$ is Baire 2 assuming $\neg \NIN_{[0,1]}$.  
Now define the closed set $C_{n}:=\{x\in [0,1]: Y(x)=n\}$ and note that $f_{\upharpoonright C_{n}}$ is indeed continuous for all $n\in \N$.
Hence, $f$ is also Baire $1^{*}$ assuming $\neg\NIN_{[0,1]}$, and we are done. 
\end{proof}
We believe the first item in Theorem \ref{trupp} is related to the Vitali-Carath\'eodory theorem as in \cite{wezinken}*{Cor.\ 4}, but we can only conjecture a connection. 

\smallskip

Secondly, we study the following supremum principle which Theorem \ref{flapke} establishes for effectively Baire $n$ functions assuming $\FIVE$.  
Theorem \ref{truppke} shows that slight generalisations are not provable in $\Z_{2}^{\omega}$.
\begin{princ}[Supremum principle for $\Gamma$]\label{tink}
For bounded $f:[0,1]\di \R$ in $\Gamma$, there is $F:\Q^{2}\di \R$ such that for $p, q\in \Q\cap[0,1]$, the real $F(p, q)$ equals $\inf_{x\in [p, q]}f(x)$.
\end{princ}
The following theorem is to be contrasted with items \eqref{W2}, \eqref{W43}, and \eqref{W8} of Theorem~\ref{tank} and with Theorem \ref{flapke}.
There is no contradiction here as $\NIN_{[0,1]}$ follows from the fact that regulated or usco functions are Baire 1 by Theorem \ref{reklam}.
\begin{thm}[$\ACAo+\IND_{0}$]\label{truppke}
The principle $\NIN_{[0,1]}$ follows from either:
\begin{itemize}
%\item A bounded Baire 2 function on $[0,1]$ has a supremum.
\item The supremum principle \(Princ.\ \ref{tink}\) for Baire 1$^{*}$ or Baire 2 functions. 
\item The supremum principle \(Princ.\ \ref{tink}\) for cliquish functions. 
\item The supremum principle \(Princ.\ \ref{tink}\) for regulated functions. 
\item The supremum principle \(Princ.\ \ref{tink}\) for usco functions. 
\item The supremum principle \(Princ.\ \ref{tink}\) for $BV$-functions. 
%\item a bounded cliquish function on $[0,1]$ has a supremum.? 
%\item a bounded regulated function on $[0,1]$ has a supremum. ?
%\item a bounded usco function on $[0,1]$ has a supremum. ?
%\item a bounded $BV$ function on $[0,1]$ has a supremum. ?
\end{itemize}
The theorem still goes through if we limit the items to functions that are pointwise discontinuous or continuous almost everywhere. 
\end{thm}
\begin{proof}
First of all, let $Y:[0,1]\di \N$ be an injection and define $f(x):= \frac{1}{2^{Y(x)+5}}$, which is Baire 2, Baire $1^{*}$, usco, regulated, cliquish, and $BV$, as we show next.  % as one readily shows using $\IND_{0}$.  
Indeed, for fixed $x_{0}\in [0,1]$, $\IND_{0}$ can enumerate the finitely many reals that are mapped below $n_{0}:=Y(x_{0})$ by $Y$.  
Thus, $f$ is arbitrary close to $0$ in a small enough punctured neighbourhood of $x_{0}$, readily implying that it is usco, regulated, and cliquish.  
Moreover, since $Y$ is an injection, the sums $\sum_{i=0}^{n-1} |f(x_{i})-f(x_{i+1})|$ as in \eqref{tomb} are bounded by $\sum_{i=0}^{n}\frac{1}{2^{i+5}}$, which is at most $2$, i.e.\ $f$ is in $BV$.
For the Baire 2 property, define $f_{n}(x)$ as $f(x)$ in case $Y(x)\leq n$, and define $f_{n}(x)$ as $0$ if $Y(x)>n$.
Clearly, for fixed $n_{0}\in \N$, the function $f_{n_{0}}$ has got at most finitely many points of discontinuity, which can be enumerated using $\IND_{0}$.  Using this finite list, one readily defines a sequence of continuous functions converging to $f_{n_{0}}$, i.e.\ the latter is Baire 1 and $f$ is Baire 2.   That $f$ is Baire $1^{*}$ follows as for Theorem \ref{trupp}.

\smallskip

Secondly, $\sup_{x\in [0,1]}f(x)$ is $\frac{1}{2^{n_{1}+5}}$ where $n_{1}\in \N$ is the least number such that $Y(x_{1})=n_{1}$ for some $x_{1}\in [0,1]$.  
Comparing $\sup_{x\in [0,\frac12]}f(x)$ and $\sup_{x\in [\frac12,1]}f(x)$, we obtain the first bit of the binary expansion of $x_{1}$.  
Using the usual interval-halving technique, we then obtain $x_{1}$ itself.  Now repeat this process for $f$ replaced by $f_{1}:[0,1]\di \R$ which is $f$ for $x\ne x_{1}$ and $0$ otherwise.  
Thus, we obtain an enumeration of $[0,1]$ and by \cite{simpson2}*{II.4.9}, there is $y\in [0,1]$ not in the latter sequence, a contradiction, and the items from the theorem follow. 

\smallskip

Thirdly, for the final sentence in the theorem, define $g:[0,1]\di \R$ as $g(x):=0$ if $x\in \Q\cap [0,1]$, and $g(x):=f(x)$ otherwise. 
Then $g$ is continuous at each rational number and thus pointwise discontinuous, which one readily establishes using $\IND_{0}$.  
In the same way as for $f$ in the previous paragraph, the function $g$ is $BV$, usco, regulated, Baire 2, and cliquish.  
Repeating the `sup construction' from the previous paragraph for $f$ replaced by $g$, one obtains an enumeration of $[0,1]\setminus \Q$, which yields the required contradiction. 

\smallskip

Finally, for the restriction to functions continuous almost everywhere, let $\mathcal{C}$ be the Cantor (middle-third) set, which has an RM-code and is recursively homomorphic to Cantor space, all in $\RCA_{0}$ by \cite{simpson2}*{I.8.6}.
Let $Y:\mathcal{C}\di\N$ be an injection and define $h:[0,1]\di \R$ as $0$ in case $x\not \in \mathcal{C}$ and $\frac{1}{2^{Y(x)+1}}$ otherwise.  
In the same way as for $f$ and $g$, the function $h$ is $BV$, usco, regulated, Baire 2, and cliquish.  Since $\mathcal{C}$ is closed and has measure zero, $h$ is continuous almost everywhere. 
Repeating the `sup construction' from the previous paragraph for $f$ replaced by $h$, one obtains an enumeration of $\mathcal{C}$, which yields a contradiction by \cite{simpson2}*{II.5.9}. 
Moreover, in case $\mathcal{C}$ is not countable, neither is $2^{\N}$ and $[0,1]$ by the results in \cite{samcie22}, and we are done.
\end{proof}
One can derive $\NIN$ from the fact that a Baire 2 function has a supremum, i.e.\ not involving parameters, but the technical details are somewhat messy.
Now, Theorem~\ref{truppke} identifies a problem with the coding of Baire 2 functions in the language of second-order arithmetic.  
Indeed, comparing Theorems~\ref{flapke} and~\ref{truppke}, we observe that the logical properties of the supremum principle for Baire~2 functions changes dramatically upon restriction to \emph{effectively} Baire 2 functions; the latter using codes for continuous functions is essentially the second-order representation used in \cites{basket, basket2}.

\smallskip

Thirdly, we have the following theorem to be contrasted with Theorem \ref{flappy}. 
Note that $\Z_{2}^{\omega}+\IND_{0}$ cannot prove $\NIN_{[0,1]}$ by \cite{dagsamX}*{\S3}. 
\begin{thm}[$\ACAo+\FIVE+\IND_{0}$]\label{truppke2}
The principle $\NIN_{[0,1]}$ follows from:
\be\label{pung}
\textup{\emph{for any open Baire 2 set in $\R$, there is an RM-code}.}
\ee
%\begin{center}
%\emph{for any open $O\subset \R$ and cliquish $\mathbb{1}_{O}$, there is an RM-code}. 
%\end{center}
\end{thm}
\begin{proof}
Fix $A\subset [0,1]$ and $Y:[0,1]\di \N$ injective on $A$.  
Note that we can use $\mu^{2}$ from Section \ref{lll} to remove any rationals from $A$.  
Note that we can also guarantee that $Y$ maps to $\N\setminus \{ 0, 1\}$.  
Now define the closed set $C\subset \R$ as follows:
\[
y\in C \asa (\exists n\in \N)( n<y<n+1 \wedge (y-n)\in A\wedge Y(y-n)=n ). 
\]
Since $Y$ is an injection on $A$, each interval $(n,n+1)$ for $n\geq 2$ contains at most one $y\in C$.  
Clearly, $C$ is closed and $O:=\R\setminus C$ is open.  For fixed $n_{0}\in \N$, we can enumerate the (at most $n_{0}$) reals in $C\cap (0, n_{0}+1)$ using $\IND_{0}$.  
With the finite list, one readily defines a sequence of continuous functions that converges to $\mathbb{1}_{O\cap (0, n_{0}+1)}$, i.e.\ the latter is Baire 1.  
Then $\mathbb{1}_{O}$ is Baire 2 and an RM-code for $O$ (and hence $C$) is provided by the centred item.  
Suppose $O=\cup_{n\in \N}(a_{n}, b_{n})$ where the latter intervals have rational end-points.  
Use $\FIVE$ to form the following set $B:= \{m\in \N: (\exists x\in \R)\varphi(x,m) \}$, where 
\[
\varphi(x,m) \equiv (m<x< m+1\wedge  x \in \R\setminus \cup_{n\in \N}(a_{n}, b_{n}))
\]
is (equivalent to) an $\L_{2}$-formula.  
Now apply $\Sigma_{1}^{1}\textsf{-AC}_{0}$, provable in $\ATR_{0}$ by \cite{simpson2}*{V.8.3}, to $(\forall m\in B)(\exists x\in \R)\varphi(x, m)$. 
The resulting sequence readily yields an enumeration of $A$.  
%However, $\exists^{2}$ can enumerate closed countable sets that have an RM-code by \cite{dagsamXIII}*{Lemma 4.11}.
By \cite{simpson2}*{II.4.9}, there is $y\in [0,1]$ not in this enumeration, and hence we can find $z\in [0,1]\setminus A$.
In this way, for any countable set $A\subseteq [0,1]$ (Def.~\ref{standard}), there is $y\in [0,1]\setminus A$.  
Thus, there is injection from $[0,1]$ to $\N$.
\end{proof}
One can replace $\R$ in \eqref{pung} by $[0,1]$ and still obtain $\NIN_{[0,1]}$; we leave this as an exercise to the reader. 

\smallskip

In conclusion, we have shown that certain \emph{slight} generalisations or variations of the third-order theorems from Sections \ref{TOT}-\ref{neweqi} go far beyond the RM of the Big Five or $(\exists^{2})$.  
It is however hard to draw a `borderline':  Theorem \ref{trupp} involves a super-class, namely Baire 2, \emph{and} a sub-class, namely Baire 1$^{*}$, of Baire 1.  
Now, Baire~1$^{*}$ functions {are} of course Baire 1, but only in strong enough systems by Theorem \ref{reklam}.  
A similar result can be obtained for Baire~1$^{**}$ functions (\cite{pawlak}). 

\subsubsection{Variations on a theme}\label{vate}
%In the above, we have obtained equivalences between the second-order Big Five on one hand, and a number of 
%third-order theorems on the other hand.  
In this section, we establish the results sketched in Section \ref{lintro}, i.e.\ we show that slight generalisations or variations of the third-order theorems from Sections \ref{TOT}-\ref{neweqi} are not provable from any Big Five system, the axiom $(\exists^{2})$, and much stronger systems.   

\smallskip

First of all, we have the following theorem, where the first five items are to be contrasted with Theorem \ref{lebber}.  
Regarding items \eqref{ta7}-\eqref{ta10}, we recall that $\Z_{2}^{\omega}$ cannot prove $\NIN_{[0,1]}$, i.e.\ $\ATR_{0}$ is relatively weak in comparison. 
\begin{thm}[$\ACAo+\IND_{0}$]\label{reklam}
The following statements imply $\NIN_{[0,1]}$.
\begin{enumerate}
\renewcommand{\theenumi}{\roman{enumi}}
\item Cousin's lemma for $BV$-functions $\Psi:[0,1]\di \R^{+}$. \label{ta0}
\item Cousin's lemma for regulated functions $\Psi:[0,1]\di \R^{+}$. \label{ta01}
\item Cousin's lemma for usco $\Psi:[0,1]\di \R^{+}$. \label{ta1}
\item Cousin's lemma for cliquish $\Psi:[0,1]\di \R^{+}$. \label{ta2}
\item Cousin's lemma for Baire $1^{*}$ $\Psi:[0,1]\di \R^{+}$. \label{ta22}
\item Cousin's lemma for Baire 2 $\Psi:[0,1]\di \R^{+}$. \label{ta21}
\item All usco \(or: lsco\) functions $f:[0,1]\di \R$ are Baire 1.\label{ta3}
%\item All lsco functions $f:[0,1]\di \R$ are Baire 1.\label{taX}????
%\item All cliquish $f:[0,1]\di \R$ are Baire 1.\label{ta4}
\item All $BV$-functions $f:[0,1]\di \R$ are Baire 1.\label{ta5}
\item All regulated functions $f:[0,1]\di \R$ are Baire 1.\label{ta51}
\item All Baire 1$^{*}$ functions $f:[0,1]\di \R$ are Baire 1.\label{tafel}
\item All usco cliquish $f:[0,1]\di \R$ are Baire 1.\label{ta6}
\end{enumerate}
Given $\ATR_{0}+\Delta_{2}^{1}\textsf{\textup{-IND}}$, the principle $\NIN_{[0,1]}$ also follows from the following.
\begin{enumerate}
\renewcommand{\theenumi}{\roman{enumi}}
\setcounter{enumi}{10}
\item All $BV$ \(or: regulated, usco, or lsco\) $f:[0,1]\di \R$ are effectively Baire~2.\label{ta7}
\item All $BV$ \(or: regulated, usco, or lsco\) $f:[0,1]\di \R$ are eff.\ Baire $n+2$.\label{ta71}
\item All $BV$ \(or: regulated, usco, or lsco\) $f:[0,1]\di \R$ have a Borel code.\label{ta8}
\item All Baire 2 $f:[0,1]\di \R$ are effectively Baire 2 \(Baire, \cite{beren2}*{p.\ 69}\).\label{ta9}
\item All Baire 2 $f:[0,1]\di \R$ are effectively Baire $n$ $(n\geq 3)$.\label{ta10}
\end{enumerate}
The theorem still goes through if we limit items \eqref{ta0}-\eqref{ta6} to functions that are pointwise discontinuous or continuous almost everywhere. 

\end{thm}
\begin{proof}
%The result for item \eqref{ta0} is proved in \cite{samcie22}*{Theorem 21}.  
Let $Y:[0,1]\di \N$ be an injection and define $\Psi(x):= \frac{1}{2^{Y(x)+5}}$.
Recall that $\Psi$ is usco (and cliquish, regulated, and $BV$) as shown in the proof of Theorem \ref{truppke}.
%note that for fixed $x_{0}\in [0,1]$, $\IND_{0}$ can enumerate the finitely many reals that are mapped by $\Psi$ to numbers larger than $\Psi(x_{0})$.  
Now, for distinct reals $x_{0}, \dots, x_{k}\in [0,1]$ we must have that $\cup_{i\leq k}B(x_{i}, \Psi(x_{i}))$ has measure at most $1/2$, since $\Psi(x_{0}), \dots, \Psi(x_{k})$ are all distinct due to $Y$ being an injection.  
However, item~\eqref{ta1} provides a finite sub-covering of $\cup_{x\in [0,1]}B(x, \Psi(x))$, which must have measure at least $1$, a contradiction.  Thus, there is no injection from $[0,1]$ to $\N$, i.e.\ $\NIN_{[0,1]}$ follows from item \eqref{ta1}.  The same proof goes through for items \eqref{ta0}, \eqref{ta01}, and \eqref{ta2}.  % in particular proving regulatedness and cliquishness using $\IND_{0}$. 
That $f$ is Baire $1^{*}$ follows as for Theorem \ref{trupp}, i.e.\ item~\eqref{ta22} also follows in the same way. 

\smallskip

For item \eqref{ta21}, one shows that $\Psi(x):= \frac{1}{2^{Y(x)+5}}$ as above, is Baire 2 in the same way as in the proof of Theorem \ref{trupp}.  
Indeed, define $\Psi_{n}(x)$ as $\Psi(x)$ if $Y(x)\leq n$, and $0$ otherwise.  Clearly, $\Psi$ is the pointwise limit of $\Psi_{n}$ while the latter has at most $n+5$ points of discontinuity.  
For fixed $n\in \N$, the discontinuity points of $\Psi_{n}$ can be enumerated using $\IND_{0}$, which readily yields a sequence of continuous functions that converges to $\Psi_{n}$.  
In this way, $\Psi$ as above is Baire 2, and $\NIN_{[0,1]}$ follows.  

\smallskip

For item \eqref{ta3}, we derive item \eqref{ta1} from the latter.  To this end, fix usco $\Psi:[0,1]\di \R^{+}$ and let $(\Psi_{n})_{n\in \N}$ be a sequence of continuous functions with pointwise limit $\Psi$.  
Use Corollary \ref{fliep} to obtain a sequence $(\Phi_{n})_{n\in \N}$ of codes for continuous functions.  The latter sequence is a code for a Baire 1 function in the sense of \cites{basket, basket2}.  
Since $\ACA_{0}$ proves Cousin's lemma for codes for Baire 1 functions (see \cites{basket, basket2}), we obtain Cousin's lemma for usco functions and hence $\NIN_{[0,1]}$ by item \eqref{ta3}.  
By definition, for usco $f:[0,1]\di \R$ with upper bound $N\in \N$ (provided by Theorem~\ref{tank}), $g(x):=N-f(x)$ is lsco, i.e.\ item \eqref{ta3} can be formulated with either lsco or usco.  
%The same proof establishes that item \eqref{ta4} implies item \eqref{ta2}.

\smallskip

For items \eqref{ta5}-\eqref{ta6}, %we note that Cousin's lemma for $BV$-functions implies $\NIN_{[0,1]}$, which follows from \cite{samcie22}*{Theorem 21}; 
the same proof as for item \eqref{ta3} goes through for regulated or $BV$-functions. 
In particular, $\Psi(x):= \frac{1}{2^{Y(x)+5}}$ is $BV$ (and regulated, cliquish, and Baire 1$^{*}$) in case $Y:[0,1]\di \N$ is an injection by the proof of Theorem \ref{truppke}.  
%The very same proof goes through for `$BV$' replaced by `cliquish'.

\smallskip

For item \eqref{ta7}, recall that $\ATR_{0}+\Delta_{2}^{1}$-induction proves (second-order) Cousin's lemma for codes of Baire 2 (or even Borel) functions (\cites{basket, basket2}).  
Moreover, we note that $\exists^{2}$ can convert an effectively Baire 2 function into a code for a Baire 2 function (by uniformly replacing the continuous functions by codes).
Hence, assuming item~\eqref{ta7}, we obtain the Cousin lemma for $BV$ (or: regulated, or: usco) functions, and hence $\NIN_{[0,1]}$ by the results for item \eqref{ta0} (and items \eqref{ta01} and \eqref{ta1}).   
The case for lsco functions implies the case for usco functions as for item \eqref{ta3}.
Items \eqref{ta71} and \eqref{ta8} follow in the same way. 
Items \eqref{ta9} and \eqref{ta10} follow in the same way, in combination with the fact that item \eqref{ta21} proves $\NIN_{[0,1]}$.

\smallskip 

For the final sentence, the supremum principle for Baire 1 functions is provable in $\RCAo+\WKL$ by Theorem \ref{tank}.
Hence, items \eqref{ta3}-\eqref{ta6} yield the supremum principle for the associated classes.  
We now obtain $\NIN_{[0,1]}$ via Theorem~\ref{truppke}.  For items \eqref{ta01}-\eqref{ta21}, let $Y:2^{\N}\di \N$ be an injection and recall the Cantor set $\mathcal{C}$ from the proof of Theorem \ref{truppke}, with associated recursive homomorphism $H:2^{\N}\di [0,1]$ defined as $H(f):=\sum_{n=0}^{\infty}\frac{2 f(n)}{3^{n+1}}$.  Now define $\Psi:[0,1]\di \R^{+}$ using $(\exists^{2})$ as:
\be\label{bydef}
\Psi(x):=
\begin{cases}
d(x, \mathcal{C})  & x\not\in \mathcal{C}\\
\frac{1}{2^{Y(I(x))+5}} & \textup{ otherwise } 
\end{cases},
\ee
where $I(x)$ is the unique $f\in 2^{\N}$ such that $H(f)=x$ in case $x\in \mathcal{C}$, and $00\dots$ otherwise.
As in the proof of Theorem \ref{truppke}, $\Psi$ is continuous almost everywhere, pointwise discontinuous, and has the properties required for items \eqref{ta01}-\eqref{ta21}.
To show that $\Psi$ is in $BV$, note that $\lambda x.d(x, \mathcal{C})$ is Lipschitz (with constant $1$) and hence $BV$ as the textbook proof goes through in $\RCAo$ (\cite{voordedorst}*{p.\ 74}).  Since $Y$ is an injection, the function $\lambda x.\big(\Psi(x)-d(x, \mathcal{C})\big)$ is also in $BV$. 
The sum of two $BV$-functions is in $BV$, as the textbook proof goes through in $\RCAo$ (\cite{voordedorst}*{Prop.\ 1.3}).

\smallskip

Finally, by the definition of $\Psi$ in \eqref{bydef}, if $x\in [0,1]\setminus \mathcal{C}$, then $\mathcal{C}\cap I_{x}^{\Psi}=\emptyset$.
Hence, let $z_{0}, \dots, z_{m}$ be those $y_{i}\in \mathcal{C}$ for $i\leq k$ and note that $\cup_{j\leq m}I_{z_{j}}^{\Psi}$ covers $\mathcal{C}$.
Then $I(z_{0}), \dots, I(z_{m})$ yields a finite sub-cover of $\cup_{f\in 2^{\N}}[\overline{f}(Y(f)+5)]$.  However, the measure of the latter is below $1/2$, a contradiction.  
Moreover, in case $2^{\N}$ is not countable, neither is $[0,1]$ by the results in \cite{samcie22}, i.e.\ $\NIN_{[0,1]}$ follows.
\end{proof}
We could restrict item \eqref{ta0} in Theorem \ref{reklam} to functions continuous on a set of positive measure; this would mean replacing the Cantor set $\mathcal{C}$ in \eqref{bydef} by a `fat' Cantor set, i.e.\ having positive measure.  Item \eqref{ta3}-\eqref{ta6} are also robust in that they still imply $\NIN_{[0,1]}$ upon replacing `Baire 1' by the equivalent definition (involving perfect sets) provided by the Baire characterisation theorem (\cite{beren}).

\smallskip

Recall the Vitali principle from Section \ref{ferengi}, which yields the following immediate corollary, to be contrasted with Theorem \ref{flebber}. 
\begin{cor}[$\ACAo+\IND_{0}$]\label{reklamcor}
The following statements imply $\NIN_{[0,1]}$.
\begin{enumerate}
\renewcommand{\theenumi}{\roman{enumi}}
\item Vitali's principle for $BV$-functions $\Psi:[0,1]\di \R^{+}$. \label{va0}
\item Vitali's principle for regulated functions $\Psi:[0,1]\di \R^{+}$. \label{va01}
\item Vitali's principle for usco $\Psi:[0,1]\di \R^{+}$. \label{va1}
\item Vitali's principle for cliquish $\Psi:[0,1]\di \R^{+}$. \label{va2}
\item Vitali's principle for Baire $1^{*}$ $\Psi:[0,1]\di \R^{+}$. \label{va22}
\item Vitali's principle for Baire 2 $\Psi:[0,1]\di \R^{+}$. \label{va21}
\end{enumerate}
\end{cor}
\begin{proof}
The proof of Theorem \ref{reklam} goes through: the single use of Cousin's lemma can be replaced by the associated version of Vitali's principle for $\eps>\frac{1}{2}$.
\end{proof}
We note that the Theorem \ref{reklam} identifies a significant problem with the coding of Baire 2 functions in the language of second-order arithmetic. 
Indeed, Cousin's lemma for \emph{codes for} Baire 2 functions is equivalent to $\ATR_{0}$ (\cite{basket, basket2}).  
In light of item \eqref{ta21} of Theorem~\ref{reklam}, this coding seriously changes the logical strength of Cousin's lemma for Baire 2 functions.  
We do have the following nice corollary to the theorem and Theorem \ref{floepje}, to be contrasted with the fact that $\Z_{2}^{\omega}$ cannot prove $\NIN_{[0,1]}$.
\begin{cor}
For $n\geq2$, the system $\ACAo+\ATR_{0}+\Delta_{2}^{1}\textup{\textsf{-IND}}$ proves that there is no effectively Baire $n$ function $Y:[0,1]\di \Q$ that is injective on $[0,1]$.
\end{cor}
\begin{proof}
By Theorem \ref{floepje}, the system at hand proves Cousin's lemma for effectively Baire $n$ functions. 
As in the (first paragraph of the) proof of the theorem, one obtains the relevant restriction of $\NIN_{[0,1]}$, and we are done. 
\end{proof}
%\end{proof}
Secondly, we have the following theorem, to be contrasted with Theorem \ref{XZ}. 
\begin{thm}[$\ACAo+\IND_{0}$]\label{reklam2}
The principle $\NIN_{[0,1]}$ follows from the following:
\begin{center}
\emph{Jordan decomposition theorem (Theorem \ref{drd}) restricted to usco $BV$-functions}.
\end{center}
\end{thm}
\begin{proof}
Let $Y:[0,1]\di \N$ be an injection and define $f(x):= \frac{1}{2^{Y(x)+5}}$.  
In the same way as in the proof of Theorem \ref{truppke}, $f:[0,1]\di \R$ is usco and $BV$.  
The centred statement from Theorem \ref{reklam2} now provides non-decreasing $g, h:[0,1]\di \R$ such that $f=g-h$ on $[0,1]$.
By \cite{dagsamXI}*{Lemma 3}, $\exists^{2}$ can enumerate the points of discontinuity of non-decreasing functions.  
As in \cite{simpson2}*{II.4.9}, this enumeration is not all of $[0,1]$, i.e.\ there is $y\in [0,1]$ such that $g$ and $h$ are continuous at $y$.  
Hence $f$ is continuous at $y$, but this is a contradiction as there are points $z$ arbitrarily close to $y$ such that $f(z)$ is arbitrarily small (use $\IND_{0}$ to establish this claim).  
\end{proof}
There are a number of similar `decomposition theorems', e.g.\ implying that cliquish and usco functions can be expressed as the sum of two quasi-continuous functions (\cites{quasibor2, malin}).  
One readily shows that the latter decompositions also yield $\NIN_{[0,1]}$, even when restricted to $BV$-functions.  

\smallskip

Thirdly, by Theorems \ref{XZ} and \ref{defzie}, $\RCAo+\ACA_{0}$ proves certain basic properties of the Riemann integral and $BV$-functions.
By contrast, Theorem \ref{drel} shows that other basic properties do not follow from the former and much stronger systems.
We note that the negation of the first item implies a very strong `non-uniqueness' of the Riemann integral. 
Similarly, the negation of the second item states that $BV$-functions need not be differentiable \emph{anywhere}.
A version of the second item, called Lebesgue's theorem, involving codes is provable in $\WKL_{0}$ by \cite{nieyo}*{\S6}.
\begin{thm}[$\ACAo+\IND_{0}$]\label{drel} The following principles imply $\NIN_{[0,1]}$.
\begin{itemize}
\item \emph{For a Riemann integrable $BV$-function $f:[0,1]\di [0,1]$ with $\int_{0}^{1}f(x)dx=0$, there is $x\in [0,1]$ such that $f(x)=0$.}
\item \emph{For $f:[0,1]\di [0,1]$ in $BV$, there is $x\in [0,1]$ where $f$ is differentiable.}
\item \emph{For regulated $f:[0,1]\di [0,1]$, there is $x\in [0,1]$ where $f$ is continuous.}
\item \emph{For usco $f:[0,1]\di \R$, there is $x\in [0,1]$ where $f$ is continuous.}
\end{itemize}
\end{thm}
\begin{proof}
Let $Y:[0,1]\di \N$ be an injection and define $f(x):= \frac{1}{2^{Y(x)+5}}$.  As in the proof of Theorems \ref{reklam} and \ref{reklam2}, this function is in $BV$ and regulated and usco. 

\smallskip

To show that $f$ is Riemann integrable with integral equal to zero, use the `epsilon-delta' definition of Riemann integrability and (essentially) the same argument why $f$ is in $BV$.
However, $f(x)>0$ for all $x\in [0,1]$ by definition, i.e.\ we obtain a contradiction from the first item of Theorem \ref{drel}. 

\smallskip

Similarly, the second and third item imply there is $y\in [0,1]$ where $f$ is continuous.  
This is a contradiction as there are points $z$ arbitrarily close to $y$ such that $f(z)$ is arbitrarily small (use $\IND_{0}$ to establish this claim).  
\end{proof}
It is an interesting exercise to show that in the final two items, \emph{continuity} can be replaced by much weaker properties, including \emph{feeble continuity} as in  \cite{thommy}*{\S24, p.\ 53}.
As noted in the latter, \emph{every} $\R\di \R$-function is feebly continuous at \emph{all but countably} many reals, i.e.\ we are dealing with a \emph{very} weak continuity notion. 
These results go back to Young (\cite{youngerandyounger}, 1907) with a (historical) overview in \cite{caulkingwood}.

\smallskip

We note that Theorems \ref{reklam2} and \ref{drel} identify a problem with the coding of functions in the language of second-order arithmetic.  
Indeed, we observe that the logical properties of Cousin's lemma for Baire~2 functions, the Jordan decomposition theorem for $BV$-functions, and Lebesgue's (differentiability) theorem for $BV$-functions change dramatically upon introducing second-order codes.  

\smallskip

Finally, while the main focus of \cite{dagsamX} is the study of $\NIN_{[0,1}$, some results are obtained for $\NBI_{[0,1]}$, i.e.\ the statement \emph{there is no \textbf{bijection} from $[0,1]$ to $\N$}. 
By \cite{dagsamX}*{\S4}, $\RCAo+\QFAC^{0,1}$ proves $\NBI_{[0,1]}$ but $\Z_{2}^{\omega}$ cannot.  In this way, $\NBI_{[0,1]}$ exhibits the Pincherle phenomenon from Remark \ref{PINX}. 
Now, we have used item~\eqref{donp} in Definition \ref{varvar} as our definition of $BV$-functions.  One could define `strong $BV$' as item \eqref{donp2} in Definition \ref{varvar}, i.e.\
the supremum \eqref{tomb} must additionally exist.  One readily verifies that e.g.\ Cousin's lemma for strong $BV$-functions (see item \eqref{ta0} in Theorem \ref{reklam}) implies $\NBI_{[0,1]}$.  
Similarly, we can derive $\NBI_{[0,1]}$ from any of the above theorems implying $\NIN_{[0,1]}$, even after replacing `$BV$' by `strong $BV$'.

\begin{ack}\rm
We thank Anil Nerode for his valuable advice.
We thank Ulrich Kohlenbach for (strongly) nudging us towards Theorem \ref{gofusefl}.
The main idea of this paper was conceived while reading \cite{damurm}, following discussions on higher-order RM with Carl Mummert.  
%We also thank the anonymous referee for the helpful suggestions.    
Our research was supported by the \emph{Deutsche Forschungsgemeinschaft} via the DFG grant SA3418/1-1 and the \emph{Klaus Tschira Boost Fund} via the grant Projekt KT43 .
\end{ack}

\appendix
\section{Higher-order Reverse Mathematics}\label{HORMreduction}
We introduce the base theory of higher-order RM (Section \ref{rmbt}), some essential notations (Section \ref{kkk}), and some axioms (Section \ref{lll}).
\subsection{The base theory of higher-order Reverse Mathematics}\label{rmbt}
We introduce Kohlenbach's base theory $\RCAo$, first introduced in \cite{kohlenbach2}*{\S2}.
\bdefi\label{kase} 
The base theory $\RCAo$ consists of the following axioms.
\begin{enumerate}
 \renewcommand{\theenumi}{\alph{enumi}}
\item  Basic axioms expressing that $0, 1, <_{0}, +_{0}, \times_{0}$ form an ordered semi-ring with equality $=_{0}$.
\item Basic axioms defining the well-known $\Pi$ and $\Sigma$ combinators (aka $K$ and $S$ in \cite{avi2}), which allow for the definition of \emph{$\lambda$-abstraction}. 
\item The defining axiom of the recursor constant $\mathbf{R}_{0}$: for $m^{0}$ and $f^{1}$: 
\be\label{special}
\mathbf{R}_{0}(f, m, 0):= m \textup{ and } \mathbf{R}_{0}(f, m, n+1):= f(n, \mathbf{R}_{0}(f, m, n)).
\ee
\item The \emph{axiom of extensionality}: for all $\rho, \tau\in \mathbf{T}$, we have:
\be\label{EXT}\tag{$\textsf{\textup{E}}_{\rho, \tau}$}  
(\forall  x^{\rho},y^{\rho}, \varphi^{\rho\di \tau}) \big[x=_{\rho} y \di \varphi(x)=_{\tau}\varphi(y)   \big].
\ee 
\item The induction axiom for quantifier-free formulas of $\L_{\omega}$.
\item $\QFAC^{1,0}$: the quantifier-free Axiom of Choice as in Definition \ref{QFAC}.
\end{enumerate}
\edefi
\noindent
Note that variables (of any finite type) are allowed in quantifier-free formulas of the language $\L_{\omega}$: only quantifiers are banned.
Recursion as in \eqref{special} is called \emph{primitive recursion}; the class of functionals obtained from $\mathbf{R}_{\rho}$ for all $\rho \in \mathbf{T}$ is called \emph{G\"odel's system $T$} of all (higher-order) primitive recursive functionals. 
%Moreover, we let $\INDD^{\omega}$ be the induction axiom for all formulas in $\L_{\omega}$.
\bdefi\label{QFAC} The axiom $\QFAC$ consists of the following for all $\sigma, \tau \in \textbf{T}$:
\be\tag{$\QFAC^{\sigma,\tau}$}
(\forall x^{\sigma})(\exists y^{\tau})A(x, y)\di (\exists Y^{\sigma\di \tau})(\forall x^{\sigma})A(x, Y(x)),
\ee
for any quantifier-free formula $A$ in the language of $\L_{\omega}$.
\edefi
As discussed in \cite{kohlenbach2}*{\S2}, $\RCAo$ and $\RCA_{0}$ prove the same sentences `up to language' as the latter is set-based and the former function-based.   
This conservation result is obtained via the so-called $\ECF$-interpretation discussed in Remark \ref{ECF}. 
\begin{rem}[The $\ECF$-interpretation]\label{ECF}\rm
The (rather) technical definition of $\ECF$ may be found in \cite{troelstra1}*{p.\ 138, \S2.6}.
Intuitively, the $\ECF$-interpretation $[A]_{\ECF}$ of a formula $A\in \L_{\omega}$ is just $A$ with all variables 
of type two and higher replaced by type one variables ranging over so-called `associates' or `RM-codes' (see \cite{kohlenbach4}*{\S4}); the latter are countable representations of continuous functionals.  
Thus, the formula $[A]_{\ECF}$ is just $A$ in case $A\in \L_{2}$.
%The definition of associate may be found just below Definition \ref{FTP}.
The $\ECF$-interpretation connects $\RCAo$ and $\RCA_{0}$ (see \cite{kohlenbach2}*{Prop.\ 3.1}) in that if $\RCAo$ proves $A$, then $\RCA_{0}$ proves $[A]_{\ECF}$, again `up to language', as $\RCA_{0}$ is 
formulated using sets, and $[A]_{\ECF}$ is formulated using types, i.e.\ using type zero and one objects.  
In light of the widespread use of codes in RM and the common practise of identifying codes with the objects being coded, it is no exaggeration to refer to $\ECF$ as the \emph{canonical} embedding of higher-order into second-order arithmetic. 
\end{rem}

\subsection{Notations and the like}\label{kkk}
We introduce the usual notations for common mathematical notions, like real numbers, as also introduced in \cite{kohlenbach2}.  
\begin{defi}[Real numbers and related notions in $\RCAo$]\label{keepintireal}\rm~
\begin{enumerate}
 \renewcommand{\theenumi}{\alph{enumi}}
\item Natural numbers correspond to type zero objects, and we use `$n^{0}$' and `$n\in \N$' interchangeably.  Rational numbers are defined as signed quotients of natural numbers, and `$q\in \Q$' and `$<_{\Q}$' have their usual meaning.    
\item Real numbers are coded by fast-converging Cauchy sequences $q_{(\cdot)}:\N\di \Q$, i.e.\  such that $(\forall n^{0}, i^{0})(|q_{n}-q_{n+i}|<_{\Q} \frac{1}{2^{n}})$.  
We use Kohlenbach's `hat function' from \cite{kohlenbach2}*{p.\ 289} to guarantee that every $q^{1}$ defines a real number.  
\item We write `$x\in \R$' to express that $x^{1}:=(q^{1}_{(\cdot)})$ represents a real as in the previous item and write $[x](k):=q_{k}$ for the $k$-th approximation of $x$.    
\item Two reals $x, y$ represented by $q_{(\cdot)}$ and $r_{(\cdot)}$ are \emph{equal}, denoted $x=_{\R}y$, if $(\forall n^{0})(|q_{n}-r_{n}|\leq {2^{-n+1}})$. Inequality `$<_{\R}$' is defined similarly.  
We sometimes omit the subscript `$\R$' if it is clear from context.           
\item Functions $F:\R\di \R$ are represented by $\Phi^{1\di 1}$ mapping equal reals to equal reals, i.e.\ extensionality as in $(\forall x , y\in \R)(x=_{\R}y\di \Phi(x)=_{\R}\Phi(y))$.\label{EXTEN}
\item The relation `$x\leq_{\tau}y$' is defined as in \eqref{aparth} but with `$\leq_{0}$' instead of `$=_{0}$'.  Binary sequences are denoted `$f^{1}, g^{1}\leq_{1}1$', but also `$f,g\in C$' or `$f, g\in 2^{\N}$'.  Elements of Baire space are given by $f^{1}, g^{1}$, but also denoted `$f, g\in \N^{\N}$'.
\item For a binary sequence $f^{1}$, the associated real in $[0,1]$ is $\r(f):=\sum_{n=0}^{\infty}\frac{f(n)}{2^{n+1}}$.\label{detrippe}
%\item An object $\textbf{Y}^{0\di \rho}$ is called \emph{a sequence of type $\rho$ objects} and also denoted $\textbf{Y}=(Y_{n})_{n\in \N}$ or $\textbf{Y}=\lambda n. Y_{n}$ where $Y_{n}:=\textbf{Y}(n)$ for all $n^{0}$.
\item Sets of type $\rho$ objects $X^{\rho\di 0}, Y^{\rho\di 0}, \dots$ are given by their characteristic functions $F^{\rho\di 0}_{X}\leq_{\rho\di 0}1$, i.e.\ we write `$x\in X$' for $ F_{X}(x)=_{0}1$. \label{koer} 
\end{enumerate}
\end{defi}
For completeness, we list the following notational convention for finite sequences.  
\begin{nota}[Finite sequences]\label{skim}\rm
The type for `finite sequences of objects of type $\rho$' is denoted $\rho^{*}$, which we shall only use for $\rho=0,1$.  
Since the usual coding of pairs of numbers goes through in $\RCAo$, we shall not always distinguish between $0$ and $0^{*}$. 
Similarly, we assume a fixed coding for finite sequences of type $1$ and shall make use of the type `$1^{*}$'.  
In general, we do not always distinguish between `$s^{\rho}$' and `$\langle s^{\rho}\rangle$', where the former is `the object $s$ of type $\rho$', and the latter is `the sequence of type $\rho^{*}$ with only element $s^{\rho}$'.  The empty sequence for the type $\rho^{*}$ is denoted by `$\langle \rangle_{\rho}$', usually with the typing omitted.  

\smallskip

Furthermore, we denote by `$|s|=n$' the length of the finite sequence $s^{\rho^{*}}=\langle s_{0}^{\rho},s_{1}^{\rho},\dots,s_{n-1}^{\rho}\rangle$, where $|\langle\rangle|=0$, i.e.\ the empty sequence has length zero.
For sequences $s^{\rho^{*}}, t^{\rho^{*}}$, we denote by `$s*t$' the concatenation of $s$ and $t$, i.e.\ $(s*t)(i)=s(i)$ for $i<|s|$ and $(s*t)(j)=t(|s|-j)$ for $|s|\leq j< |s|+|t|$. For a sequence $s^{\rho^{*}}$, we define $\overline{s}N:=\langle s(0), s(1), \dots,  s(N-1)\rangle $ for $N^{0}<|s|$.  
For a sequence $\alpha^{0\di \rho}$, we also write $\overline{\alpha}N=\langle \alpha(0), \alpha(1),\dots, \alpha(N-1)\rangle$ for \emph{any} $N^{0}$.  By way of shorthand, 
$(\forall q^{\rho}\in Q^{\rho^{*}})A(q)$ abbreviates $(\forall i^{0}<|Q|)A(Q(i))$, which is (equivalent to) quantifier-free if $A$ is.  For sequences $f^{1}, g^{1}$, the sequence $f\oplus g$ is $ f(0) * g(0) *f(1)* g(1)* \dots $.
\end{nota}

\subsection{Some comprehension functionals}\label{lll}
In second-order RM, the logical hardness of a theorem is measured via what fragment of the comprehension axiom is needed for a proof.  
For this reason, we introduce some axioms and functionals related to \emph{higher-order comprehension} in this section.
We are mostly dealing with \emph{conventional} comprehension here, i.e.\ only parameters over $\N$ and $\N^{\N}$ are allowed in formula classes like $\Pi_{k}^{1}$ and $\Sigma_{k}^{1}$.

\smallskip

First of all, the following functional is clearly discontinuous at $f=11\dots$; in fact, $(\exists^{2})$ is equivalent to the existence of $F:\R\di\R$ such that $F(x)=1$ if $x>_{\R}0$, and $0$ otherwise (\cite{kohlenbach2}*{\S3}).  This fact shall be repeated often.  
\be\label{muk}\tag{$\exists^{2}$}
(\exists \varphi^{2}\leq_{2}1)(\forall f^{1})\big[(\exists n)(f(n)=0) \asa \varphi(f)=0    \big]. 
\ee
Related to $(\exists^{2})$, the functional $\mu^{2}$ in $(\mu^{2})$ is also called \emph{Feferman's $\mu$} (\cite{avi2}) and can be found in Hilbert-Bernays' \emph{Grundlagen} (\cite{hillebilly2}*{Supplement V}).
\begin{align}\label{mu}\tag{$\mu^{2}$}
(\exists \mu^{2})(\forall f^{1})\big[ (\exists n)(f(n)=0) \di [f(\mu(f))=0&\wedge (\forall i<\mu(f))(f(i)\ne 0) ]\\
& \wedge [ (\forall n)(f(n)\ne0)\di   \mu(f)=0]    \big], \notag
\end{align}
We have $(\exists^{2})\asa (\mu^{2})$ over $\RCAo$ and $\ACAo\equiv\RCAo+(\exists^{2})$ proves the same sentences as $\ACA_{0}$ by \cite{hunterphd}*{Theorem~2.5}. 

\smallskip

Secondly, the functional $\SS^{2}$ in $(\SS^{2})$ is called \emph{the Suslin functional} (\cite{kohlenbach2}).
\be\tag{$\SS^{2}$}
(\exists\SS^{2}\leq_{2}1)(\forall f^{1})\big[  (\exists g^{1})(\forall n^{0})(f(\overline{g}n)=0)\asa \SS(f)=0  \big], 
\ee
The system $\FIVE^{\omega}\equiv \RCAo+(\SS^{2})$ proves the same $\Pi_{3}^{1}$-sentences as $\FIVE$ by \cite{yamayamaharehare}*{Theorem 2.2}.   
By definition, the Suslin functional $\SS^{2}$ can decide whether a $\Sigma_{1}^{1}$-formula as in the left-hand side of $(\SS^{2})$ is true or false.   We similarly define the functional $\SS_{k}^{2}$ which decides the truth or falsity of $\Sigma_{k}^{1}$-formulas from $\L_{2}$; we also define 
the system $\SIXK$ as $\RCAo+(\SS_{k}^{2})$, where  $(\SS_{k}^{2})$ expresses that $\SS_{k}^{2}$ exists.  
We note that the operators $\nu_{n}$ from \cite{boekskeopendoen}*{p.\ 129} are essentially $\SS_{n}^{2}$ strengthened to return a witness (if existant) to the $\Sigma_{n}^{1}$-formula at hand.  %  if it exists. 
The operator $\nu_{n}$ is essentially Hilbert-Bernays' operator $\nu$ (see \cite{hillebilly2}*{Supplement V}) restricted to $\Sigma_{n}^{1}$-formulas. 

\smallskip

\noindent
Thirdly, full second-order arithmetic $\Z_{2}$ is readily derived from $\cup_{k}\SIXK$, or from:
\be\tag{$\exists^{3}$}
(\exists E^{3}\leq_{3}1)(\forall Y^{2})\big[  (\exists f^{1})(Y(f)=0)\asa E(Y)=0  \big], 
\ee
and we therefore define $\Z_{2}^{\Omega}\equiv \RCAo+(\exists^{3})$ and $\Z_{2}^\omega\equiv \cup_{k}\SIXK$, which are conservative over $\Z_{2}$ by \cite{hunterphd}*{Cor.\ 2.6}. 
Despite this close connection, $\Z_{2}^{\omega}$ and $\Z_{2}^{\Omega}$ can behave quite differently, as discussed in e.g.\ \cite{dagsamIII}*{\S2.2}.   
The functional from $(\exists^{3})$ is also called `$\exists^{3}$', and we use the same convention for other functionals.  Hilbert-Bernays' operator $\nu$ (see \cite{hillebilly2}*{Supplement V}) is essentially Kleene's $\exists^{2}$, modulo a non-trivial fragment of the Axiom of (quantifier-free) Choice.

\bye